\documentclass[11pt]{article}
\usepackage{xcolor}
\usepackage{hyperref}
\usepackage{amsmath}
\usepackage{amssymb}
\usepackage{amsthm}
\usepackage{stmaryrd}
\usepackage{mhequ}
\usepackage{enumitem}

\usepackage[a4paper,innermargin=1.2in,outermargin=1.2in,bottom=1.5in,marginparwidth=1in,marginparsep=3mm]{geometry}

\definecolor{Brown}{rgb}{.75,.5,.25}
\definecolor{DGreen}{rgb}{0,0.55,0}

\newcommand{\mathd}{\mathrm{d}}
\newcommand{\dd}{\mathrm{d}}

\newtheorem{theorem}{Theorem}[section]
\newtheorem{ass}[theorem]{Assumption}
\newtheorem{prop}[theorem]{Proposition}
\newtheorem{lemma}[theorem]{Lemma}
\newtheorem{corollary}[theorem]{Corollary}

\theoremstyle{definition}
\newtheorem{remark}[theorem]{Remark}
\newtheorem{example}[theorem]{Example}
\newtheorem{definition}[theorem]{Definition}

\newcommand{\red}[1]{{\color{red} #1}}
\newcommand{\blue}[1]{{\color{blue} #1}}
\newcommand{\cF}{\mathcal{F}}
\newcommand{\cA}{\mathcal{A}}
\newcommand{\bF}{\mathbb{F}}
\newcommand{\cL}{\mathcal{L}}
\newcommand{\cP}{\mathcal{P}}
\newcommand{\cD}{\mathcal{D}}

\newcommand{\cS}{\mathcal{S}}
\newcommand{\cH}{\mathcal{H}}
\newcommand{\R}{\mathbb{R}}
\newcommand{\E}{\mathbb{E}}
\newcommand{\N}{\mathbb{N}}
\newcommand{\PP}{\mathbb{P}}
\newcommand{\eps}{\varepsilon}

\newcommand{\one}{\mathbf{1}}
\newcommand{\id}{{\mathrm{id}}}
\newcommand{\var}{{\mathrm{var}}}
\newcommand{\loc}{\mathrm{loc}}
\newcommand{\sign}{{\mathrm{sign}}}

\title{Solution theory of fractional SDEs\\ in complete subcritical regimes}
\author{Lucio Galeati and M\'at\'e Gerencs\'er}

\begin{document}
\maketitle
\begin{abstract}
We consider stochastic differential equations (SDEs) driven by a fractional Brownian motion with a drift coefficient that is allowed to be arbitrarily close to criticality in a scaling sense. We develop a comprehensive solution theory that includes strong existence, path-by-path uniqueness, existence of a solution flow of diffeomorphisms, Malliavin differentiability and $\rho$-irregularity.
As a consequence, we can also treat McKean-Vlasov, transport and continuity equations.
\end{abstract}
\tableofcontents

\section{Introduction}\label{sec:intro}
Given a vector field $b:\R_+\times\R^d\to\R^d$, an initial condition $x_0\in\R^d$, and a function $f:\R_+\to\R^d$, consider the differential equation
\begin{equation}\label{eq:ODE}
  X_t = x_0 + \int_0^t b_r ( X_r) \mathd r + f_t.
\end{equation}
When $f$ is chosen according to some random distribution, one obtains a stochastic differential equation (SDE), which often exhibits much better properties than the unperturbed equation ($f\equiv 0$), even at the level of existence and uniqueness of solutions.
This phenomenon is often referred to as \textit{regularisation by noise} and its study goes back to the works of Zvonkin \cite{zvonkin1974} and Veterennikov \cite{veretennikov1981}, see the monograph \cite{flandoli2011random} for a survey in the case of standard  Brownian $f$.


Although there is plenty of evidence \cite{Davie,CG16,HP,galeati2021noiseless} that it is the \emph{pathwise} properties of the perturbation that determine the regularisation effects, the available results are far more abundant in the Brownian, and in general, the Markovian case. 

However, a wide variety of applications motivate models with \emph{anomalous diffusions} with long-range memory, including statistical description of turbulence \cite{Kolmogorov}, hydrology \cite{Hurst}, anomalous polymer dynamics \cite{Panja}, diffusion in living cells \cite{Cell}, rough volatility models in finance \cite{Gatheral}.
Such non-Markovian processes are commonly modeled by fractional Brownian motion (fBm).
In this case the lack of Markovian and semimartingale structure renders a large part of a ``standard'' toolbox (It\^o's formula, Kolmogorov equations, Zvonkin transformation, martingale problem) unavailable.
Nevertheless, since fBm paths share many properties with the standard Brownian ones (up to changes in the scaling exponents), one would expect similar regularisation phenomena.

The goal of the present work is twofold.
First, we provide the first well-posedness results in the case of non-Markovian noise under demonstrably sharp conditions on $b$. The optimality follows both from a scaling heuristic (see Section \ref{sec:scaling-heuristic} below) and from rigorous construction of counterexamples (see Section \ref{sec:counterexample} below).
The second goal is to expand the existing well-posedness theory by studying various properties of the solutions that are well-known (though often nontrivial) in the Brownian case, but much less so for fractional noise.
These include existence, regularity, invertibility of the solution flow, stability with respect to perturbations of the initial condition and/or the nonlinearity, and Malliavin differentiability. The proofs can also be of interest in cases where the results are not new: the methods presented here go beyond not only the Markovian framework, but also the scope of Girsanov's theorem (see Remark \ref{rem:Girsanov} and Appendix \ref{app:girsanov}).

At the same time, the idea is quite intuitive: in order to develop a strong solution theory for \eqref{eq:ODE}, it is natural to investigate first the solvability of the \emph{linearised equation} around any given solution $X$, namely to show that
\begin{equation}\label{eq:ODE-linearised}
Y_t= y + \int_0^t \nabla b_r(X_r) Y_r\, \dd r 
\end{equation}
has a well-defined, unique solution for any $y\in \R^d$; observe that, due to its additive nature, the perturbation $f$ does not appear in \eqref{eq:ODE-linearised}. The study of \eqref{eq:ODE-linearised} is perfectly in line with the classical setting of a continuously differentiable drift $b$, where \eqref{eq:ODE-linearised} can be solved directly and its behaviour matches the Gr\"onwall-type estimates encountered when looking at the difference of any two solutions.
However if $b$ is not assumed to be differentiable, $\nabla b_r(X_r)$ a priori doesn't make sense and thus a standard interpretation for \eqref{eq:ODE-linearised} is no longer possible. The key idea in order to overcome this difficulty is two-fold:
\begin{itemize}
\item[a)] $\nabla b(\cdot)$ in \eqref{eq:ODE-linearised} is not evaluated at arbitrary space points, rather along the solution $X$, which can have very special properties inherited from the noise $f$.
\item[b)] In order to give meaning to \eqref{eq:ODE-linearised} in a \textit{Young integral sense}, we don't need to define $\nabla b_r(X_r)$ pointwise, instead it suffices to show that the path
\begin{equation}\label{eq:intro-additive}
t\mapsto L_t:=\int_0^t \nabla b_r(X_r) \dd r
\end{equation}
is well-defined and enjoys sufficiently nice time regularity (more precisely, it is of finite $p$-variation for some $p<2$).  In view of a), depending on the structure of the noise $f$, this can be a much more reasonable requirement.
\end{itemize}
In analogy with the Lipschitz setting, one can then transfer estimates for classical linear Young equations of the form
\begin{equ}\label{eq:introbound}
\sup_{t\in [0,1]} |Y_t|\lesssim e^{C\|L\|_{p-\var}^p} |y|
\end{equ}
to \emph{pathwise bounds} for the difference of any two solutions $X$ and $\tilde X$ with different initial conditions, up to replacing $L$ by another process $\hat L=\hat L(X,\tilde X)$ similar in spirit to \eqref{eq:intro-additive}.

In order to rigorously formalise all of the above, it is crucial to identify the correct space of perturbations $\varphi$ such that $X=\varphi+f$ indeed inherits the relevant properties from $f$; these are the a priori estimates given by Lemmas \ref{lem:drift-regularity}-\ref{lem:apriori-estim}. Correspondingly, we formulate two new versions of the 
Stochastic Sewing Lemma (SSL) by L\^e \cite{Khoa}, cf. Lemmas \ref{lem:SSL1} and \ref{lem:SSL2} below, which are tailor-made for our analysis.
Once this setup is in place, it provides exponential moment estimates of certain additive functionals of $X$, like the one defined in \eqref{eq:intro-additive}, turning pathwise bounds like \eqref{eq:introbound} into moment bounds. Finally, once the behaviour of the linearised equation \eqref{eq:ODE-linearised} is understood, many further properties (uniqueness, stability, differentiability of the flow) of the ODE follow similarly. 

\subsection{Scaling heuristics and existing literature}\label{sec:scaling-heuristic}
One way to have a unified view on the many works on regularization by noise is by a scaling argument; for a similar approach in the Brownian setting and $L^q_t L^p_x$ spaces, see \cite[Section 1.5]{beck2019stochastic}.
 
From now on we sample the perturbation as a fBm $B^H$ with Hurst parameter $H\in(0,+\infty) \setminus \mathbb{N}$, which satisfies the scaling relation
\begin{equ}\label{eq:fBM scaling}
(B^H_t)_{t\geq0}\overset{\mathrm{law}}{=}(\lambda^{-H}B^H_{\lambda t})_{t\geq 0}, \quad \forall\, \lambda>0.
\end{equ}
Details about the processes $B^H$ are given in Section \ref{sec:prelim-fbm} below; let us just briefly recall that $H=1/2$ gives the standard Brownian motion, that this is the only case where $B^H$ is a Markov process.
For the values $H=k+1/2$, $k\in\N_+$ (which we call ``degenerate Brownian'') 
the Markovian toolbox is still available, since the SDE can be rewritten as a higher dimensional equation driven by degenerate Brownian noise, see e.g. \cite{CM-weak}.
For all other choices of $H$ such tools
are unavailable and the study of the SDE requires a fundamentally different approach.
The equation then takes the form
\begin{equation}\label{eq:SDE}
  X_t = x_0 + \int_0^t b_r ( X_r) \mathd r + B^H_t.
\end{equation}
In order for the regularising effects of $B^H$ to dominate the irregularities of $b$, it is natural to require that, when zooming into small scales in a way that keeps the noise strength constant, the nonlinearity vanishes; if this weren't the case, and the nonlinearity were dominant, we would expect to see all the same pathologies (e.g. coalescence or branching of solutions) which could manifest in the ODE without noise.
Therefore, keeping \eqref{eq:fBM scaling} in mind, for a fixed parameter $H$ we call a space $V$ of functions (or distributions) on $\R_+\times\R^d$ \emph{critical} (resp. \emph{subcritical}/\emph{supercritical}) if for the rescaled drift coefficient
\begin{equ}
b^\lambda_t(x)=\lambda^{1-H} b(\lambda t, \lambda^H x),
\end{equ}
the \emph{leading order seminorm} $\llbracket b^\lambda\rrbracket_V$ (see the examples below for its practical meaning) scales like $\lambda^\gamma \llbracket b\rrbracket_V$, for all $\lambda\leq 1$,\footnote{
Specifically, we are interested in understanding how $\llbracket b^\lambda\rrbracket_V$ scales as $\lambda\to 0$, which is related to studying the local behaviour of solutions; instead the scaling of $\llbracket b^\lambda\rrbracket_V$ as $\lambda\to\infty$ reflects a ``zoom out'' which identifies the dominant term concerning the long-time dynamics.} with $\gamma=0$ (resp. $\gamma>0$/$\gamma<0$).

We refer to Section \ref{sec:notation} for more details on the function spaces appearing in the upcoming examples.

\begin{example}\label{ex:Calpha}
Consider autonomous, inhomogeneous H\"older-Besov spaces $V=B^\alpha_{\infty,\infty}$, where $b$ does not depend on the time variable.
Here the leading order seminorm is the associated homogeneous seminorm, namely we set $\llbracket f\rrbracket_V:=\| f\|_{\dot{B}^\alpha_{\infty,\infty}}$ as defined in \cite{BCD}; alternatively, for $f\in B^\alpha_{\infty,\infty}$ and $\alpha\geq 0$, one can regard it as $\| (-\Delta)^{\alpha/2} f\|_{B^0_{\infty,\infty}}$ , while for $\alpha<0$ one can define it by duality with the homogeneous seminorm of $\dot{B}^{-\alpha}_{1,1}$. Either way, one finds the scaling relation
\begin{equation*}
\| f(\eta\, \cdot)\|_{\dot{B}^\alpha_{\infty,\infty}} \sim_\alpha \eta^\alpha \| f(\cdot)\|_{\dot{B}^\alpha_{\infty,\infty}}  \quad\, \forall\, (\eta,\alpha)\in \R_{>0}\times \R.
\end{equation*}
Combined with our definition of $b^\lambda$, one finds $\gamma=1-H+\alpha H$ and so the subcriticality condition reads as
\begin{equ}\label{444}
\alpha>1-\frac{1}{H}.
\end{equ}
However, even in the classical Brownian case, where one gets the condition $\alpha>-1$, this remains out of reach.
Weak well-posedness is known for $\alpha>-1/2$ \cite{FIR}, and a nonstandard kind of well-posedness (where uniqueness is even weaker than uniqueness in law) is shown for $\alpha>-2/3$ \cite{Delarue, CC}, for special classes of drift $b$.
The classical works \cite{zvonkin1974, veretennikov1981} show strong well-posedness for $V=C^\alpha_x$ and $\alpha\geq 0$.\footnote{Please see our convention on the definition of $C^\alpha_x$ from Section \ref{sec:notation} below, especially for $\alpha\in\N$; in particular, $C^0_x$ is understood as the space bounded and measurable functions, with $L^\infty$-norm.}
Interestingly, in the degenerate Brownian case weak well-posedness is proved in \cite{CM-weak} in the full regime $\alpha>(2k-1)/(2k+1)$, which is precisely the condition \eqref{444}.
For strong well-posedness one requires the more restrictive condition
\begin{equ}
\alpha>1-\frac{1}{2H},
\end{equ}
see \cite[Equation (1.11)]{CHM}.
The same condition is required for strong well-posedness in the non-Markovian case for all $H\in(0,\infty)\setminus\N$, cf. \cite{NO1, CG16, galeati2021noiseless, gerencser2020regularisation}. After the first version of this manuscript, the work \cite{butkovsky2023} appeared, where the authors are able to establish (among several results) weak existence of solutions in the full subcritical regime \eqref{444}, under the additional assumption that $b$ is a Radon measure; however, uniqueness is still open.
\end{example}

\begin{example}\label{ex:LqLp}
Another well-studied case is the mixed Lebesgue space $V=L^q_tL^p_x$. Here we can take the seminorm to be $\| \cdot\|_V$ itself; using the scaling relation $\|f(\eta\,\cdot)\|_{L^p_x}= \eta^{-d/p}$, one finds $\gamma=1-H-1/q-(Hd)/p$ and the subcritical regime is
\begin{equ}\label{eq:LqLp-exponents}
\frac{1}{q}+\frac{Hd}{p}<1-H.
\end{equ}
In the classical case $H=1/2$, equation \eqref{eq:LqLp-exponents} reads as
\begin{equation}\label{eq:LPS}
\frac{2}{q}+\frac{d}{p}<1
\end{equation}
which is precisely the condition from the classical work \cite{KR}, where strong well-posedness is proved (under the additional constrant $p\geq 2$); instead the critical regime corresponds to the celebrated Ladyzhenskaya--Prodi--Serrin (LPS) condition.
This case has then been extensively studied by several authors, allowing also for multiplicative noise with Sobolev diffusion coefficients, see among others \cite{zhang2010stochastic, fedrizzi2013holder,zhang2016,xia2020}.
In recent years, even the critical case has been reached \cite{Krylov-Critical, RZ-Critical} under certain constraints on $d,p,q$; the results have been further refined by allowing coefficients in Morrey spaces, cf. \cite{krylov2022strong,Krylov2025}, or form-bounded drifts, cf. \cite{KinSem2023,kinzebulatov2023strong} and the references therein. 
It was recently understood in \cite{zhang2021stochastic} that one can go beyond condition \eqref{eq:LqLp-exponents}, up to imposing additional constraints on ${\rm div}\, b$; for further progress in this exciting direction, see also \cite{hao2023sdes,grafner2024weak}.

For $H\in(1/2,1)$ no results are known and for $H\in(0, 1/2)$ the main previously known results for weak and strong well-posedness are both from \cite{Khoa}, under the stronger conditions
\begin{equ}\label{eq:LqLp suboptimal}
\frac{1}{q}+\frac{Hd}{p}<\frac{1}{2},\qquad\frac{1}{q}+\frac{Hd}{p}<\frac{1}{2}-H,
\end{equ}
respectively, with the additional constraint $p\in[2,\infty]$, later removed in \cite{galeati2021noiseless}.
It is conjectured  in \cite{Khoa} that the first condition in \eqref{eq:LqLp suboptimal} is enough to guarantee strong well-posedness. One particular corollary of our result is that for $q\in(1,2]$ even \eqref{eq:LqLp-exponents} is sufficient.
Therefore we propose to update the conjecture of \cite{Khoa} (if $q\in(1,2]$, now a theorem) to assert strong well-posedness under the scaling condition \eqref{eq:LqLp-exponents}. Let us also mention that we have recently learned about an ongoing work \cite{Toyomu} towards improving \eqref{eq:LqLp suboptimal}.
\end{example}

\begin{example}\label{ex:LqCa}
A common generalisation of Examples \ref{ex:Calpha} and \ref{ex:LqLp} is the space $V=L^q_t C^\alpha_x$, where (adopting the leading seminorm to be the one of $L^q_t \dot{B}^\alpha_{\infty,\infty}$, in agreement with both previous cases,\footnote{By Besov embeddings $L^p_x \hookrightarrow B^{-d/p}_{\infty,\infty}$ with homogeneous norms behaving in the same way under rescaling.}) the scaling works out to be $\gamma=1-H-1/q+\alpha H$.
Therefore the subcriticality condition reads as
\begin{equ}
\alpha>1-\frac{1}{H}+\frac{1}{Hq}=1-\frac{1}{q'H},
\end{equ}
where, here and in the rest of the paper, $q$ and $q'$ are conjugate exponents, $1/q+1/q'=1$.
This generality has only been studied recently in \cite{galeati2021noiseless, galeati2021distribution}, where strong well-posedness is proved under the stronger condition
\begin{equ}\label{eq:LqCa suboptimal}
\alpha>1-\frac{1}{2H}+\frac{1}{Hq},
\end{equ}
with the additional constraints $H\in(0,1/2]$, $q\in(2,\infty]$. Note that, by setting $\alpha=-d/p$, condition \eqref{eq:LqCa suboptimal} coincides with the second one in \eqref{eq:LqLp suboptimal}.
\end{example}

In summary, to the best of our knowledge, weak well-posedness results in a whole subcritical regime are available only in the degenerate Brownian case $H=k+1/2$, $k\in\N$, and strong well-posedness only in the standard Brownian case $H=1/2$. 

\subsection{Discussion of the main results}\label{sec:main-results}
In the present paper we establish strong well-posedness in the full subcritical regime for all $H\in(0,\infty)\setminus \N$, with coefficients from the class in Example \ref{ex:LqCa}, under the additional constraint $q\in(1,2]$. In other terms, our main conditions are summarised by the assumption
\begin{equ}\label{eq:main exponent}
H\in(0,\infty) \setminus \mathbb{N},\qquad q\in(1,2],\qquad
\alpha\in\Big(1-\frac{1}{q'H},1\Big).\tag{A}
\end{equ}
The solution theory we present in fact goes beyond strong well-posedness.
We show existence in the strong sense not only of solutions but also of solution flows, and uniqueness in the path-by-path sense.
Furthermore, several further properties of solutions are established such as stability, continuous differentiability of the flow and its inverse, Malliavin differentiability, and $\rho$-irregularity.

Many of these results are even new in the time-independent case: if $b$ is only a function of $x$ and belongs to $C^\alpha_x$, then the optimal choice to put it in the framework of \eqref{eq:main exponent} is to choose $q=2$, leading to the condition $\alpha>1-1/(2H)$. This is the classical condition under which strong well-posedness is known \cite{NO1, CG16, gerencser2020regularisation}, but several of the further properties have not been previously established.

Our main findings are loosely summarised (without aiming for full precision or generality) in the following statement; the corresponding results (often in a somewhat sharper form) can be found throughout the paper in Theorems \ref{thm:functional-existence}, \ref{thm:functional:PBP-uniqueness}, \ref{thm:dist-existence}, \ref{thm:distributional:PBP-uniqueness} for \emph{i}), \ref{thm:stability} for \emph{ii}), \ref{thm:flow-diffeo} for \emph{iii}), \ref{thm:malliavin}  for \emph{iv}), \ref{thm:MKV} for \emph{v}), \ref{thm:rho-irr} for \emph{vi}), \ref{thm:transport-equation} for \emph{vii}).
For simplicity, we restrict ourselves to the time interval $t\in [0,1]$, but it's clear that up to rescaling we could consider any finite $[0,T]$ (up to allowing the hidden constants to depend on $T$).

\begin{theorem}\label{thm:BIG}
Assume \eqref{eq:main exponent} and let $x_0\in\R^d$, $b\in L^q_t C^\alpha_x$, $m\in[1,\infty)$. Then:
\begin{enumerate}[label=\roman*)]
\item Strong existence and path-by-path uniqueness holds for \eqref{eq:SDE};
\item For any other $\tilde x_0\in\R^d$ and $\tilde b\in L^q_t C^\alpha_x$, the associated solutions $X$ and $\tilde X$ satisfy the stability estimate
\begin{equ}
\E\bigg[\sup_{t\in[0,1]}|X_t-\tilde X_t|^m\bigg]^{1/m}\lesssim |x_0-\tilde x_0|+\|b-
\tilde b\|_{L^q_t C^{\alpha-1}_x};
\end{equ}
\item The solutions form a stochastic flow of diffeomorphisms $\Phi_{s\to t}(x)$, whose spatial gradient $\nabla \Phi$ is $\PP$-a.s. continuous in all variables; moreover it holds
\begin{equation*}
\sup_{0\leq s\leq t\leq 1, x\in \R^d} \E\big[| \nabla \Phi_{s\to t} (x) |^m\big] <\infty;
\end{equation*}
\item For each $s<t$ and $x\in \R^d$, the random variable $\omega\mapsto \Phi_{s\to t}(x;\omega)$ is Malliavin differentiable; moreover it holds
\begin{equation*}
\sup_{0\leq s\leq t\leq 1, x\in \R^d} \E\big[ \| D \Phi_{s\to t}(x)\|_{\mathcal{H}^H}^m \big]<\infty,
\end{equation*}
where $D$ is the Malliavin derivative and $\cH^H$ the Cameron-Martin space of $B^H$;
\item Strong existence and uniqueness holds also for the McKean-Vlasov equation
\begin{equation}\label{eq:MKV}
X_t=x_0+\int_0^t(b_r\ast\mu_r)(X_r)\dd r +B^H_t,\qquad\mu_t=\cL(X_t);
\end{equation}
\item Solutions $X$ are $\PP$-a.s. $\rho$-irregular for any $\rho<1/(2H)$;
\item If additionally $\alpha>0$, then for any $p>1$ strong existence and path-by-path uniqueness holds for solutions $u\in L^\infty_t W^{1,p}_x$ to the transport equation
\begin{equ}
\partial_t u + b\cdot \nabla u + \dot{B}^H_t\cdot \nabla u=0
\end{equ}
for all initial data $u_0\in W^{1,p}_x$.
\end{enumerate}
\end{theorem}


The various aspects of the main results are discussed in detail in their respective sections, so here let us just briefly comment on them.

The notion of path-by-path uniqueness in \emph{i}), as a strengthening of the classical pathwise uniqueness, was first established in the seminal work \cite{Davie} by Davie, with a simpler proof was later provided by Shaposhnikov \cite{shaposhnikov2016some}.
This kind of result was then generalised to fBm in \cite{CG16}, suggesting it is a consequence of the pathwise properties of the trajectories of the driving noise.
Such a uniqueness concept requires to give a pathwise interpretation to the SDE, which becomes nontrivial for $\alpha<0$, where $b$ can be a distribution of negative regularity and not a function anymore. In this case, following \cite{CG16}, we will give meaning to \eqref{eq:SDE} as a nonlinear Young ODE, see Section \ref{sec:distributional} for more details.

Stability estimates in the style of \emph{ii}) are useful to bypass abstract Yamada-Watanabe arguments and get strong existence directly. Among other possible applications, let us mention their importance in numerical schemes with distributional drifts, see e.g. the recent work \cite{goudenege2023numerical}.
In this paper, stability estimates play a key role when solving McKean-Vlasov equations as in \emph{v}), see Section \ref{sec:MKV}.
The study of stochastic flows \emph{iii}) for SDEs goes back to the classical work \cite{kunita1984stochastic}, see also \cite{fedrizzi2013holder, X-M-Flow, Menoukeu2013} for flows in irregular settings.
In \emph{iv}), we can in fact derive differentiability with respect to perturbations of the noise quite a bit more generally than in Cameron-Martin directions (see Remark \ref{rem:Malliavin}), in line with the observations from \cite{kusuoka1993regularity, friz2006}.
Concerning \emph{v}), regularisation by fractional noise for distribution dependent SDEs has been investigated in \cite{galeati2021distribution} and recently in \cite{han2022}.
Above we only stated the simplest example of McKean-Vlasov equation for the sake of presentation, Theorem \ref{thm:MKV} below allows for more general dependence on $(X,\mu)$.
The notion of $\rho$-irregularity in \emph{vi}) was introduced by \cite{CG16} as a powerful measurement of the averaging properties of paths.
Extending $\rho$-irregularity from Gaussian processes to perturbed Gaussian processes has previously only been achieved efficiently via Girsanov transform, here we provide a simple and more robust alternative. Concerning \emph{vii}), regularisation by noise results for the transport equation were first established for Brownian noise in \cite{flandoli2010well} and further developed in \cite{MoNiPr2015,fedrizzi2013noise}, see also \cite{catellier2016rough,nilssen2020,galeati2021noiseless} for further investigations in the fractional case.

The scope of some intermediate estimates we obtain is larger than \eqref{eq:main exponent}, and therefore in some regime where we do not obtain strong well-posedness, we still get compactness and therefore existence of weak solutions.
To state the result, we need to enforce the following different condition:

\begin{equation}\label{eq:condition-existence}
H\in(0,1)
,\quad q\in(1,\infty],\quad \alpha>\frac{1}{2}-\frac{1}{2H}, \quad \alpha>1-\frac{1}{H q'}.\tag{B}
\end{equation}

The proof of Theorem \ref{thm:weak-existence-intro} is presented in Section \ref{sec:weak-compactness}, where we also define rigorously what we mean by weak solution to \eqref{eq:SDE} in this case; see Theorem \ref{thm:weak-existence} for a more precise statement.

\begin{theorem}\label{thm:weak-existence-intro}
Assume \eqref{eq:condition-existence} and let $x_0\in\R^d$, $b\in L^q_t C^\alpha_x$; then there exists a weak solution to the SDE \eqref{eq:SDE}.
\end{theorem}

\begin{remark}\label{rem:other-condition-intro}
For $b\in L^q_t C^\alpha_x$ with $\alpha>0$, existence of weak solutions can be shown classically by standard Peano-type arguments, for any choice of $H\in (0,\infty)\setminus\mathbb{N}$. Therefore condition \eqref{eq:condition-existence} is of real interest only when considering $\alpha<0$; in this case $H\in (0,1)$ is not a real restriction, as it follows from the first condition on $\alpha$. Note further that in the case $q\in (1,2]$, it always holds $1-\frac{1}{Hq'}\geq \frac{1}{2}-\frac{1}{2H}$ and so \eqref{eq:condition-existence} reduces to \eqref{eq:main exponent}; thus the interesting cases covered by Theorem \ref{thm:weak-existence-intro} are for $q\in (2,\infty]$.
\end{remark}

\begin{remark}
For $q=\infty$, condition \eqref{eq:condition-existence} reduced to $b\in L^\infty_t C^\alpha_x$, $\alpha>\frac{1}{2}-\frac{1}{2H}$. In the Brownian case $H=1/2$, this recovers the condition $\alpha>-1/2$ obtained in \cite{FIR}, which showed uniqueness in law. Recently, \cite[Theorem 6.7]{kremp2023rough} provided counterexamples to uniqueness in law for Brownian SDEs with drifts $b\in C_t C^\alpha_x$, for any $\alpha\leq -1/2$; non-uniqueness here is meant in the class of ``canonical weak solutions'', i.e. satisfying a definition à la Bass-Chen \cite{Bass-Chen} (cf. Definition \ref{defn:weak-solution} below).
So there can be a non trivial gap between well-posedness results and the prediction offered by scaling arguments.
On the positive side, recently\cite{butkovsky2024weak} proved uniqueness in law of the solutions constructed by Theorem \ref{thm:weak-existence-intro}, at least in the case $H\in (0,1/2]$ and autonomous drift $b\in C^\alpha_x$ with $\alpha>\frac{1}{2}-\frac{1}{2H}$.
\end{remark}

\begin{remark}\label{rem:Girsanov}
One fundamental stochastic analytic tool that still applies in the non-Markovian fBm setting is Girsanov's transform. Indeed, it is heavily used in the seminal works \cite{NO1,CG16} and many subsequent ones.
However, it has its limitations: in our setting it only applies under the additional assumption $1-1/(q'H)<0$ (which in turn may only happen if $H\in(0,1/2)$), see Appendix \ref{app:girsanov} for details.
Even in the Brownian case $H=1/2$ our methods yield results beyond the scope of Girsanov's theorem, which is not available for $q<2$, see Remark \ref{rem:Brownian} below.
Therefore, throughout the article we avoid Girsanov's transform altogether.

Another motivation for a Girsanov-free approach is to develop tools that are robust enough to extend to other classes of process; see \cite{BDG-Levy} for some first results on such equations via stochastic sewing for L\'evy-driven SDEs and Remarks \ref{rem:possible-extensions}-\ref{rem:other-fractional-processes} below for other classes of Gaussian processes which fit our framework.
\end{remark}

\begin{remark}\label{rem:Brownian}
Theorem \ref{thm:BIG} gives new results also in the classical $H=1/2$ case.
Indeed, to solve \eqref{eq:SDE} with classical tools, one would require a good solution theory of the corresponding Kolmogorov equation 
\begin{equ}\label{eq:PDE}
\partial_t u-\tfrac{1}{2}\Delta u=b\cdot\nabla u.
\end{equ}
Suppose that $b\in L^q_t C^\alpha_x$ with $q\in(1,2)$. Then the naive power counting fails: replacing first $u$ by a smooth function on the right-hand side gives, by Schauder estimates, $u\in L^\infty_t C^\beta_x$ with $\beta=\alpha+2-2/q$, and so $b\cdot\nabla u\in L^q_t C^{\alpha+1-2/q}_x$.
Since $\alpha+1-2/q<\alpha$, iterating the procedure implies worse and worse spatial regularity on $u$, and after finitely many steps the product $b\cdot\nabla u$ becomes even ill-defined.
This is somewhat similar to the issue of the Kolmogorov equation of L\'evy SDEs with low stability index, which was circumvented in \cite{supcrit}. After this manuscript appeared,  Schauder estimates for \eqref{eq:PDE} with $b\in L^q_t C^\alpha_x$ with $q\in (1,2)$ were developed in \cite{wei2023stochastic}.
\end{remark}

\begin{remark}
By the embedding $L^p_x\subset C^{-d/p}_x$ our result immediately implies well-posedness of \eqref{eq:SDE} with $L^q_t L^p_x$ drift in the full subcritical regime (with respect to $p$) \eqref{eq:LqLp-exponents} if $q\in(1,2]$, which can be seen as a fractional analogue of \cite{KR}.
Note that unlike in \cite{Khoa}, $p\in[1,2)$ is also allowed.
\end{remark}

The rest of the article is structured as follows. In Section \ref{sec:counterexample} we present some counterexamples in the supercritical regime, demonstrating that (up to reaching the critical equality) condition \eqref{eq:main exponent} cannot be improved; we then conclude the introduction by recalling some fundamental properties of fBm in Section \ref{sec:prelim-fbm} and by introducing the main notations used throughout the paper in Section \ref{sec:notation}.
In Section \ref{sec:first} we state and prove some fundamental lemmata, including the aforementioned a priori estimates for solutions of \eqref{eq:SDE} and the two new forms of the stochastic sewing lemma of \cite{Khoa}.
Section \ref{sec:stability} contains further estimates for additive functionals of processes, as well as a key stability property of solutions.
In Sections \ref{sec:functional WP} and \ref{sec:distributional} we use these estimates to establish well-posedness of \eqref{eq:SDE}; we distinguish the cases $\alpha>0$ and $\alpha<0$ cases, which require a  different analysis.
Along the way we prove the existence of a solution semiflow, which we upgrade to a flow of diffeomphisms in Section \ref{sec:flow}.
Section \ref{sec:MKV} contains applications of our stability estimates to McKean-Vlasov equations.
In Section \ref{sec:weak-compactness} we construct weak solutions under condition \eqref{eq:condition-existence}, via a compactness argument enabled by the available a priori estimates.
In Section \ref{sec:rho} we show $\rho$-irregularity of solutions and more general perturbations of fractional Brownian motions.
Finally, Section \ref{sec:transport} contains applications to transport and continuity equations. 
In the appendices we collect some useful tools for which we did not find exact references in the literature: Appendix \ref{app:kolmogorov} contains variants of Kolmogorov continuity criterion, Appendix \ref{app:Young} gives two basic bounds for solutions of Young differential equations, and in Appendix \ref{app:girsanov} we summarise relations of various Sobolev spaces and their use in Girsanov transform for fractional Brownian motions.

\subsection{Counterexamples to uniqueness in the supercritical regime}\label{sec:counterexample}

Although the scaling argument is heuristic, one can often construct counterexamples in the supercritical case.
The constructions below are motivated by \cite{CDR1}, which gives counterexamples for $q=\infty$, $\alpha>0$.

Assume $d\geq 1$, $H\in(0,1)$ and $(\alpha,q)\in \R\times (1,\infty)$ satisfy
\begin{equ}\label{eq:counterexample-exponents}
\alpha<1-\frac{1}{H q'},\qquad \alpha>-1;
\end{equ}
let $B$ be an $\R^d$-valued stochastic process such that $\PP$-almost surely $B\in C^\gamma$ for all $\gamma\in(0,H)$.
We claim that under \eqref{eq:counterexample-exponents} there exists $b\in L^q_t C^\alpha_x$ such that the equation
\begin{equ}\label{eq:SDE-counterexample}
X_t=x+\int_0^t b_s(X_s)\,ds+B_t
\end{equ}
with initial condition $x_0=0$ has at least two solutions whose laws are mutually singular.

We will treat separately the cases $\alpha\in (0,1)$ and $\alpha\in (-1,0)$.

For $\alpha\in (0,1)$, the construction is actually one-dimensional and can be extended trivially to higher dimensions by taking $b=(b^i)_{i=1}^d$ with $b^i\equiv 0$ for $i\geq 2$; therefore here we will set $d=1$. Take $\tilde q>q$ such that $(\alpha,\tilde q)$ still satisfies \eqref{eq:counterexample-exponents} and define the function
\begin{equ}
b_t(x)=t^{-1/\tilde q}\,\sign(x)|x|^\alpha;
\end{equ}
clearly $b\in L^q_t C^\alpha_x$.
Let $\gamma=1/(\tilde q'(1-\alpha))$; by definition, $\gamma$ satisfies the identity
\begin{equ}\label{eq:exponents}
\gamma=1-\frac{1}{\tilde q}+\gamma\alpha,
\end{equ}
and furthermore $\gamma<H$ thanks to \eqref{eq:counterexample-exponents}. Fix furthermore $\delta>0$ small such that
$\delta^\alpha/\gamma>2\delta$, which exists since $\alpha\in (0,1)$.
Take $x\in(0,1]$ and consider a weak solution $(X^x,B)$ of \eqref{eq:SDE-counterexample}, which is well-known to exist due to the spatial continuity and sublinear growth of $b$.
Define the stopping time
\begin{equ}
\tilde\tau:=\inf\{t>0:\, |B_t|\geq \delta t^\gamma\}\wedge 1;
\end{equ}
it is strictly positive $\PP$-almost surely, since $\gamma<H$ and $B\in C^{\tilde\gamma}$ with $\tilde\gamma\in (\gamma,H)$. Also define
\begin{equ}
\tau_x:=\inf\{t\geq 0:\, X_t^x\leq\delta t^\gamma\}\wedge 1.
\end{equ}
We claim that $\tau_x\geq\tilde\tau$. Indeed, $\tau_x>0$ since $x>0$, and for all $t\leq \tau_x$ by \eqref{eq:SDE-counterexample} and our construction it holds
\begin{equ}\label{eq:counterexample2}
X_t^x>\int_0^ts^{-1/\tilde q}(\delta s^\gamma)^{\alpha}\,ds+B_t= (\delta^\alpha/\gamma) t^\gamma+B_t>\delta t^\gamma+\big(\delta t^\gamma+B_t)
\end{equ}
where in the intermediate passage we used \eqref{eq:exponents}.
Since $\tau_x\geq\tilde\tau>0$ $\PP$-a.s., there exist $\rho>0$ independent of $x\in (0,1]$ such that
\begin{equ}
\PP(\tilde\tau_x>\rho)\geq 3/4.
\end{equ}
The laws of $(X^x,B)$ on $C([0,1])^2$ are tight, and therefore by Skorohod's representation theorem, we may assume that for a sequence $x_n\searrow 0$ the random variables $(X^{x_n},B^{x_n})$ live on the same probability space and converge in $C([0,1])^2$ $\PP$-a.s.
The limit $(X^{0,+},B^{0,+})$ is a solution to \eqref{eq:SDE-counterexample} with initial condition $0$ and satisfies
\begin{equs}
\PP\big(X^{0,+}_t>0\,\,\forall t\in(0,\rho]\big)&\geq \PP\big(X^{0,+}_t\geq \delta t^\gamma\,\,\forall t\in[0,\rho]\big)
\\&=\lim_{n\to\infty}\PP\big(X^{x_n}_t\geq \delta t^\gamma\,\,\forall t\in[0,\rho]\big)\geq3/4.
\end{equs}
Since $b$ is odd, we can run the same argument for $y\in [-1,0)$: if $X^y$ is a solution to \eqref{eq:SDE-counterexample}, then  $-X^y$ is a solution for $(-y,-B)$, and the definition of $\tilde \tau$ only depends on $|B|$. Therefore for the same choice of $\rho$, in this case one finds that
\begin{equs}
\PP\big(X^y_t \leq -\delta t^\gamma \,\forall t\in(0,\rho]\big) \geq 3/4
\end{equs}
and so by considering a sequence $y_n\nearrow 0$ by compactness one can construct $(X^{0,-}, B^{0,-})$ another weak solution to \eqref{eq:SDE-counterexample} with initial condition $0$ satisfying
\begin{equ}
\PP\big(X^{0,-}_t<0\,\,\forall t\in(0,\rho]\big)\geq 3/4.
\end{equ}
This shows that $X^{0,+}$ and $X^{0,-}$ do not have the same law, yielding weak non-uniqueness 
(we leave it as an exercise to the reader to show that their laws are in fact mutually singular).

In the distributional case $\alpha\in(-1,0)$, we have to be a bit more careful, since the meaning of the SDE becomes unclear if $X$ gets too close to $0$. To this end, we argue again by stopping times, and the construction we present this time is genuinely $d$-dimensional.
Again take $\tilde q>q$ such that $(\alpha,\tilde q)$ still satisfy \eqref{eq:counterexample-exponents} and define a vector field $b=(b^i)_{i=1}^d$ by
\begin{equ}
b^1(t,x)=t^{-1/\tilde q}\,\sign(x^1)|x|^{\alpha},\qquad b^i(t,x)\equiv 0 \text{ for } i=2,\ldots,d;
\end{equ}
again $b\in L^q_t C^\alpha_x$. Take $x\in(0,1]$ and consider a local-in-time solution $X^x$ of \eqref{eq:SDE-counterexample} with initial condition $x_0=(x,0,\ldots,0)$, which is well-known to exist due to the spatial regularity of $b$ locally around $x_0$.
Define $\gamma$ as before, so that $\gamma<H$ and \eqref{eq:exponents} holds; let us furthermore take an auxiliary parameter $\delta$ that will be specified later. Define the stopping times
\begin{equ}
\tilde\tau:=\inf\{t>0:\, |B_t|\geq \delta t^\gamma\}\wedge 1, \quad \tau_x:=\inf\{t\geq 0:\, (X_t^x)^1\leq\delta t^\gamma\}\wedge 1;
\end{equ}
as before, $\tilde\tau$ is strictly positive $\PP$-almost surely since $\gamma<H$.
We claim that $\tau_x\geq\tilde\tau$, for which it suffices to show that for $t\leq \tau_x\wedge\tilde \tau$ one has $(X_t^x)^1\geq 2\delta t^\gamma$. If $x>3\delta t^\gamma$, then by simply using the nonnegativity of the first component of $b$ up to $\tau_x$ and the definition of $\tilde \tau$, we see that
\begin{equ}
(X_t^x)^1\geq x+B_t^1\geq 3\delta t^\gamma-\delta t^\gamma
\end{equ}
as required. 
Suppose now that $x\leq 3\delta t^\gamma$.
Clearly, for $s\leq\tau_x$ one also has $|X_s^x|\geq \delta s^\gamma$. Inserting this bound in the equation, we get for $s\leq \tau_x\wedge\tilde \tau$ 
\begin{equs}
(X_s^x)^1&\leq x+\int_0^sr^{-1/\tilde q}(\delta r^\gamma)^{\alpha}\,dr+B_s^1
\\
&=x+ (\delta^{\alpha}/\gamma) s^\gamma+B_s^1
\\
&\leq x+ \big(\delta^{\alpha}/\gamma+\delta\big)s^\gamma;
\end{equs}
observe that since $\alpha<0$, we find reversed inequalities compared to the previous case.
In particular, if $s\geq t/2$, then using $x\leq 3\delta t^\gamma$ we also get
\begin{equ}
(X_s^x)^1 \leq \big(3\delta 2^\gamma+\delta+\delta^{\alpha}/\gamma\big)s^\gamma.
\end{equ}
For $\delta\in (0,1)$, there exist constants $C',C$ depending only on $d,\alpha,\gamma$ such that the above bound implies $(X_s^x)^1\leq C'\delta^{\alpha}s^\gamma$, as well as $|X_s^x|\leq C \delta^{\alpha}s^\gamma$. Using this bound in the equation once more:
\begin{equs}
(X_t^x)^1&>\int_{t/2}^ts^{-1/\tilde q}\big(C \delta^{\alpha}s^\gamma)^{\alpha}\,ds+B_t^1
\\
&\geq (1/2)C^{\alpha}\delta^{\alpha^2}t^\gamma
-\delta t^\gamma.
\end{equs}
At this point (using the condition $\alpha>-1$, so that $\alpha^2<1$) one can choose $\delta$ small enough so that the right-hand side is bounded from below by $2\delta t^\gamma$. With this we conclude the proof of the property $\tau_x\geq \tilde \tau$.
In other words, for all $t\leq\tilde \tau$, for all $x\in(0,1]$, we have $(X_t^x)^1\geq\delta t^\gamma$. In a symmetric way, for all $t\leq\tilde \tau$, for all $y\in[-1,0)$ we have $(X_t^y)^1\leq-\delta t^\gamma$.

We now want to pass to the $x\to 0$ limit, which we can do by noticing that the laws of $(B,\tilde\tau,X^x,X^{-x})$ are tight on the space
\begin{equ}
\cS=
C([0,1])\times\{(a,g):\,a\in(0,1],g\in C([0,a])^2\}
\end{equ}
with the metric
\begin{equ}
d\big((f,a,g),(f',a',g')\big)=\|f-f'\|_{C([0,1])}+|a-a'|+\|g-g'\|_{C([0,a\wedge a'])^2}.
\end{equ}
By Prokhorov's theorem and Skorohod's representation, we get a sequence $x_n\to 0$ and on another probability space a sequence $(\bar B^{x_n},\bar{\tilde\tau}^{x_n},\bar X^{x_n},\bar X^{-x_n})\overset{\mathrm{law}}{=}(B,\tilde\tau,X^{x_n},X^{-x_n})$ converging $\PP$-almost surely as random variables taking values $\cS$. The limits $X^{0,+}:=\lim\bar X^{x_n}$ and $X^{0,-}:=\lim\bar X^{-x_n}$ both solve \eqref{eq:SDE-counterexample} with initial condition $0$ and driving noise $B^{0}:=\lim \bar B^{x_n}$. Moreover, $X^{0,+}_t\geq \delta t^\gamma$ for $t\leq\tilde\tau^0:=\lim\bar{\tilde\tau}^{x_n}$ and $X^{0,-}_t\leq-\delta t^\gamma$ for $t\leq\tilde\tau^0$.
Since $\tilde\tau^0\overset{\mathrm{law}}{=}\tilde\tau$, it is $\PP$-a.s. positive, and therefore the laws of $X^{0,+}$ and $X^{0,-}$ are mutually singular (for example on $C([0,1])$ after extending them as constants after $\tilde\tau^0$).

\begin{remark}
Up to multiplying $b$ by a cutoff function at infinity, by taking $\alpha=-d/(p+\eps)$ for sufficiently small $\eps>0$, the construction presented in the regime $\alpha<0$ provides non-uniqueness for $b\in L^q_tL^p_x$, for any pair $(p,q)\in [1,\infty]^2$ satisfying
\begin{equ}\label{eq:LqLp-counterexample}
\frac{1}{q}+\frac{Hd}{p}>1-H,\qquad p>d.
\end{equ}
If $H=1/2$, then $B$ can be taken as Brownian motion and \eqref{eq:LqLp-counterexample} becomes
\begin{equ}\label{eq:LqLp-counterexample-brownian}
\frac{2}{q}+\frac{d}{p}>1,\qquad p>d;
\end{equ}
in particular, the exponents $p,q$ violate the LPS condition \eqref{eq:LPS}. It is interesting to compare \eqref{eq:LqLp-counterexample-brownian} to the result from \cite{Krylov2020}, where weak existence for the Brownian SDE was established under the condition
\begin{equation}\label{eq:parameters-weak-existence-brownian}
\frac{1}{q}+\frac{d}{p}\leq 1
\end{equation}
which is further shown to be optimal by construction of counterexamples in the case $1/q+d/p>1$.
Let us also mention \cite{Gal2024} for a heuristic explanation on why condition \eqref{eq:parameters-weak-existence-brownian} (as well as \eqref{eq:parameters-weak-existence} below) arises naturally when only focusing on weak existence results.
Our counterexample shows that under \eqref{eq:parameters-weak-existence-brownian} uniqueness in law in general does not hold, answering a problem left open in \cite{Krylov2020} (see the discussion right above Remark 3.1 therein).

After the completion of this work, it has been further shown in \cite{butkovsky2023} that in the time-independent case, for $H\in (0,1)$, there exist $b\in C^\alpha$ with supercritical $\alpha<1-1/H$ for which even weak existence doesn't hold, see Theorem 2.7 therein. More recently, \cite{ButGal2023} expanded the result from \cite{Krylov2020} by establishing weak existence of solutions for $H\in (0,1)$ and $b\in L^q_t L^p_x$ with
\begin{equation}\label{eq:parameters-weak-existence}
	\frac{1-H}{q}+ \frac{H d}{p} < 1-H.
\end{equation}
Combined with our counterexample, one gets a regime (namely, the intersection of \eqref{eq:LqLp-counterexample} and \eqref{eq:parameters-weak-existence}) where weak existence holds but uniqueness in law does not.
\end{remark}

\subsection{Preliminaries on fractional Brownian motion}\label{sec:prelim-fbm}

We recall here several 
facts about fractional Brownian motion (fBm); for some standard references we refer to \cite{nualart2006,Picard}.

An $\R^d$-valued fBm of Hurst parameter $H$ is defined as the unique centered Gaussian process with covariance
\begin{equation*}
\E(B^H_t\otimes B^H_s)=\tfrac{1}{2}\big(|t|^{2H}+|s|^{2H}-|t-s|^{2H}\big) I_d
\end{equation*}
where $I_d$ denotes is the $d\times d$ identity matrix; in other words, its components are i.i.d. one dimensional fBms. FBm paths are well-known to be $\PP$-a.s. $(H-\eps)$-H\"older, but nowhere $H$-H\"older continuous. 
FBm admits several representations as a stochastic integral; in particular, given any fBm $B^H$ defined on a probability space, one can construct therein a standard Bm $W$ such that 
\begin{equ}\label{eq:Volterra}
B^H_t=\int_0^t K_H(t,r)\dd W_r\quad\forall\, t \geq 0.
\end{equ}
Such Volterra kernel representation is referred as \emph{canonical} since $B^H$ and $W$ generate the same filtration. The exact formula for the kernels $K_H$ can be found in e.g. \cite{NO1}, for our purposes it is enough to recall that $K_H$ is deterministic and $K_H(t,\cdot)\in L^2([0,t])$.

Another standard representation of fBm is the one introduced in \cite{MvN}: given $B^H$, one can construct a two-sided Bm $\tilde W$ such that
\begin{equation}\label{eq:MvN-representation}
B^H_t = \gamma_H \int_{-\infty}^t \big[(t-r)_+^{H-1/2}-(-r)_+^{H-1/2}\big]\, \dd \tilde W_r;
\end{equation}
where $\gamma_H=\Gamma(H+1/2)^{-1}$ is a normalizing constant and $x_+$ denotes the positive part.


We will mostly work with representation \eqref{eq:Volterra}, but we invite the reader to keep in mind \eqref{eq:MvN-representation} since it is usually easier to manipulate in order to derive key properties of the process, like its local nondeterminism, see \eqref{eq:LND-fBm} and the discussion below.
Given a filtration $\bF$, we say that $B^H$ is a $\bF$-fBM if the associated $W$ given by \eqref{eq:Volterra} is a $\bF$-Brownian motion.

FBm of parameter $H=1$ is somewhat trivial or ill-defined, see \cite{Picard};
however one can extend the definition to all values $H\in (0,+\infty)\setminus\N$ inductively as in \cite{IEEE} by $B^{H+1}_t:=\int_0^t B^H_s\dd s$.

Such definition is consistent with most aforementioned properties: it is still a centered, Gaussian process, with trajectories $\PP$-a.s. in $C^{H-\eps}_t$ but nowhere $C^H_t$, satisfying the scaling relation \eqref{eq:fBM scaling}; using stochastic Fubini one can also easily derive similar representations as \eqref{eq:Volterra}-\eqref{eq:MvN-representation}.
A key consequence of the last property is that for any $H\in (0,+\infty)\setminus \N$ there exists a constant $c_H\in (0, +\infty)$ such that
\begin{equation}\label{eq:LND-fBm}
{\rm Cov} \big(B^H_t - \E_s B^H_t \big) = c_H |t-s|^{2H} I_d \quad\forall \, s\leq t,
\end{equation}
see \cite[Proposition 2.1]{gerencser2020regularisation}; here $\E_s B^H_t:=\E[B^H_t|\mathcal{F}_s]$, where $\mathcal{F}_s$ can be the natural filtration of $B^H$ or more generally any filtration such that $B^H$ is a $\bF$-fBm.
Property \eqref{eq:LND-fBm} is a special form of \emph{strong local nondeterminism} (LND)\footnote{In fact, any integral in time of an LND Gaussian process admitting a \emph{moving average representation} in the style of \eqref{eq:MvN-representation} is still LND, see \cite[Sec. 4.2, Example iv.]{galeati2020prevalence}.}, see \cite[Section 2.4]{galeati2020prevalence} for a deeper discussion on its relevance on regularisation by noise. 
Since conditional expectations are also $L^2$-projections, $B^H_t-\E_s B^H_t$ and $\E_s B^H_t$ are orthogonal Gaussian variables, thus independent; more generally, $B^H_t-\E_sB^H_t$ is independent of all the history up to time $s$.
Therefore for any $s\leq t$, any bounded measurable function $f:\R^d\to\R$ and any other $\cF_s$-measurable random variable $X$, it holds
\begin{equ}\label{eq:conditioning}
\E_s f(B^H_t+X)= P_{ {\rm Cov}(B^H_t - \E_s B^H_t)} f(\E_s B^H_t+X) = P_{c_H |t-s|^{2H} I_d} f(\E_s B^H_t+X).
\end{equ}
where in the last passage we applied \eqref{eq:LND-fBm}; here given a symmetric nonnegative $\Sigma$, $P_\Sigma$ denotes the convolution with the Gaussian density $p_\Sigma$ associated to $\mathcal{N}(0,\Sigma)$. Throughout the paper we will adopt the convention that $P_{t I_d}=P_t$, in agreement with the standard notation for heat kernels, and for simplicity we will drop the constant $c_H$, so that in expressions like \eqref{eq:conditioning} only $P_{|t-s|^{2H}}$ will appear.

\begin{remark}\label{rem:possible-extensions}
At the price of slightly anticipating some key concepts which will be introduced throughout the paper, let us discuss here how our methods extend to a larger class of random perturbations $B^H$ than just pure fBm.
The main requirement we need, relaxing \eqref{eq:LND-fBm}, is for $B^H$ to be a Gaussian process\footnote{For non-Gaussian processes one can still find a replacement for \eqref{eq:requirements}, for example in the case of L\'evy processes see \cite{BDG-Levy}. } satisfying a two-sided bound
\begin{equation}\label{eq:requirements}
C^{-1} |t-s|^{2H} I_d \leq {\rm Cov} \big( B^H_t-\E_s B^H_t \big)\leq C |t-s|^{2H} I_d
\end{equation}
for some $C\in (0,+\infty)$ and for all $s<t$ with $|t-s|$ sufficiently small; here $\mathcal{F}_t$ is the natural filtration of $B^H$.
More precisely, the upper bound in \eqref{eq:requirements} provides a priori estimates in the style of Lemma \ref{lem:drift-regularity}, while the lower bound (which is the actual LND property), ensures the regularising effect of $B^H$ and the application of stochastic sewing techniques. Indeed, by using properties of Gaussian convolutions, heat kernel bounds and a relation of the form \eqref{eq:conditioning}, one can still find estimates of the form
\begin{align*}
\| \E_s f(B^H_t+X)\|_{L^\infty}
& = \| \big( P_{ {\rm Cov} (B^H_t-\E_s B^H_t)} f \big) (\E_s B^H_t+X)\|_{L^\infty}
\leq \| P_{ {\rm Cov} (B^H_t-\E_s B^H_t)} f\|_{L^\infty}\\
& \lesssim \| P_{C^{-1}|t-s|^{2H}} f\|_{L^\infty}
\lesssim |t-s|^{\alpha H} \| f\|_{C^\alpha},
\end{align*}
for $\alpha\leq 0$, which are the typical bounds needed throughout the proof.
There are some passages where condition \eqref{eq:requirements} alone is not enough and we exploited other properties of fBm.
Specifically: the counterxamples in Section \ref{sec:counterexample} assume $B^H$ to be $(H-\eps)$-H\"older continuous and symmetric;
the flows constructed in Sections \ref{sec:functional WP}-\ref{sec:distributional} need some basic time-continuity $\E|B^H_t-B^H_s|\lesssim |t-s|^{H\wedge 1}$ in order to apply Kolmogorov-type criteria;
more substantially, the results from Section \ref{sec:weak-compactness} rely on a Volterra representation $B^H_t =\int_0^t K(t,s) \dd W_s$. These properties are satisfied by other interesting examples, e.g. type-II fBm and mixed fBm discussed in Remark \ref{rem:other-fractional-processes} below.

The only section truly specific to fBm is Appendix  \ref{app:girsanov}, which however exactly for this reason is not used throughout the main body of the paper. In this case, ad hoc criteria to check Girsanov transform for fBm are presented; any extension to other processes would require precise knowledge of the associated kernel $K(t,s)$ and its verification can be very technical, cf. \cite{sonmez2021}.
\end{remark}

\begin{remark}\label{rem:other-fractional-processes}
Standard examples of processes satisfying \eqref{eq:requirements} are deterministic additive perturbations of fBm (cf. Lemma \ref{lem:perturbed-sde}), the so called type-II fBm \cite{marinucci1999} and mixed fBm introduced in \cite{cheridito2001}; given any $H_1\neq H_2$, the process $B^{H_1}+B^{H_2}$ will satisfy condition \eqref{eq:requirements} with $H=H_1\wedge H_2$, both in the case $B^{H_1}$ and $B^{H_2}$ are sampled independently and the one instead where they are constructed from the same reference Brownian motion. In this case, our results yield a far reaching generalization (also to any $d\geq 2$) of the ones provided in \cite{sonmez2021}, while not requiring highly technical use of Girsanov transform as therein.

Another interesting example is Bifractional Brownian motion of parameters $(H,K)$ (see \cite{RuTu2006}) which is known to be LND with parameter $HK$ \cite{TuXi2007}; it is a generalization of fBm ($K=1$), but even in the case $HK=1/2$ is not a semimartingale nor a Dirichlet process, although it scales like standard Bm. Our results show that it has a comparable regularising effect, although not amenable to Markovian/martingale techniques.

Another generalization of fBm is the so called multifractional Brownian motion, in which the Hurst parameter is allowed to vary continuously in time, $H=H(t)$; two non-equivalent definitions for this process are given respectively in \cite{peltier1995} (by modifying representation \eqref{eq:MvN-representation} by allowing $H=H(t)$) and in \cite{benassi1997} (by a harmonisable representation). In both cases, the process can be shown to be ``locally LND around $t$'' with parameter $H(t)$ (see \cite{AyShXi2011} in the harmonisable case) and thus we still expect our strategy to yield interesting results, under appropriate modifications. Likely, the admissible range of $\alpha$ here would depend on both the supremum and infimum of $H(t)$; we leave more precise investigations for future research.

Finally, let us mention that for (sufficiently regular) solutions $u(x,t)$ to certain linear stochastic PDEs for any fixed $x$ the process $t\mapsto u(x,t)$ is LND, see e.g. \cite{TuXia2017}; this fact was exploited crucially in regularisation by noise for nonlinear SPDEs in \cite{ABLM}.
\end{remark}

\subsection{Setup and notation}\label{sec:notation}
We provide here in a  list all the main notations and conventions adopted throughout the paper.
\begin{itemize}
\item We always work on the time interval $t\in [0,1]$. Increments of functions $f$ on $[0,1]$ are denoted by $f_{s,t}:=f_t-f_s$.
\item Whenever considering a filtered probability space $(\Omega,\cF,\bF,\PP)$, we will implicitly assume that the filtration $\bF=(\cF_t)_{t\in[0,1]}$ satisfies the standard assumptions; in particular, $\cF_0$ is complete.
To denote conditional expectations, we use the shortcut notation $\mathbb{E}_s Y :=\mathbb{E} [Y| \mathcal{F}_s]$. 
\item $L^m$-norms without further notation are understood with respect to $\omega$, that is, $\|Y\|_{L^m}=\big(\E|Y|^m\big)^{1/m}$ for $m<\infty$ and $\|Y\|_{L^\infty}=\mathrm{esssup}_{\omega\in\Omega}|Y(\omega)|$. For conditional $L^m$-norms we use the notation  $\|Y\|_{L^m|\cF_s}=\big(\E(|Y|^m|\cF_s)\big)^{1/m}.$
For any $X,Y\in L^m$ such that $Y$ is $\cF_s$-measurable, by conditional Jensen's inequality one has the $\PP$-a.s. bound
\begin{equ}\label{eq:conditional replacement}
\|X-\E_s X\|_{L^m|\cF_s} \leq \|X-Y\|_{L^m|\cF_s} +\|Y-\E_s X\|_{L^m|\cF_s} \leq 2\|X-Y\|_{L^m|\cF_s}.
\end{equ}
Apart from the usual $L^m$-norms, we also use the norms $\big\|\,\|\cdot\|_{L^m|\cF_s}\big\|_{L^n}$. We will always consider $n\geq m$, in which case again by conditional Jensen it holds
\begin{equation*}
\| X\|_{L^m} \leq \big\|\,\| X \|_{L^m|\cF_s}\big\|_{L^n}
\end{equation*}
with equality in the case $m=n$. Such mixed norms still satisfy natural analogues of classical inequalities like Jensen's, H\"older's and Minkowski's, as can be verified using properties of conditional expectation.
Moreover, by the tower property, one can see that for $t\geq s$, $\big\|\,\|\cdot\|_{L^m|\cF_t}\big\|_{L^n}$ is stronger than $\big\|\,\|\cdot\|_{L^m|\cF_s}\big\|_{L^n}$.
\item Whenever talking about a weak solution $X$ to the SDE \eqref{eq:SDE}, we will actually mean a tuple $(X,B^H; \Omega, \bF, \PP)$ such that $(\Omega,\bF,\PP)$ is a filtered probability space as above, $X$ is $\bF$-adapted and $B^H$ is a $\bF$-fBm of parameter $H$. As usual, $X$ is a strong solution if it is adapted to the (standard augmentation of) the filtration generated by $B^H$. We say that pathwise uniqueness holds for the SDE if for any two solutions $X^1$, $X^2$, defined on the same $(\Omega,\bF,\PP)$, driven by the same $B^H$ and with same initial condition $x_0$, it holds $X^1\equiv X^2$ $\PP$-a.s.
We warn the reader to keep in mind that all such concepts are rather classical when $b$ is at least a measurable function, so that \eqref{eq:SDE} is meaningful in the Lebesgue sense. In the distributional regime $\alpha<0$, this is not the case anymore, therefore the concept of weak solution becomes less standard; we postpone this discussion to the relevant Section \ref{sec:distributional}, similarly for the concept of path-by-path uniqueness.
\item Function spaces in the variable $x\in\R^d$ will often be denoted by the subscript $x$. For instance, standard Lebesgue spaces $L^p(\R^d;\R^m)$ with $p\in[1,\infty]$ will often be denoted, when the target dimension $m$ is clear, simply by $L^p_x$.
For $\alpha\in\R\setminus\N$, we denote by $C^\alpha_x$ the inhomogeneous H\"older-Besov space $B^{\alpha}_{\infty,\infty}$ (cf. \cite{BCD}); instead for nonnegative integer $\alpha$, by $C^\alpha_x$ we mean the space of bounded measurable functions whose all partial weak derivatives up to order $\alpha$ are also essentially bounded and measurable (in other words, $C^\alpha_x=W^{\alpha,\infty}_x$ Sobolev spaces); note that with this convention, elements of $C^0_x$ are not necessarily continuous.
Recall that for $\alpha\in(0,1)$ the space $C^\alpha_x=B^\alpha_{\infty,\infty}$ coincides with the usual space of bounded $\alpha$-H\"older continuous functions.
By $C^{\alpha,\loc}_x$ we mean the space of functions $f$ such that for all compactly supported smooth $g$ one has $f g\in C^\alpha_x$.
More quantitative versions of them are the weighted H\"older spaces $C^{\alpha,\lambda}_x$, for $\alpha\in(0,1]$ and $\lambda\in\R$, defined through the (semi)norms
\begin{equ}
\|f\|_{C^{\alpha,\lambda}_x}:=|f(0)|+\llbracket f\rrbracket_{C^{\alpha,\lambda}_x}:=|f(0)|+\sup_{R\geq 1} \sup_{x\neq y\in B_R} \frac{|f(x)-f(y)|}{|x-y|^\alpha\, R^\lambda},
\end{equ}
where $B_R$ is the ball of radius $R$ around the origin.
\item Given a Banach space $E$, we will use the shortcut notation $L^q_t E$ to denote the space $L^q([0,1];E)$ of Bochner measurable function with finite norm $\| f\|_{L^q E}^q=\int_0^1 \| f_t\|_E^q\, \dd t$, for any $q\in [1,\infty]$ (up to the standard essential supremum convention for $q=\infty$). We use the shortcut notation $C_t E = C([0,1];E)$ for the space of continuous, $E$-valued functions with supremum norm; similarly for $\gamma\in (0,1)$, $C^\gamma_t E= C^\gamma([0,1];E)$ is the space of $E$-valued, bounded and $\gamma$-H\"older continuous functions. All definitions can be extended classically to Fr\'echet spaces $E$ (in particular allowing for $E=C^{\alpha,\loc}_x$ or $L^{p,\loc}_x$), for instance in the the case of $L^q_t E$ by requiring the associated countable seminorms $t\mapsto \| f_t \|_k$ to be all $L^q$-integrable.
%
\item Given a metric space $E$ and $p\in[1,\infty)$, we say that a continuous $E$-valued function $f$ on $[0,1]$ is of finite $p$-variation, in notation $f\in C^{p-\var}_t E$, if 
\begin{equ}
\llbracket f\rrbracket_{p-\var,E}^p:=\sup\sum_{i=1}^n d_E(f_{t_{i-1}},f_{t_i})^p<\infty,
\end{equ}
where the supremum runs over all possible partitions $0=t_0\leq t_1\leq\cdots\leq t_n=1$ of $[0,1]$.
The $p$-variation seminorm on subintervals $[s,t]\subset[0,1]$ is defined similarly and denoted by $\llbracket\cdot\rrbracket_{p-\var,E;[s,t]}$. Whenever $E=\R^m$ for some $m\in\N$, for simplicity we just drop it and write $C^{p-\var}_t$, $\llbracket \cdot \rrbracket_{p-\var;[s,t]}$, similarly for $C^\alpha_t$.

\item All the notations introduced above can be concatenated, by considering a different Banach/Fréchet space at each step. The convention we adopt is that, when writing spaces with respect to different variables, this is to be read from left to right; for example $L^q_t C^\alpha_x L^m$ stands for $L^q\big([0,1], C^\alpha(\R^d,L^m(\Omega))\big)$. Similarly one can define e.g. $L^m C^{p-\var}_t C^{\alpha,\loc}_x$, $C^\gamma_t L^\infty_x$ and so on. Mind in particular that with this convention $C^\alpha_t C^\alpha_x\neq C^\alpha_{t,x}$, the latter denoting the space of $\alpha$-H\"older continuous functions in $(t,x)$.
\item Let us recall some standard heat kernel estimates: for any $\alpha\geq \beta$ there exists a constant $N=N(d,\alpha,\beta)$ such that, for all $t\in(0,1]$, one has the bound
\begin{equ}\label{eq:HK-estimates}
\|P_{t}f\|_{C^\alpha_x}\leq N t^{(\beta-\alpha)/2}\|f\|_{C^\beta_x};
\end{equ}
see \cite[Lemma A.10]{galeati2021noiseless} and the references therein for a more general statement.
\item For $0\leq S\leq T\leq 1$, we denote $[S,T]^2_\leq=\{(s,t)\in[S,T]^2:\,s\leq t\}$ and $[S,T]^3_\leq=\{(s,u,t)\in[S,T]^3:\,s\leq u\leq t\}$.
For $(s,t)\in[S,T]^2_\leq$, define $s_-=s-(t-s)$.
We then set the slightly more restricted sets of pairs/triples as
\begin{align*}
	& \overline{[S,T]}^2_\leq:=\{(s,t)\in[S,T]^2_\leq:\,s_-\geq S\},\\
	& \overline{[S,T]}^3_\leq=\{(s,u,t)\in[S,T]^3_\leq:\,(u-s)\wedge(t-u)\geq (t-s)/3,\,s_-\geq S\}.
\end{align*}
\item Given a Frechét space $E$ and a map $A:[S,T]^2_\leq\to E$, we define $\delta A:[S,T]^3_\leq \to E$ by
$\delta A_{s,u,t} = A_{s,t}- A_{s,u}- A_{u,t}$.
\item We say that a function $w:[0,1]^2_\leq\to\R_+ $ is a \emph{control} if it is continuous and superadditive, i.e. $w(s,u)+w(u,t)\leq w(s,t)$ for all $(s,u,t)\in[S,T]^3_\leq$. 
The most common controls for us will be of the form
\begin{equ}\label{eq:the-real-control}
w_{b,\alpha,q}(s,t):=\int_s^t\|b_r\|_{C^\alpha_x}^q\,dr.
\end{equ}
Recall that for any two controls $w_1,w_2$ and $\theta_1,\theta_2\in[0,\infty)$ such that $\theta_1+\theta_2\geq 1$, $w=w_1^{\theta_1}w_2^{\theta_2}$ is also a control (see \cite[Exercises~1.8,1.9]{FVBook}).
Note also that if $w$ is a control, $\psi$ is an $\R^m$-valued path and $\gamma\in(0,1]$, then
\begin{equ}\label{eq:trivial-var}
\|\psi\|_{\frac{1}{\gamma}-\var}\leq w(0,1)^\gamma\sup_{0\leq s<t\leq 1}\frac{|\psi_{s,t}|}{w(s,t)^\gamma};
\end{equ}
conversely, for $p\geq 1$, if $\psi\in C^{p-\var}_t$ then $w(s,t)=\llbracket \psi\rrbracket_{p-\var; [s,t]}^p$ is a control and $|\psi_{s,t}|\leq w(s,t)^{1/p}$, cf. \cite[Propositions 5.8-5.10]{FVBook}.
\item The space of probability measures on $\R^d$ is denoted by $\mathcal{P}(\R^d)$.
The law of a random variable $X$ is denoted by $\cL(X)$.
For $p\geq 1$ we denote the $p$-Wasserstein distance on $\mathcal{P}(\R^d)$ by $\mathbb{W}_p$, defined as
\begin{equ}
\mathbb{W}_p(\mu,\nu)^p=\inf_{\gamma\in \Gamma(\mu,\nu)}\int_{\R^{d}\times\R^d}|x-y|^p\gamma(\dd x,\dd y),
\end{equ}
where $\Gamma(\mu,\nu)$ is the set of all couplings of $\mu$ and $\nu$, i.e. the probability measures on $\R^{d}\times\R^d$ whose first and second marginals are $\mu$ and $\nu$ respectively. Note that $\mathbb{W}_p$ can take value $+\infty$ and is defined for any $\mu$, $\nu$, without any moment assumption.
\item When a statement contains an estimate with a constant depending on a certain set of parameters, in the proof we do not carry the constants from line to line. Rather, we write $A\lesssim B$ to denote the existence of a constant $N$ depending on the same set of parameters such that $A\leq N B$.
Whenever such set of parameters includes a parameter that is a norm (this will typically be the norm of the coefficient $b$), this dependence is always monotone increasing.

\end{itemize}

\section{A priori estimates and stochastic sewing}\label{sec:first}
The key consequence of the subcriticality condition \eqref{eq:main exponent} is that in terms of local nondeterminism, drifts of solutions are more regular than the noise; in particular, the solution decomposes as $X=\varphi+B^H$, where $\varphi$ plays the role of a \emph{slow variable}, while $B^H$ is the \emph{highly oscillating component}.\footnote{In the regularisation by noise literature, to the best of our knowledge this concept originates from \cite{CG16}, where a similar \emph{pathwise} solution ansatz leads to the formalism of nonlinear Young integrals, based on deterministic sewing. Here, also inspired by the works \cite{Khoa,FHL,gerencser2020regularisation,BDG-Levy}, we take a step further and readapt the concept to a more probabilistic setup, where a combination of \eqref{eq:key-quantity}, LND and stochastic sewing yields sharper results.}
This can be formulated as a precise quantitative bound, by looking at the best conditional error committed by predicting the process $\varphi_t$, given the history up to time $s$; more precisely, we look for estimates of the form
\begin{equation}\label{eq:key-quantity}
\big\|\|\varphi_t-\E_s\varphi_t\|_{L^m|\cF_s}\big\|_{L^\infty}\leq w(s,t)^{1/q}|t-s|^{1/q'+\alpha H}\quad \forall\, (s,t)\in [0,1]^2_{\leq}, 
\end{equation}
where $m\in [1,\infty)$, $w$ is a suitable control and $(q,\alpha,H)$ are the parameters related to $b$, $B^H$.

The subcritical regime $\alpha>1-1/(q' H)$ corresponds to the exponent $1/q'+\alpha H$ appearing in \eqref{eq:key-quantity} being greater than $H$; this is in stark contrast with the lower bound provided by the LND property of fBm \eqref{eq:LND-fBm}, which tells us that such an estimate cannot hold for $\varphi$ replaced by $B^H$, justifying the slow-fast heuristic above.

It is also worth pointing out that $1/q'+\alpha H$ is allowed to exceed $1$ (this is indeed always the case for $H>1$), which will be used crucially in the following; in this case, the same bound could not hold if in \eqref{eq:key-quantity} $\E_s \varphi_t$ were replaced by $\varphi_s$, as one can easily check that the only processes satisfying the corresponding condition are the constant ones.

It will become clear in the sequel why \eqref{eq:key-quantity} is exactly the right condition needed in our analysis; for the moment, let us show that solutions to SDEs naturally enjoy \eqref{eq:key-quantity}.

Lemma \ref{lem:drift-regularity} below is based on a readaption of \cite[Lemma~2.4]{gerencser2020regularisation}, \cite[Lemma~4.2]{BDG-Levy} to our setting.
Note that in the statement, while we enforce the subcritical condition $\alpha>1-1/(q'H)$, the restriction $q\leq 2$ is not necessary; we do however restrict to $\alpha\geq 0$ first. For distributional drifts, similar bounds will be derived from stochastic sewing, see Lemma \ref{lem:apriori-estim} below.

\begin{lemma}\label{lem:drift-regularity}
Let $H\in(0,\infty)\setminus\N$, $q\in[1,\infty)$, and $\alpha\in[0,1]$ satisfy $\alpha>1-1/(q'H)$; let $b\in L^q_t C^\alpha_x$, $X$ be a weak solution of \eqref{eq:SDE} and set $\varphi:=X-B^H$, so that
\begin{equation*}
\varphi_t = x_0 + \int_0^t b_r(X_r) \dd r.
\end{equation*}
Then, for any $m\in [1,\infty)$, there exists a constant $N=N(d,H,\alpha,m,\| b\|_{L^q_t C^\alpha_x})$ such that estimate \eqref{eq:key-quantity} holds with the choice
\begin{equation}\label{eq:drift-regularity}
w(s,t)=N w_{b,\alpha,q}(s,t)= N \int_s^t \| b_r\|_{C^\alpha}^q \dd r.
\end{equation}
\end{lemma}

\begin{proof}
First assume that, for some given $\beta \geq 0$, the bound \eqref{eq:key-quantity} holds with $w$ as above and exponent $\beta$ in place of $1/q'+\alpha H$.
This is definitely the case with $\beta=1/q'$, as one can see from
\begin{equation}\label{eq:drift-regularity-proof}\begin{split}
\big\|\|\varphi_t-\E_s\varphi_t\|_{L^m|\cF_s}\big\|_{L^\infty}
& \leq 2\big\|\|\varphi_t-\varphi_s\|_{L^m|\cF_s}\big\|_{L^\infty}\\
& \leq 2 \int_s^t \| b_r\|_{C^0}\, \dd r
\leq 2 w_{b,\alpha,q}(s,t)^{1/q} |t-s|^{1/q'};
\end{split}\end{equation}
in the above passages, we applied \eqref{eq:conditional replacement}, the definition of $\varphi$ and lastly H\"older's inequality.

Assuming we already have the bound for a generic $\beta\geq 1/q'$, we can then apply \eqref{eq:conditional replacement} for the choice $Y=\varphi_s + \int_s^t b_r(\E_s X_r) \dd r$, together with the definition of $\varphi$, to find
\begin{align*}
\| \varphi_t   -\mathbb{E}_s \varphi_t \|_{L^m|\cF_s}
& \leq 2 \Big\| \varphi_t - \varphi_s - \int_s^t b_r ( \mathbb{E}_s \varphi_r +\E_s B^H_r) \mathd r \Big\|_{L^m|\cF_s}\\
& \leq 2 \int_s^t \big\| b_r (\varphi_r + B^H_r) - b_r (\mathbb{E}_s \varphi_r +\E_s B^H_r) \big\|_{L^m|\cF_s} \,\mathd r\\
& \leq 2 \int_s^t \| b_r \|_{C^{\alpha}_x} \big\| \varphi_r -\mathbb{E}_s \varphi_r + B^H_r -\E_s B^H_r \big\|_{L^m|\cF_s}^{\alpha}\,\dd r\\
& \leq 2 \int_s^t \| b_r \|_{C^{\alpha}_x} \big( \| \varphi_r -\mathbb{E}_s \varphi_r\|_{L^m|\cF_s}^{\alpha} + \| B^H_r -\E_s B^H_r \|_{L^m|\cF_s}^{\alpha}\big)\,\dd r;
\end{align*}
in the above estimates we used multiple times basic properties of conditional norms like Jensen's and Minkowski's inequality.
By the properties of fBm recalled in Section \ref{sec:prelim-fbm} and the independence of $B_r^H-\E_s B^H_r$ from $\cF_s$, we have the bound
\begin{equation*}
\big\| \|B^H_r -\E_s B^H_r \|_{L^m|\cF_s} \big\|_{L^\infty}\lesssim |r-s|^H \quad \forall s\leq r.
\end{equation*}
Combined with our standing assumption on $\varphi$, by taking $L^\infty$-norms on both sides and using Minkowski's and H\"older's inequalities for the integral, we get
\begin{align*}
\big\|\| \varphi_t  -\mathbb{E}_s \varphi_t \|_{L^m|\cF_s}\big\|_{L^\infty}
& \lesssim \int_s^t \| b_r\|_{C^\alpha_x} \Big(  \big\|\| \varphi_r  -\mathbb{E}_s \varphi_r \|_{L^m|\cF_s}\big\|_{L^\infty}^\alpha + |r-s|^{\alpha H} \Big) \dd r\\
& \lesssim \int_s^t \| b_r\|_{C^\alpha_x} \Big(  w_{\alpha,b,q}(s,r)^{\alpha/q}|r-s|^{\alpha\beta} + |r-s|^{\alpha H} \Big) \dd r\\
& \lesssim w_{\alpha,b,q}(s,t)^{1/q}\Big(w_{\alpha,b,q}(s,t)^{\alpha/q}|t-s|^{\alpha\beta+1/q'} + |t-s|^{\alpha H+1/q'}\Big).
\end{align*}
In other terms, if $\varphi$ satisfies \eqref{eq:key-quantity} with $1/q'+\alpha H$ replaced by $\beta$, then it does so also with $\tilde{\beta}=f(\beta)=\alpha (\beta \wedge H)+1/q'$ (up to a change in the generic constant $N$).

From here, the argument is identical to the one from \cite[Lemma~2.4]{gerencser2020regularisation}: by iterating, we can define a sequence $\{\beta^n\}_n$ by $\beta^{n+1}= f(\beta^n)$ with $\beta_0=1/q'$; it remains to note that the condition $\alpha>1-1/(q'H)$ guarantees that the only fixed point $\bar{\beta}$ of the map $\tilde{f}(\beta)= \alpha \beta+1/q'$ is strictly larger than $H$ and is attracting exponentially fast any orbit defined by $\tilde{\beta}_{n+1}=\tilde{f}(\tilde \beta_n)$. Given that the sequences $\{\beta_n\}_n$ and $\{\tilde\beta_n\}_n$ coincide as long as the first one doesn't exceed $H$, this necessarily implies that the first one stabilizes to $\beta=\alpha H+ 1/q'$ after a finite number of iterations $\bar{n}$ (i.e. $\beta_n=\alpha H+ 1/q'$ for all $n\geq \bar{n}$).
\end{proof}

\begin{remark}\label{rem:drift-regularity}
The case $m=\infty$ can be handled with an appropriate stopping argument, see \cite[Lemma~2.4]{gerencser2020regularisation}. This can be used to derive similar bounds for processes that are not exact solutions (for example Picard iterates), but we do not need this generality.
\end{remark}

The next ingredient is an a priori estimate for $\alpha<0$, analogous to Lemma \ref{lem:drift-regularity}.
Recall that for any adapted process $\varphi$, by \eqref{eq:conditional replacement} one has
\begin{equation*}
\big\|\|\varphi_t-\E_s\varphi_t\|_{L^m|\cF_s}\big\|_{L^\infty}\leq 2\big\|\|\varphi_{s,t}\|_{L^m|\cF_s}\big\|_{L^\infty};
\end{equation*}
in the distributional case, we will directly bound the latter quantity.
Unlike Lemma \ref{lem:drift-regularity}, here we cannot allow for any $q\in(2,\infty]$ and subcritical $\alpha$, rather we need to impose the stronger condition \eqref{eq:condition-existence}, which was introduced just before Theorem \ref{thm:weak-existence-intro}.
%
\begin{remark}\label{rem:about-the-other-condition}
As mentioned in Remark \ref{rem:other-condition-intro}, for $q\in (1,2]$, condition \eqref{eq:condition-existence} reduces to \ref{eq:main exponent}. For $q\in(2,\infty)$, the a priori estimate below will be relevant in Section \ref{sec:weak-compactness}, where we establish existence of weak solutions in a regime where the uniqueness is not known.
Contrary to Lemma \ref{lem:drift-regularity}, the proof of Lemma \ref{lem:apriori-estim} will rely on stochastic sewing techniques. We could use the upcoming very general (but quite technical) Lemma \ref{lem:SSL1} for this task; but in order to help the intuition, we prefer first to invoke the result from \cite{FHL}, whose statement is simpler, and postpone the application of Lemma \ref{lem:SSL1} to where it is truly needed (e.g. Lemma \ref{lem:integral-estimates}).
\end{remark}

\begin{lemma}\label{lem:apriori-estim}
Assume \eqref{eq:condition-existence} and in addition $\alpha<0$.
Let $b\in L^q_tC^1_x$ and let $X$ be the unique strong solution to \eqref{eq:SDE} for some initial condition $x_0\in\R^d$; set $w:=w_{b,\alpha,q}$ and $\varphi=X-B^H$.
Then for any $m\in [1,\infty)$ there exists a constant $N=N(m,d,\alpha,q,H,\| b\|_{L^q_t C^\alpha_x})$ such that for all $(s,t)\in[0,1]_\leq^2$ one has the bound
\begin{equation}\label{eq:apriori-estim}
\big\|\| \varphi_{s,t}\|_{L^m\vert \mathcal{F}_s}\big\|_{L^\infty} \leq N w(s,t)^{1/q} |t-s|^{\alpha H + 1/q'}.
\end{equation}
\end{lemma}

\begin{proof}
Up to shifting, we can assume without loss of generality $x_0=0$; moreover we only need to deal with $m\in [2,\infty)$, since $\| \cdot\|_{L^m\vert \cF_s} \leq \| \cdot\|_{L^2\vert \cF_s}$ otherwise.
Fix $m\in [2,\infty)$, set the shorthand $\beta:=\alpha H +1/q'$; recall that by \eqref{eq:condition-existence}, one has $\beta>H$.

Let us first assume that \eqref{eq:apriori-estim} holds with $w$ replaced by another control $\tilde{w}$; this is definitely the case for $\tilde w = w_{b,0,q}$, arguing as in \eqref{eq:drift-regularity-proof}.
Given such $\tilde w$ and any closed subinterval $I\subset [0,1]$, define
\begin{equation*}
\llbracket \varphi\rrbracket_{\beta,\tilde w, I} := \sup_{s,t\in I, s<t} \frac{ \big\|\| \varphi_{s,t}\|_{L^m\vert \mathcal{F}_s} \big\|_{L^\infty}}{|t-s|^\beta \tilde w(s,t)^{1/q}}
\end{equation*}
with the convention $0/0=0$.
Fix $(s,t)\in[0,1]_\leq^2$ and, for any $(s',t')\in[s,t]_\leq^2$, set
$$
A_{s',t'}:=\E_{s'} \int_{s'}^{t'} b_r(\varphi_{s'}+B^H_r)\dd r
= \int_{s'}^{t'} P_{|r-s'|^{2H}} b_r (\varphi_{s'}+\E_{s'} B^H_r) \dd r
$$
where in the second passage we used conditional Fubini and property \eqref{eq:conditioning} (please remember our convention about not writing explicitly the constant $c_H$ or the matrix $I_d$).

Our aim is to apply the stochastic sewing lemma (in the version given by \cite[Theorem 2.7]{FHL}) to $A$ in order to find a closed estimate for $\llbracket \varphi\rrbracket_{\beta,w,I}$. By the heat kernel estimates \eqref{eq:HK-estimates}, we have $\PP$-almost surely
\begin{equs}
| A_{s',t'}|
&\leq \int_{s'}^{t'} \| P_{|r-s'|^{2H}} b_r\|_{C^0} \dd r
\lesssim \int_{s'}^{t'} |r-s'|^{\alpha H} \|b_r\|_{C^\alpha} \dd r
\lesssim |t'-s'|^\beta w(s',t')^{1/q},
\end{equs}
where in the last passage we applied H\"older's inequality and the $L^{q'}$-integrability of $|r-s|^{\alpha H}$ follows from \eqref{eq:condition-existence}.
Similarly, we have the $\PP$-a.s. bound
\begin{align*}
| \E_{s'} \delta A_{s',u',t'} |
& = \bigg| \E_{s'} \E_{u'} \int_{u'}^{t'} b_r(\varphi_{s'}+B_r)- b_r(\varphi_{u'} + B_r) \dd r \bigg|\\
& \lesssim \bigg| \int_{u'}^{t'} |r-u'|^{H(\alpha-1)} \| b_r\|_{C^\alpha} \dd r\, \E_{s'}| \varphi_{s',u'} | \bigg|\\
& \lesssim |t'-s'|^{2\beta-H}  w(s',t')^{1/q} \tilde w(s',t')^{1/q} \llbracket \varphi \rrbracket_{\beta,\tilde w, [s,t]}.
\end{align*}
The integrability of the power follows again from \eqref{eq:condition-existence}, as do the inequalities $\beta+1/q>1/2$, $2\beta-H +2/q>1$ (we remark that it is only the latter for which the additional condition in \eqref{eq:condition-existence} was introduced).
Therefore the stochastic sewing lemma \cite[Theorem 2.7]{FHL} applies and allows us to derive estimates for the sewing $\cA$ associated to $A$. However, one can easily identify $\cA_\cdot$; indeed, by the spatial regularity of $b$, we have the bound
\begin{equ}
\|\varphi_{s',t'}-A_{s',t'}\|_{L^m}\lesssim |t'-s'|^\eps\, w_{b,1,1}(s',t')
\end{equ}
for some $\eps>0$, which allows to conclude that $\cA_{\cdot}=\varphi_{s,\cdot}$ again by \cite[Theorem 2.7-(b)]{FHL}.
Overall, we deduce that there exists a constant $N_0=N_0(m,d,\alpha,q,H)$ such that
\begin{equation*}
\big\| \|\varphi_{s',t'}\|_{L^m\vert \mathcal{F}_{s'}} \big\|_\infty
\leq N_0 |t'-s'|^\beta w(s',t')^{1/q} \Big(1 + |t'-s'|^{\beta-H} \tilde{w}(s',t')^{1/q} \llbracket\varphi\rrbracket_{\beta,\tilde w,[s',t']} \Big).
\end{equation*}
Diving both sides by $|t'-s'|^\beta w^{1/q}(s',t')$, taking supremum over $[s',t']\subset [s,t]$ and using the fact that all our estimates are on $[s,t]\subset [0,1]$, we obtain
\begin{equation}\label{eq:proof-existence-0}
\llbracket \varphi \rrbracket_{\beta,w,[s,t]} \leq N_0  \Big(1 + |t-s|^{\beta-H} \tilde{w}(s,t)^{1/q}\llbracket\varphi\rrbracket_{\beta,\tilde w,[s,t]} \Big).
\end{equation}
In particular, \eqref{eq:proof-existence-0} shows that $\llbracket \varphi \rrbracket_{\beta,w,[s,t]}$ is finite;
we can then go again through the whole argument, with $\tilde{w}$ replaced by $w$, to find
\begin{equation}\label{eq:proof-existence-1}
\llbracket \varphi \rrbracket_{\beta,w,[s,t]} \leq N_0  \Big(1 + |t-s|^{\beta-H} w(s,t)^{1/q} \llbracket\varphi\rrbracket_{\beta,w,[s,t]} \Big)
\end{equation}
which readily yields a closed estimate for $\llbracket \varphi \rrbracket_{\beta,w,[s,t]}$, at least for $[s,t]$ sufficiently small.

Our last task is to remove the smallness condition on $[s,t]$ in order to achieve a global bound.
To this end, define a new control $w_\ast$ by $w_\ast(s,t)^{1/q+\beta-H}=w(s,t)^{1/q} |t-s|^{\beta-H}$ and an increasing sequence $\{t_n\}_n$ by $t_0=0$ and $w_\ast(t_n,t_{n+1})^{1/q+\beta-H}=(2N_0)^{-1}$. Applying \eqref{eq:proof-existence-1} for $[s,t]=[t_n,t_{n+1}]$, by construction one finds $\llbracket \varphi \rrbracket_{\beta,w,[t_n,t_{n+1}]} \leq 2N_0$.

If $t_1=1$, this immediately yields the conclusion. Suppose this is not the case, then for any pair $s<t$ which do not belong to the same subinterval $[t_n,t_{n+1}]$, there exist $\ell,m\in\N$ such that $t_{\ell-1}< s\leq t_\ell\leq \ldots \leq t_m\leq t<t_{m+1}$. Set $\tau_{\ell-1}=s$, $\tau_i=t_i$ for $i=\ell,\ldots, m$ and $\tau_{m+1}=t$. It holds
\begin{align*}
\big\| \| \varphi_{s,t}\|_{L^m\vert \mathcal{F}_s}\big\|_{L^\infty}
& \leq \sum_{i=\ell-1}^{m} \big\|\| \varphi_{\tau_i,\tau_{i+1}}\|_{L^m\vert \mathcal{F}_s} \big\|_{L^\infty}
\leq \sum_{i=\ell-1}^{m} \big\|\| \varphi_{\tau_i,\tau_{i+1}}\|_{L^m\vert \mathcal{F}_{\tau_i}} \big\|_{L^\infty}\\
& \lesssim_{N_0} \sum_{i=\ell-1}^{m} w(\tau_i,\tau_{i+1})^{1/q} |\tau_i-\tau_{i+1}|^\beta \\
& \leq (m+1-\ell)^{-\alpha H} \Big( \sum_{i=\ell-1}^{m} \big [w(\tau_i,\tau_{i+1})^{1/q} |\tau_i-\tau_{i+1}|^\beta\big]^{\frac{1}{1+\alpha H}} \Big)^{1+\alpha H}\\
& \leq (m+1-\ell)^{-\alpha H} w(s,t)^{1/q} |t-s|^\beta
\end{align*}
where in the last two passages we used the fact that $\beta+1/q=1+\alpha H\in (0,1)$, Jensen's inequality and the superadditivity of the control $[w(s,t)^{1/q} |t-s|^\beta]^{\frac{1}{1+\alpha H}}$. Observe that $m+1-\ell$ is less or equal to the overall amount of intervals $[t_n,t_{n+1}]$. In turn, by their definition and subadditivity of $w_\ast$, this is bounded by a multiple of
\[ w_\ast (0,1)=w(0,1)^{\frac{(\alpha H + 1-H)^{-1}}{q}}=\| b\|_{L^q C^\alpha}^{(\alpha H + 1-H)^{-1}}\]
which finally yields the conclusion.
\end{proof}

Next we formulate two appropriate versions of the stochastic sewing lemma (SSL).
After its introduction by L\^e \cite{Khoa}, in recent years the SSL has seen many variations.
Our first SSL combines three modifications: it incorporates shifting (as in \cite{gerencser2020regularisation}), as well as controls and general $\big\|\,\|\cdot\|_{L^m|\cF_s}\big\|_{L^n}$ norms (as in \cite{FHL, Khoa-Banach}).
Let us remark that this combination is not completely obvious and comes with a price: due to the shifting, we need a nontrivial ``time component'' $|t-s|^\eps$ in our estimates, which does not appear in \cite{FHL, Khoa-Banach}.
Nontheless, the resulting statement is well suited for our applications, where such ``time component'' always appears naturally.

Recall the notations from Section \ref{sec:notation}, concerning $[0,1]_\leq$, $\overline{[S,T]}^2_\leq$, $s_{-}$ and so on.

\begin{lemma}\label{lem:SSL1}
Let $w_1,w_2$ be controls, and let $m,n$ satisfy $2\leq m\leq n\leq \infty$ and $m<\infty$. Let $(S,T)\in[0,1]_\leq$. Assume that $(A_{s,t})_{(s,t)\in\overline{[S,T]}^2_\leq}$ is a continuous mapping from $\overline{[S,T]}^2_\leq$ to $L^m$ such that for all $(s,t)\in\overline{[S,T]}^2_\leq$, $A_{s,t}$ is $\cF_t$-measurable.
Suppose that there exist constants $\eps_1,\eps_2>0$ such that the bounds
\begin{align}
\big\|\|A_{s,t}\|_{L^m|\cF_s}\big\|_{L^n}&\leq w_1(s_-,t)^{1/2}|t-s|^{\eps_1},\label{eq:SSL-cond1} \\
\|\E_{s_-}\delta A_{s,u,t}\|_{L^n}&\leq w_2(s_-,t)|t-s|^{\eps_2}\label{eq:SSL-cond2}
\end{align}
hold for all $(s,u,t)\in\overline{[S,T]}^3_\leq$.
Then for all $S<s\leq t\leq T$ the Riemann sums
\begin{equation}\label{eq:SSL-conv}
\sum_{j=0}^{2^\ell-1} A_{s+j2^{-\ell}(t-s),s+(j+1)2^{-\ell}(t-s)}
\end{equation}
converge as $\ell\to\infty$ in $L^m$, to the increments $\cA_t-\cA_s$ of
an adapted stochastic process $(\cA_t)_{t\in[S,T]}$ that is continuous as a mapping from $[S,T]$ to $L^m$ and $\cA_S=0$.
Moreover $\cA$ is the unique such process that satisfies the bounds
\begin{align}
\big\|\|\cA_{t}-\cA_s-A_{s,t}\|_{L^m|\cF_s}\big\|_{L^n}&\leq K_1 w_1(s_-,t)^{1/2}|t-s|^{\eps_1}+K_2 w_2(s_-,t)|t-s|^{\eps_2},\label{eq:SSL-conc1}
\\
\|\E_{s_-}\big(\cA_t-\cA_s-A_{s,t}\big)\|_{L^n}&\leq K_2 w_2(s_-,t)|t-s|^{\eps_2},\label{eq:SSL-conc2}
\end{align}
with some $K_1,K_2$ for all $(s,u,t)\in\overline{[S,T]}^3_\leq$. Furthermore, there exists a constant $K$ depending only on $\eps_1,\eps_2,m,n,d$ such that the bounds \eqref{eq:SSL-conc1}-\eqref{eq:SSL-conc2} hold with $K_1=K_2=K$, and moreover the bound
\begin{equation}\label{eq:SSL-conc3}
\big\|\|\cA_t-\cA_s\|_{L^m|\cF_s}\big\|_{L^n}\leq K\big( w_1(s,t)^{1/2}|t-s|^{\eps_1}+w_2(s,t)|t-s|^{\eps_2}\big)
\end{equation}
holds for all $(s,t)\in[S,T]^2_\leq$.
\end{lemma}

\begin{proof}
Since by the time of the present work there is an abundance of SSLs in the recent literature, we do not aim to give a fully self-contained proof. We only provide the details as long as the combination of the arguments of \cite{gerencser2020regularisation} and \cite{FHL, Khoa-Banach} is nontrivial.

\emph{Step 1 (convergence along dyadic partitions).}
Let $(s,t)\in\overline{[S,T]}_\leq^2$ and for each $k=0,1,\ldots$ define $\cD_k=\{t_0^k,t_1^k,\ldots,t_{2^k}^k\}$, where $t_i^k=s+i2^{-k}(t-s)$, and
set
\begin{equ}
\cA^k_{s,t}=\sum_{i=1}^{2^k}A_{t_{i-1}^k,t_i^k}.
\end{equ}
We claim that $\cA^k_{s,t}$ converges and its limit $\tilde\cA_{s,t}$ satisfies the bounds \eqref{eq:SSL-conc1}-\eqref{eq:SSL-conc2} with $K=K_1=K_2$ when replacing $\cA_t-\cA_s$ by it.
In particular, this would also imply the bound
\begin{equ}\label{eq:SSL-proof0}
\big\|\|\tilde \cA_{s,t}\|_{L^m|\cF_s}\big\|_{L^n}\leq K\big( w_1(s_-,t)^{1/2}|t-s|^{\eps_1}+w_2(s_-,t)|t-s|^{\eps_2}\big)
\end{equ}
for all $(s,t)\in\overline [S,T]^2_\leq$.
The claim clearly follows from the following two bounds:
\begin{equs}
\big\|\| \cA^{k-1} _{s,t} -\cA^{k} _{s,t}   \|_{L^m|\cF_s}\big\|_{L^n}  & \lesssim  w_1(s_-,t)^{1/2}|t-s|^{\eps_1}2^{-k\eps_1}+w_2(s_-,t)|t-s|^{\eps_2}2^{-k\eps_2},    \label{eq:A-pi-1/2}
\\
\| \E_{s_-} \big(   \cA^{k-1} _{s,t} -\cA^{k} _{s,t}  \big)  \| _{L^n}& \lesssim  w_2(s_-,t)|t-s|^{\eps_2}2^{-k\eps_2}.   \label{eq:A-pi-1}
\end{equs}
It is no loss of generality to assume $k\geq 2$ (otherwise the trivial bounds below suffice), in which case we write
\begin{equ}\label{eq:SSL-proof1}
\cA^{k+1}_{s,t}-\cA^{k}_{s,t}=-\delta A_{t_0^{k},t_1^{k},t_2^{k}}-\sum_{j=1}^{2^{k-1}-1}\delta A_{t_{2j}^{k},t_{2j+1}^{k},t_{2j+2}^{k}}.
\end{equ} 
For the first term we used the conditions \eqref{eq:SSL-cond1}-\eqref{eq:SSL-cond2} in a trivial way:
\begin{equs}
\,&\big\|\|\delta A_{t_{0}^{k},t_{1}^{k},t_{2}^{k}}\|_{L^m|\cF_s}\big\|_{L^n}  \lesssim w_1(t_{0}^{k}-(t_{2}^{k}-t_{0}^{k}),t_{2}^{k})^{1/2}|t_{2}^{k}-t_{0}^{k}|^{\eps_1}\lesssim w_1(s_-,t)^{1/2}|t-s|^{\eps_1}2^{-k\eps_1},
\\
\,&\|\E_{s_-} \delta A_{t_{0}^{k},t_{1}^{k},t_{2}^{k}}\|_{L^n}  \leq w_2(t_{0}^{k}-(t_{2}^{k}-t_{0}^{k}),t_{2}^{k})|t_{2}^{k}-t_{0}^{k}|^{\eps_2}\lesssim w_2(s_-,t)|t-s|^{\eps_2}2^{-k\eps_2}.
\end{equs}
For the sum in \eqref{eq:SSL-proof1}
we write
\begin{equs}
\sum_{j=1}^{2^{k-1}-1}\delta A_{t_{2j}^{k},t_{2j+1}^{k},t_{2j+2}^{k}}&=
\sum_{j=1}^{2^{k-1}-1}\E_{t_{2j-2}^k}\delta A_{t_{2j}^k,t_{2j+1}^k,t_{2j+2}^k}
\\
&\qquad+\sum_{\ell=0}^1\sum_{j=0}^{2^{k-2}}(\id-\E_{t_{4j+2\ell}^k}\big)\delta A_{t_{4j+2\ell+2}^k,t_{4j+2\ell+3}^k,t_{4j+2\ell+4}^k}
\\
&=:I_1+I_2,\label{eq:SSL-proof1.5}
\end{equs}
where the term $\delta A_{t_{2^k}^k,t_{2^k+1}^k,t_{2^k+2}^k}$ is defined to be $0$.
The point of this unaesthetic decomposition is twofold. First, since $t^k_{2j-2}=t^k_{2j}-(t^k_{2j+2}-t^k_{2j})$, in the terms in the first sum there is sufficient shifting in the conditioning so that they can be estimated via the assumed bound \eqref{eq:SSL-cond2}. Second, for each $\ell=0,1$, the inner sum above is one of martingale differences.

Therefore, we first estimate by the triangle inequality
\begin{equs}
\big\|\|I_1\|_{L^m|\cF_s}\big\|_{L^n}&\leq\sum_{j=1}^{2^{k-1}-1}\big\|\|\E_{t_{2j-2}^k}\delta A_{t_{2j}^k,t_{2j+1}^k,t_{2j+2}^k}\|_{L^m|\cF_s}\big\|_{L^n}  
\\
&\leq\sum_{j=1}^{2^{k-1}-1}\|\E_{t_{2j}^k-(t_{2j+2}^k-t_{2j}^k)}\delta A_{t_{2j}^k,t_{2j+1}^k,t_{2j+2}^k}\big\|_{L^n}  
\\
&\leq \sum_{j=1}^{2^{k-1}-1} w_2(t_{2j-2}^k,t_{2j+2}^k)|t_{2j+2}^k-t_{2j}^k|^{\eps_2}
\\
&\lesssim  |t-s|^{\eps_2}2^{-k\eps_2}w_2(s,t),\label{eq:SSL-proof2}
\end{equs}
using the superadditivity of $w_2$ in the last line.
Similarly, but replacing the triangle inequality by the Burkholder-Davis-Gundy and Minkowski inequalities (e.g. in the form given in \cite[Lemma 2.5]{Khoa-Banach} for $\mathfrak{p}=2$),
we have
\begin{equs}
\big\|\|I_2\|_{L^m|\cF_s}\big\|_{L^n}&\lesssim\sum_{\ell=0}^1\Big(\sum_{j=0}^{2^{k-2}}\big\|\|\delta A_{t_{4j+2\ell+2}^k,t_{4j+2\ell+3}^k,t_{4j+2\ell+4}^k}\|_{L^m|\cF_s}\big\|_{L^n}^2\Big)^{1/2}
\\
&\lesssim 2^{-k\eps_1}\sum_{\ell=0}^1  \Big(\sum_{j=0}^{2^{k-2}}w_1(t_{4j+2\ell}^k,t_{4j+2\ell+4}^k)\Big)^{1/2}
\\
&\lesssim |t-s|^{\eps_1}2^{-k\eps_1}w_1(s,t)^{1/2}.\label{eq:SSL-proof3}
\end{equs}
This proves \eqref{eq:A-pi-1/2}. As for \eqref{eq:A-pi-1}, it is only easier: noting that
\begin{equ}
\E_s\sum_{j=1}^{2^{k-1}-1}\delta A_{t_{2j}^{k},t_{2j+1}^{k},t_{2j+2}^{k}}=\E_s I_1,
\end{equ}
we can bound $\|\E_sI_1\|_{L^n}\leq\|I_1\|_{L^n}$ just as in \eqref{eq:SSL-proof2}. This concludes the proof of \eqref{eq:A-pi-1/2}-\eqref{eq:A-pi-1}.

\emph{Step 2 (convergence along regular partitions)}. Let us say that a partition $\pi=\{s=t_0<t_1<\cdots<t_n=t\}$ is regular, if $|\pi|:=\max(t_i-t_{i-1})\leq 2\min(t_i-t_{i-1})$. For any partition we can define 
\begin{equ}
\cA^{\pi}_{s,t}=\sum_{i=1}^n A_{t_{i-1},t_i}.
\end{equ}
Very similarly to Step 1, we get that for any sequence of regular partitions $(\pi_n)_{n\in\N}$ with $|\pi_n|\to 0$, $\cA^{\pi}_{s,t}$ converges (for details see \cite[Lemma~2.2]{gerencser2020regularisation}). Therefore on one hand this limit has to coincide with $\tilde\cA_{s,t}$, on the other hand, this limit is clearly additive.
Moreover notice that by construction $\tilde\cA_{s,t}$ is $\cF_t$-measurable for all $(s,t)\in\overline{[S,T]}_\leq^2$, and since it vanishes in $L^m$, the additivity implies that it is continuous in both arguments as a two-parameter process with values in $L^m$.

\emph{Step 3 (the process $\cA$ and its bounds)}. For any $t\in(S,T]$ we set $t_i:=S+2^{-i}(t-S)$. We then claim that the series
\begin{equ}
\cA_t:=\sum_{i=1}^\infty \tilde\cA_{(S+2^{-i})\wedge t,(S+2^{-i+1})\wedge t}=:\sum_{i=1}^\infty \tilde\cA_{s_i,s_{i-1}}
\end{equ}
converges. Indeed, since $(s_i,s_{i-1})\in\overline{[S,T]}^2_\leq$, we may use the bound \eqref{eq:SSL-proof0}.
By the trivial bounds $w((s_i)_-,s_{i-1})\leq w(S,t)$ and $|s_{i-1}-s_i|\leq 2^{-i}\one_{t-S\geq 2^{-i}}$, we get not only the convergence of the series but also the bound
\begin{equ}
\big\|\|\cA_t\|_{L^m|\cF_S}\big\|_{L^n}\leq K\big( w_1(S,t)^{1/2}|t-S|^{\eps_1}+w_2(S,t)|t-S|^{\eps_2}\big).
\end{equ}
This is precisely \eqref{eq:SSL-conc3} with $s=S$. The case for general $(s,t)\in[S,T]_{\leq}^2$ follows in the same way.
It is also clear that $\cA_0=0$, and by the remarks in Step 2, that $\cA$ is adapted and continuous in $L^m$. Therefore $\cA$ satisfies all of the claimed properties.

\emph{Step 4 (Uniqueness).} The proof of this is standard and can be found in e.g. \cite{Khoa-Banach}.
\end{proof}

The other version of SSL that we use seems to be new.
In Lemma \ref{lem:SSL1} one can transfer $L^m$ bounds from $A$ to $\cA$ if $m<\infty$.
The $m=\infty$ case is a bit different: $L^\infty$ bounds on $A$ imply Gaussian moment bounds on $\cA$.
An alternative way to obtain Gaussian moment bounds via stochastic sewing is presented in \cite{BDG} (see e.g. Theorem 3.3. and Lemma 4.6. therein), but the conditions herein are easier to verify.
The proof relies on a conditional version of Azuma--Hoeffding inequality, see Lemma \ref{lem:azuma-hoeffding-conditional} in Appendix \ref{app:kolmogorov}.


%
%

\begin{lemma}\label{lem:SSL2}
Let $(A_{s,t})_{(s,t)\in\overline{[S,T]}^2_\leq}$ be a continuous mapping from $\overline{[S,T]}^2_\leq$ to $L^2$, with $A_{s,t}$ $\cF_t$-measurable for all $(s,t)\in\overline{[S,T]}^2_\leq$, such that the conditions of Lemma \ref{lem:SSL1} hold with $m=n=\infty$; namely, assume that there exist controls $w_1,w_2$ and constants $\eps_1,\eps_2>0$ such that the bounds
\begin{align}
\big\|\|A_{s,t}\|_{L^\infty|\cF_s}\big\|_{L^\infty}&\leq w_1(s_-,t)^{1/2}|t-s|^{\eps_1},\label{eq:SSL2-cond1} \\
\|\E_{s_-}\delta A_{s,u,t}\|_{L^\infty}&\leq w_2(s_-,t)|t-s|^{\eps_2}\label{eq:SSL2-cond2}
\end{align}
hold for all $(s,u,t)\in\overline{[S,T]}^3_\leq$. Denote by $(\mathcal{A}_t)_{t\in [S,T]}$ the associated process coming from Lemma \ref{lem:SSL1}.
Then there exist positive constants $\mu$ and $K$ depending only on $\eps_1,\eps_2,d$ such that the bound
\begin{equation}\label{eq:SSLexp-conc}
\E\bigg[\exp\bigg(\mu\,\frac{|\cA_t-\cA_s|^2}{\big(w_1(s,t)^{1/2}|t-s|^{\eps_1}+w_2(s,t)|t-s|^{\eps_2}\big)^2}\bigg)\bigg\vert \mathcal{F}_s\bigg]\leq K
\end{equation}
holds for all $(s,t)\in[S,T]_{\leq}^2$.
\end{lemma}

\begin{proof}
We continue using the notation of the proof of Lemma \ref{lem:SSL1}.
Let $(s,t)\in\overline{[S,T]}_\leq^2$ and $k=0,1,\ldots$, and let us bound $\cA_{s,t}^{k+1}-\cA_{s,t}^k$.
The first term on the right-hand side of \eqref{eq:SSL-proof1} is trivially bounded by $2w_1(s_-,t)^{1/2}|t-s|^{\eps_1}2^{-k\eps_1}$ with probability $1$.
Decomposing the second term into $I_1$ and $I_2$ as in \eqref{eq:SSL-proof1.5}, a simple use of triangle inequality as in \eqref{eq:SSL-proof2} yields the $\PP$-almost sure bound
\begin{equ}
|I_1|\lesssim 2^{-k\eps_2}|t-s|^{\eps_2}w_2(s,t).
\end{equ}
As for $I_2$, recalling that it is the sum of two martingales, for each we may use the Azuma-Hoeffding inequality. The role of $\delta_j$ as in Lemma \ref{lem:azuma-hoeffding-conditional} is played by $4w_1(t_{4j+2\ell}^k,t_{4j+2\ell+4}^k)^{1/2}$, so similarly to the calculation as in \eqref{eq:SSL-proof3}, we get
\begin{equ}
\Lambda:=\sum_i \delta_i^2 \lesssim 2^{-2k\eps_1}|t-s|^{2\eps_1}w_1(s,t).
\end{equ}
Therefore by \eqref{eq:azuma-hoeffding-conditional}, combined with the aforementioned $\PP$-almost sure bounds, we get that with some $\mu_1>0$, $K_1$
\begin{equ}
\E\bigg[\exp\Big(\mu_12^{k(\eps_1\wedge\eps_2)}\frac{|\cA_{s,t}^{k+1}-\cA_{s,t}^k|^2}{(w_1(s_-,t)^{1/2}|t-s|^{\eps_1}+w_2(s_-,t)|t-s|^{\eps_2})^2}\Big)\bigg\vert \mathcal{F}_{S}\bigg]\leq K_1.
\end{equ}
Since one can write
\begin{equ}
|(\cA_t-\cA_s)-A_{s,t}|\leq \sum_{k=0}^\infty 2^{-k(\eps_1\wedge\eps_2)}2^{k(\eps_1\wedge\eps_2)}|\cA_{s,t}^{k+1}-\cA_{s,t}^k|,
\end{equ}
we get by conditional Jensen's inequality
\begin{equs}
\E&\bigg[\exp\Big(\mu_1\frac{|(\cA_t-\cA_s)-A_{s,t}|^2}{(w_1(s_-,t)^{1/2}|t-s|^{\eps_1}+w_2(s_-,t)|t-s|^{\eps_2})^2}\Big)\bigg\vert\mathcal{F}_{S}\bigg]\leq \sum_{k=0}^\infty 2^{-k(\eps_1\wedge\eps_2)}K_1.
\end{equs}
Using again the assumed bounds on $A_{s,t}$, we get with some other constant $K_2$ 
\begin{equs}
\E&\bigg[\exp\Big(\mu_1\frac{|\cA_t-\cA_s|^2}{w_1(s_-,t)^{1/2}|t-s|^{\eps_1}+w_2(s_-,t)|t-s|^{\eps_2}}\Big)\bigg\vert \mathcal{F}_{S}\bigg]\leq K_2.
\end{equs}
It only remains to remove the shifts in the denominator and substitute $\mathcal{F}_S$ with $\mathcal{F}_s$, which can be done just as in Step 3 of the proof of Lemma \ref{lem:SSL1}, and therefore we obtain \eqref{eq:SSLexp-conc}.
\end{proof}

\section{Stability}\label{sec:stability}
The use of the tools from Section \ref{sec:first} is illustrated by the following lemma, which will play a key role in our analysis.
Let us emphasise the important feature of the statement that although $h$ is assumed to have $\delta$ spatial regularity, in the estimate only its $\alpha-1$ norm is used.

\begin{lemma}\label{lem:integral-estimates}
Assume \eqref{eq:main exponent} and let $(S,T)\in[0,1]_\leq^2$.
Suppose that $h\in L^q_t C^\delta_x$ for some $\delta>0$ and let $\varphi$ be an adapted process satisfying \eqref{eq:key-quantity} with $m=1$ and some control $w$.
For $t\in [S,T]$, define the process
\begin{equ}
\psi_t=\int_S^t h_r\big(B^H_r+\varphi_r\big)\dd r
\end{equ}
and set $\eps=1/q'+(\alpha-1)H$.
Then there exists positive constants $\mu$ and $K$, depending only on $H$, $q$, $\alpha$, and $d$, such that for all $(s,t)\in[S,T]^2_\leq$ one has the bound
\begin{equation}\label{eq:first average bound}
\E\bigg[ \exp\bigg(\mu\,\frac{|\psi_t-\psi_s|^2}{w_{h,\alpha-1,q}(s,t)^{2/q}|t-s|^{2\eps} \big(1+w(s,t)^{1/q}|t-s|^{\eps})^2}\bigg)\bigg\vert \mathcal{F}_s\bigg]\leq K.
\end{equation}
As a consequence, for any $\widetilde m\in[1,\infty)$ there exists a constant $\tilde K$, depending only on
 $\tilde m$, $H$, $q$, $\alpha$, and $d$, such that for all $(s,t)\in[S,T]^2_\leq$ one has the bound
 \begin{equation}\label{eq:second average bound}
 \big\|\|\psi_t-\psi_s\|_{L^{\widetilde m}|\cF_s}\big\|_{L^\infty}\leq \tilde K w_{h,\alpha-1,q}(s,t)^{1/q}|t-s|^{\eps}\big(1+w(s,t)^{1/q}|t-s|^{\eps}\big).
 \end{equation}
\end{lemma}

\begin{proof}
Note that thanks to the condition \eqref{eq:main exponent}, $\eps>0$.
For $(s,t)\in\overline{[S,T]}^2_\leq$ let us set
\begin{equ}
 A_{s,t}=\E_{s-(t-s)}\int_s^t h_r(B^H_r+\E_{s-(t-s)}\varphi_r)\dd r,
\end{equ}
and verify the conditions of Lemma \ref{lem:SSL2} (namely those of Lemma \ref{lem:SSL1} with $m=n=\infty$).

Fix $(s,u,t)\in\overline{[S,T]}_\leq^3$ and denote $s_1=s-(t-s)$, $s_2=s-(u-s)$, $s_3=u-(t-u)$, $s_4=s$, $s_5=u$, $s_6=t$. These points are almost ordered according to their indices, except $s_3$ and $s_4$, for which $s_4\leq s_3$ may happen, but this plays no role whatsoever.
First, by property \eqref{eq:conditioning} we have
\begin{equ}
A_{s,t}=\int_s^t P_{|r-s_1|^{2H}}h_r\big(\E_{s_1}(B^H_r+\varphi_r)\big)\dd r.
\end{equ}
Therefore, by \eqref{eq:HK-estimates} and H\"older's inequality, it holds
\begin{equs}
|A_{s,t}|\leq\int_s^t \|P_{|r-s_1|^{2H}}h_r\|_{C^0_x}\dd r\lesssim &\int_s^t|r-s_1|^{(\alpha-1) H}\|h_r\|_{C^{\alpha-1}_x}\dd r
\\&\lesssim |t-s|^{1/q'+(\alpha-1) H}w_{h,\alpha-1,q}(s,t)^{1/q}.
\end{equs}
Since $q\leq 2$, by the definition of $\eps$, \eqref{eq:SSL-cond1} is satisfied with $\eps_1=\eps$ and $w_1=N w_{h,\alpha-1,q}^{2/q}$.

Next, we need to bound $\E_{s-(t-s)}\delta A_{s,u,t}=\E_{s_1}\delta A_{s_4,s_5,s_6}$.
After an elementary rearrangement we get
\begin{equs}
\E_{s_1}\delta A_{s_4,s_5,s_6}=I+J:&=\E_{s_1}\E_{s_2}\int_{s_4}^{s_5} h_r(B^H_r+\E_{s_1}\varphi_r)-h_r(B^H_r+\E_{s_2}\varphi_r)\dd r
\\
&\quad+\E_{s_1}\E_{s_3}\int_{s_5}^{s_6} h(B^H_r+\E_{s_1}\varphi_r)-h_r(B^H_r+\E_{s_3}\varphi_r)\dd r.
\end{equs}
The two terms are treated in exactly the same way, so we only detail $I$.
We use \eqref{eq:HK-estimates} similarly as before to get
\begin{equs}
|I|&\leq\E_{s_1}\int_{s_4}^{s_5}\big|P_{|r-s_2|^{2H}}h_r(\E_{s_2}B^H_r+\E_{s_1}\varphi_r)-P_{|r-s_2|^{2H}}h_r(\E_{s_2}B^H_r+\E_{s_2}\varphi_r)\big|\dd r
\\
&\leq \E_{s_1}\int_{s_4}^{s_5}\|P_{|r-s_2|^{2H}} h_r\|_{C^1_x}|\E_{s_1}\varphi_r-\E_{s_2}\varphi_r|\dd r
\\
&\lesssim \E_{s_1}\int_{s_4}^{s_5}|r-s_2|^{(\alpha-2)H}\|h_r\|_{C^{\alpha-1}_x}|\E_{s_1}\varphi_r-\E_{s_2}\varphi_r|\dd r.
\end{equs}
By Jensen's inequality and the assumption on $\varphi$ we have the $\PP$-almost sure bound
\begin{equ}
\E_{s_1}|\E_{s_1}\varphi_r-\E_{s_2}\varphi_r|\leq\E_{s_1}|\E_{s_1}\varphi_r-\varphi_r|\leq w(s_1,r)^{1/q}|t-s|^{1/q'+\alpha H}.
\end{equ}
Also note that $r\mapsto|r-s_2|^{(\alpha-2)H}\in L^{q'}([s_4,s_5])$ because of the shifted basepoint, in general this would not be true with $s_2$ replaced by $s_4$. Therefore, by H\"older's inequality
\begin{equ}
|I|\lesssim |t-s|^{1/q'+(\alpha-2)H+1/q'+\alpha H}w_{h,\alpha-1,q}(s,t)^{1/q}w(s_1,t)^{1/q}.
\end{equ}
Note that the exponent of $|t-s|$ is simply $2\eps$.
Using again that $q\leq 2$, we see that condition \eqref{eq:SSL-cond2} is satisfied with $\eps_2=2\eps$ and $w_2=N w_{h,\alpha-1,q}(s,t)^{1/q}w(s_1,t)^{1/q}$.

It remains to verify that the process $\cA$ of Lemma \ref{lem:SSL1} is given by $\psi$. Since $\psi_0=0$, it suffices to show that 
\begin{equ}\label{eq:identify cA}
\|\psi_t-\psi_s-A_{s,t}\|_{L^1}\leq \tilde{w}(s_-,t)|t-s|^\kappa
\end{equ}
for all $(s,t)\in\overline{[S,T]}^2_{\leq}$, with some control $\tilde w$ and some $\kappa>0$.
This follows from three easy bounds: first,
\begin{equs}
\Big\|&\psi_t-\psi_s-\int_s^t h_r(B^H_r+\E_{s_-}\varphi_r\big)\dd r\Big\|_{L^1}
\\&\leq \int_s^t\|h_r\|_{C^\delta_x} w(s_-,r)^{\delta/q}\dd r
\leq w_{h,\delta,q}(s,t)^{1/q}|t-s|^{1/q'}w(s_-,t)^{\delta/q},
\end{equs}
second,
\begin{equs}
\Big\|\int_s^t &h_r(B^H_r+\E_{s_-}\varphi_r\big)\dd r-\int_s^t h_r(\E_{s_-}B^H_r+\E_{s_-}\varphi_r\big)\dd r\Big\|_{L^1}
\\&\leq \int_s^t\|h_r\|_{C^\delta_x} |r-s_-|^{\delta H}\dd r
\lesssim w_{h,\delta,q}(s,t)^{1/q}|t-s|^{1/q'+\delta H},
\end{equs}
and third,
\begin{equs}
\Big\|\int_s^t &h_r(\E_{s_-}B^H_r+\E_{s_-}\varphi_r\big)\dd r-A_{s,t}\Big\|_{L^1}
\\&\leq \int_s^t\|h_r-P_{|r-s_-|^{2H}}h_r\|_{C^0_x} \dd r
\lesssim w_{h,\delta,q}(s,t)^{1/q}|t-s|^{1/q'+\delta H}.
\end{equs}
Hence we can conclude $\psi=\cA$ and \eqref{eq:first average bound} follows from \eqref{eq:SSLexp-conc}.
\end{proof}

We will often consider \eqref{eq:SDE} with nonzero initial time. If $b$ is a function, a solution of \eqref{eq:SDE} on some interval $[S,T]\subset[0,1]$ with initial condition $X_S$ is a process $X$ satisfying
\begin{equ}
X_t=X_S+\int_S^t b_r(X_r)\dd r+B^H_t-B^H_S
\end{equ}
for all $t\in[S,T]$.
Our main stability estimate for solutions is then formulated as follows.

\begin{theorem}\label{thm:stability}
Assume \eqref{eq:main exponent}. Let $\delta>0$.
Let $[S,T]\subset [0,1]$ and
for $i=1,2$ let $X^i$ be adapted continuous processes satisfying \eqref{eq:SDE} on $[S,T]$ with initial conditions $X^i_S$ and drifts $b^i\in L^q_t C^{1+\delta}_x$.
Denote $M=\max_{i=1,2}\|b^i\|_{L^q_t C^\alpha_x}$.
Then for any $m\in[2,\infty)$ there exists a positive constant $N=N(m,M,H,\alpha,q,d)$, such that one has the $\PP$-almost sure bound
\begin{equation}\label{eq:stability-estimate}
\Big\|\sup_{t\in[S,T]}|X^1_t-X^2_t|\Big\|_{L^m|\cF_s}\leq N\Big(|X_S^1-X_S^2|+\|b^1-b^2\|_{L^q_t([S,T];C^{\alpha-1}_x)}\Big).
\end{equation}
Moreover, if $b^1=b^2$, then one also has the $\PP$-almost sure bound
\begin{equation}\label{eq:inverse-estimate}
\Big\|\sup_{t\in[S,T]}\big(|X^1_t-X^2_t|^{-1}\big)\Big\|_{L^m|\cF_s}\leq N |X_S^1-X_S^2|^{-1}.
\end{equation}
\end{theorem}

\begin{proof}
As usual, we denote $\varphi^1=X^1-B^H$ and $\varphi^2=X^2-B^H$.
For $t\in [S,T]$, we write
\begin{equs}
X^1_t-X^2_t&=X^1_S-X^2_S+\int_S^t\Big(\int_0^1\nabla b^1_r\big(B^H_r+\lambda\varphi^1_r+(1-\lambda)\varphi^2_r\big)\dd \lambda\Big)\cdot(X^1_r-X^2_r)\dd r
\\
&\qquad+\int_0^t(b^1-b^2)_r(B^H_r+\varphi^2_r)\dd r.\label{eq:diff-eq}
\end{equs}
Note that $\nabla b^1\in L^q_tC^\delta_x$, and therefore the process
\begin{equ}
A_t:=\int_0^1A^\lambda_t\dd \lambda:=\int_0^1\Big(\int_S^t\nabla b^1_r\big(B^H_r+\lambda\varphi^1_r+(1-\lambda)\varphi^2_r\big)\dd r\Big)\dd \lambda
\end{equ}
is well defined.
Define furthermore
\begin{equ}
z_t:=\int_0^t(b^1-b^2)_r(B^H_r+\varphi^2_r)\dd r.
\end{equ}
We then apply Lemma \ref{lem:integral-estimates} with $\varphi=\lambda\varphi^1_r+(1-\lambda)\varphi^2_r$ and $h=\nabla b^1$, as well as with $\varphi=\varphi^2$ and $h=b^1-b^2$.
Since $\varphi^1$ and $\varphi^2$ are the drift parts of solutions, by Lemma \ref{lem:drift-regularity} the processes $\varphi=\lambda\varphi^1+(1-\lambda)\varphi^2$ satisfy the bound \eqref{eq:drift-regularity} with control $w=w_{b^1,\alpha,q}+w_{b^2,\alpha,q}$,
 and so Lemma \ref{lem:integral-estimates} indeed applies.
Combining the bound \eqref{eq:first average bound} with Lemma \ref{lem:kolmogorov}, we get that there exist random variables $\eta_A,\eta_z$ with Gaussian moments\footnote{Note that in terms of the coefficients, the moments of $\eta_A$ depend on $w_{b^1,\alpha,q}+w_{b^2,\alpha,q}$, while the moments of $z$ depend only on $w_{b^2,\alpha,q}$.} conditionally on $\cF_S$, as well as $\delta>0$ and $p\in(1,2)$, such that
\begin{equs}
\|A\|_{p-\var;[S,T]}&\leq w_{b^1,\alpha,q}(S,T)^{1/q}\sup_{S\leq s<t\leq T}\frac{|A_t-A_s|}{w_{b^1,\alpha,q}(s,t)^{1/q}|t-s|^\delta}
\\
&\leq w_{b^1,\alpha,q}(S,T)^{1/q}\eta_A,
\\
\|z\|_{p-\var;[S,T]}&\leq w_{b^1-b^2,\alpha-1,q}(S,T)^{1/q}\sup_{S\leq s<t\leq T}\frac{|z_t-z_s|}{w_{b^1-b^2,\alpha-1,q}(s,t)^{1/q}|t-s|^\delta}
\\
&\leq w_{b^1-b^2,\alpha-1,q}(S,T)^{1/q}\eta_z.
\end{equs}
We can rewrite  \eqref{eq:diff-eq} as
\begin{equ}\label{eq:diff-Young}
\dd(X^1_t-X^2_t)= \dd A_t (X^1_t-X^2_t)+\dd z_t, \quad (X^1_t-X_t^2)\vert_{t=S}=X^1_S-X^2_S,
\end{equ}
meaning that we are interpreting \eqref{eq:diff-eq} as an affine Young differential equation; see also Appendix \ref{app:Young} for more details.
By applying Lemma \ref{lem:young-estimate} for $x=X^1-X^2$ and $\tilde{p}=p$, we get
\begin{equ}
\sup_{t\in[S,T]}|X_t^1-X_t^2|\lesssim e^{C\|A\|_{p-\var;[S,T]}^p}\big(|X^1_S-X^2_S|+\|z\|_{p-\var;[S,T]}\big).
\end{equ}
Recall that $\eta_A$ satisfies $\E_S [e^{\mu \eta_A^2}]\lesssim 1$ for some $\mu>0$, thus also $\E_S[e^{K \eta_A^p}] \lesssim_{K,p} 1$ for all $K>0$ since $p<2$. Therefore we obtain
\begin{equs}
\E_S\Big[ \sup_{t\in[S,T]}|X_t^1-X_t^2|^m \Big]
& \lesssim  \E_S[e^{mC\| A\|_{p-\var; [S,T]}^p} ] |X^1_S-X^2_S|^m
\\&\qquad+ \E_S\Big[ e^{mC\| A\|_{p-\var;[S,T]}^p} \| z\|_{p-\var;[S,T]}^m \Big]\\
& \lesssim |X^1_S-X^2_S|^m + w_{b^1-b^2,\alpha-1,q}(S,T)^{m/q},
\end{equs}
using conditional H\"older's inequality to get the last line.
This gives \eqref{eq:stability-estimate}.

In case $b^1=b^2$, we have $z=0$ and the Young equation \eqref{eq:diff-Young} becomes homogeneous. Moreover, note that Young equations allow time-reversal: if we fix $\tau\in[S,T]$, write $\tilde A_t=A_{\tau-t}$, and
\begin{equ}
\dd Y_t=\dd \tilde A_t Y_t,\quad Y_t\vert_{t=0}=X^1_\tau-X^2_\tau,
\end{equ}
then $Y_{\tau-S}=X_S^1-X_S^2$. Therefore by Lemma \ref{lem:young-estimate} we also have the pathwise estimate
\begin{equ}
|X_S^1-X_S^2|\lesssim e^{C\|\tilde A\|_{p-\var;[0,\tau-S]}^p}|X^1_\tau-X^2_\tau|.
\end{equ}
Of course $\|\tilde A\|_{p-\var;[0,\tau-S]}^p=\| A\|_{p-\var;[S,\tau]}^p \leq \| A \|_{p-\var;[S,T]}^p$, so after rearranging for the inverses, taking supremum in $\tau\in[S,T]$, and taking $L^m|\cF_S$ norms, we get \eqref{eq:inverse-estimate}.
\end{proof}

\section{Strong well-posedness for functional drift}\label{sec:functional WP}
We first apply the stability estimate to establish existence and uniqueness of solutions of \eqref{eq:SDE} with $\alpha>0$.
In this case the meaning of solutions is unambiguous, but we will also need the following stronger concepts of solutions.

In the next definition, we denote by $C^{\loc}_x$ the space of continuous functions from $\R^d$ to itself, endowed with the topology of uniform convergence on compact sets. Correspondingly, $L^1_t C^{\loc}_x$ denotes the set of functions $f:[0,1]\times \R^d\to \R^d$ such that, for all smooth compactly supported $g$, $f g\in L^1([0,T]; C_b(\R^d;\R^d))$, where $C_b(\R^d;\R^d)$ denotes the Banach space of continuous and bounded functions, endowed with the supremum norm. As for most localized spaces, it is easy to check that $L^1_t C^{\loc}_x$ is a separable Fréchet space.

\begin{definition}\label{def:flow-classical}
\begin{enumerate}
\item[(i)] Assume $b\in L^1_t C^{\loc}_x$ and let $\gamma:[0,1]\to\R^d$ be bounded and measurable.
A semiflow associated to the ODE
\begin{equ}\label{eq:ODE simple}
y_t = y_0 +  \int_0^t b_s (y_s) \dd s + \gamma_t
\end{equ}
is a jointly measurable map $\Phi:[0,1]^2_\leq\times\R^d\to\R^d$ such that
\begin{itemize}
\item for all $(s,x)\in[0,1]\times\R^d$ and all $t\in[s,1]$ one has
\begin{equ}
\Phi_{s\to t}(x)=x+\int_s^t b_r\big(\Phi_{s\to r}(x)\big)\dd r+\gamma_t-\gamma_s;
\end{equ}
\item for all $(s,r,t,x)\in[0,1]^3_\leq\times \R^d$ one has $\Phi_{s\to t}(x)=\Phi_{r\to t}\big(\Phi_{s\to r}(x)\big)$.
\end{itemize}
\item[(ii)] A flow is a semiflow such that for all $(s,t)\in\times[0,1]^2_\leq$ the map $x\mapsto \Phi_{s\to t}(x)$ is a homeomorphism of $\R^d$.
\item[(iii)] If $\gamma$ is a stochastic process, a random (semi)flow  is a jointly measurable map $\Phi:\Omega\times[0,1]^2_\leq\times\R^d\to\R^d$ such that for $\PP$-almost all $\omega\in\Omega$, the map $\Phi^\omega:[0,1]^2_\leq\times\R^d\to\R^d$ is a (semi)flow associated to \eqref{eq:ODE simple} with $\gamma=\gamma(\omega)$.
\item[(iv)] We say that a random (semi)flow is adapted if for all $(s,t,x)\in[0,1]^2_\leq\times\R^d$, the random variable $\Phi_{s\to t}(x)$ is $\cF_t$-measurable.
\item[(v)] Given $\beta\in(0,1)$, we say that a (semi)flow is locally $\beta$-H\"older continuous if for all $K$ there exists a constant $N$ such that for all $(s,t,x,y)\in[0,1]_\leq^2\times B_K^2$ one has $|\Phi_{s\to t}(x)-\Phi_{s\to t}(y)|\leq N|x-y|^\beta$.
\end{enumerate}
\end{definition}

\begin{remark}
Definition \ref{def:flow-classical} is based on Kunita's classical one, cf. \cite[Theorem II.4.3]{kunita1984stochastic}; it is slightly different (in fact, stronger) from other definitions proposed in the literature, like \cite[Definition 5.1]{fedrizzi2013holder}, due to the ordering of the quantifiers. One can draw a nice analogy between this kind of difference and the one between so called \textit{crude} and \textit{perfect} random dynamical systems, cf. \cite[Remark 2.5]{zhang2010stochastic}.
\end{remark}

\begin{theorem}\label{thm:functional-existence}
Assume \eqref{eq:main exponent}, $\alpha>0$, and let $b\in L^q_t C^\alpha_x$. Then there exists an adapted random semiflow of solutions to \eqref{eq:SDE} that is furthermore $\PP$-almost surely locally $\beta$-H\"older continuous for all $\beta\in(0,1)$.
\end{theorem}
\begin{proof}
Let $m\in[2,\infty)$, to be specified later.
Take a sequence of functions $(b^{n})_{n\in\N}$ such that $b^{n}\in L^q_tC^{2}_x$ and $\|b^{n}\|_{L^q_t C^\alpha_x}\leq \|b\|_{L^q_t C^\alpha_x}$ for all $n\in\N$,
and $\|b^{n}-b\|_{L^q_t C^{\alpha-1}_x}\to 0$ as $n\to\infty$.
Replacing $b$ by $b^{n}$ in \eqref{eq:SDE}, the equation clearly admits an adapted random semiflow which we denote by $\Phi^{n}$.
For fixed $(s,t)\in[0,1]^2_\leq$, $x\in\R^d$, and $n,n'\in\N$, we may apply Theorem \ref{thm:stability} to obtain the bound
\begin{equ}
\big\|\Phi_{s\to t}^{n}(x)-\Phi_{s\to t}^{n'}(x)\big\|_{L^m}\lesssim \|b^{n}-b^{n'}\|_{L^q_t C^{\alpha-1}_x}.
\end{equ}
Here and below the only important feature of the hidden proportionality constant in $\lesssim$ is that it is independent of $n,n'$.
Next, let $(s,s',t),(s,s',t')\in[0,1]^3_\leq$, $x,x'\in\R^d$, and $n\in\N$.
Then from applying Theorem \ref{thm:stability} again we get
\begin{equ}
\big\|\Phi_{s\to t}^{n}(x)-\Phi_{s\to t}^{n}(x')\big\|_{L^m}\lesssim |x-x'|;
\end{equ}
by a trivial estimate we get
\begin{equ}
\big\|\Phi_{s\to t}^{n}(x)-\Phi_{s\to t'}^{n}(x)\big\|_{L^m}\lesssim |t-t'|^{H\wedge (1/q')},
\end{equ}
and using the semigroup property and Theorem \ref{thm:stability} once more we have
\begin{equs}[eq:s-regularity]
\|\Phi_{s\to t}^{n}(x)-\Phi_{s'\to t}^{n}(x)\|_{L^m}&=\|\Phi_{s'\to t}^{n}(\Phi_{s\to s'}^{n}(x))-\Phi_{s'\to t}^{n}(x)\|_{L^m}\\&\lesssim\|\Phi_{s\to s'}^{n}(x)-x\|_{L^m}\lesssim |s'-s|^{H\wedge (1/q')}.
\end{equs}
We therefore get that the sequence $\big(\Phi^{n}\big)_{n\in\N}$ is on the one hand Cauchy in $C_{s,t,x} L^m_\omega$, and on the other hand, bounded in $C_{s,t}C^1_xL^m_\omega\cap C_x C^{H\wedge (1/q')}_{s,t}L^m_\omega$.
This implies that for some random field $\Phi$,
one has $\Phi^{n}\to \Phi$ in $C_{s,t}C^{1-\kappa}_xL^m_\omega\cap C_x C^{H\wedge (1/q')-\kappa}_{s,t} L^m_\omega$, where $\kappa>0$ is arbitrary.
By Kolmogorov's continuity theorem, for sufficiently large $m$, the convergence also holds in $L^m_\omega C_{s,t}C^{1-2\kappa,\loc}_x\cap L^m_\omega C^{\loc}_x C^{H\wedge (1/q')-2\kappa}_{s,t}$.
This yields the claimed spatial regularity of $\Phi$; the fact that $\Phi$ is indeed a semiflow for \eqref{eq:SDE} instead follows from the locally uniform convergence of $\Phi^n$ to $\Phi$, $\Phi^n$ being semiflows, and the spatial continuity of the drift $b$.
\end{proof}

\begin{theorem}\label{thm:functional:PBP-uniqueness}
Assume \eqref{eq:main exponent}, $\alpha>0$, and let $b\in L^q_t C^\alpha_x$. Then there exists an event $\tilde\Omega$ of full probability such that for all $\omega\in\tilde \Omega$, for all $(S,T)\in[0,1]^2_\leq$, $x\in\R^d$, there exists only one solution to \eqref{eq:SDE} on $[S,T]$ with initial condition $x$.
\end{theorem}
The theorem will follow immediately from Theorem \ref{thm:functional-existence} and the following lemma, which is a refinement of the technique illustrated in \cite[Theorem 3.1]{shaposhnikov2016some}.

\begin{lemma}\label{lem:shaposhnikov}
Let $\gamma:[0,1]\to \R^d$ be bounded and measurable, $b\in L^1_t C^{\alpha,\loc}_x$ and consider the ODE \eqref{eq:ODE simple}.
%
%
Suppose 
%
%
that it admits a locally $\beta$-H\"older continuous semiflow $\Phi$ with
\begin{equation}\label{eq:shaposhnikov-condition}
\beta(1+\alpha)>1.
\end{equation}
Then for any $(S,T)\in[0,1]_\leq^2$ and $y\in \R^d$ there exists a unique solution to the ODE on the interval $[S,T]$ with initial condition $y$, given by $\Phi_{S\to \cdot}(y)$.
\end{lemma}

\begin{proof}
Suppose that there exists another solution to the ODE, given by $(z_t)_{t\in[S,T]}$.
Since both $z$ and $\Phi_{S\to\cdot}(y)$ are bounded, we may and will assume $b\in L^1_t C^{\alpha}_x$
and that $\Phi$ is globally $\beta$-H\"older continuous. Define the control $w=w_{b,\alpha,1}$.

Now let us fix $\tau\in [S,T]$ and define the map $f_t:= \Phi_{t\to \tau} (z_t) - \Phi_{S\to \tau}(y)$. If we are able to show that $f$ is constant in time, then $f \equiv f_0=0$, which implies $\Phi_{t\to\tau}(z_t)=\Phi_{S\to\tau}(y)$ and in turn by choosing $t=\tau$ gives $z_\tau=\Phi_{\tau\to\tau}(z_\tau)=\Phi_{S\to\tau}(y)$. In particular, if we above argument holds for any $\tau\in [S,T]$, we reach the conclusion.

It remains to prove that $f$ is constant on $[S,\tau]$. To this end, first observe that for any $S\leq s\leq t\leq\tau$ it holds
\begin{equs}[eq:sha trick]
|f_{s,t}|
& =|\Phi_{t\to \tau}(z_t)-\Phi_{s\to \tau}(z_s)|\\
& =|\Phi_{t\to \tau}(z_t)-\Phi_{t\to \tau}(\Phi_{s\to t}(z_s))|
\lesssim |\Phi_{s\to t}(z_s)-z_t|^\beta.
\end{equs}
Next, by definition of flow it holds
\[ \Phi_{s\to t}(z_s)-z_t=\int_s^t [b_r(\Phi_{s\to r} (z_s))-b_r(z_r)] \dd r \]
which immediately implies $|\Phi_{s\to t}(z_s)-z_t|\lesssim w(s,t)$; we can improve the estimate by recursively inserting it in the above identity:
\begin{align*}
|\Phi_{s\to t}(z_s)-z_t|
& \leq \int_s^t |b_r(\Phi_{s\to r} (z_s))-b_r(z_r)| \dd r\\
& \leq \int_s^t \| b_r\|_{C^\alpha} |\Phi_{s\to r} (z_s))-z_r|^\alpha \dd r
\leq w(s,t)^{1+\alpha}.
\end{align*}
Inserting the above in estimate \eqref{eq:sha trick}, we can conclude that
\[ |f_{s,t}| \lesssim |\Phi_{s\to t}(z_s)-z_t|^\beta \lesssim w(s,t)^{\beta (1+\alpha)}.\]
Since $\beta(1+\alpha)>1$ and $w$ is a control, $f$ must be necessarily constant.
\end{proof}

\begin{remark}\label{rem:weak}
In the functional setting of Definition 	\ref{def:flow-classical}, path-by-path uniqueness implies pathwise uniqueness, which in turn implies uniqueness in law by the Yamada--Watanabe theorem \cite[Proposition~1]{YW}; we refer to \cite{ShaWre2022} for a general overview on the various notions of strong/weak existence and uniqueness.
\end{remark}

\begin{remark}\label{rem:random ic}
The statement of Lemma \ref{lem:shaposhnikov} is given for deterministic initial data $y$ and semiflow $\Phi$, but immediately extends to random ones: if $X_0$ is a $\cF_0$-measurable random variable, then $(\Phi_{0\to t}(X_0)\big)_{t\in[0,1]}$ is clearly the unique adapted solution with initial condition $X_0$.
\end{remark}

\section{Strong well-posedness for distributional drift}\label{sec:distributional}

When $\alpha<0$, the very first question one has to address is the meaning of the equation, more precisely the meaning of the integral in \eqref{eq:SDE}.
We start by some consequences of Lemma \ref{lem:integral-estimates}.
Denote by $\overline{C^\alpha}$ the closure of $C^1$ in $C^\alpha$. Recall that for any $\alpha<\alpha'$ one has $C^{\alpha'}\subset\overline{C^\alpha}$.
\begin{corollary}\label{cor:T}
Assume \eqref{eq:main exponent} and $\alpha<0$, and take $\delta>0$.
Define the linear map $T^{B^H}:L^q_tC^{1+\delta}_x\to L^\infty_\omega C_tC^\delta_x$ by
\begin{equ}
\big(T^{B^H}h\big)_t(x)=\int_0^t h_r(B^H_r+x)\dd r.
\end{equ}
Denote $w=w_{h,\alpha,q}$.
Then, for any $m\in[2,\infty)$, there exists a constant $K=K(m,H,\alpha,q,d,w(0,1))$ such that for all $(s,t)\in[0,1]^2_\leq$ and $x,y\in\R^d$ one has the bound
\begin{equation}\begin{split}
\big\|&\|\big(T^{B^H}h\big)_{s,t}(x)-\big(T^{B^H}h\big)_{s,t}(y)\|_{L^m|\cF_s}\big\|_{L^\infty}
\\
&\qquad\leq K|x-y|w(s,t)^{1/q}|t-s|^{1/q'+(\alpha-1)H}.\label{eq:dist-bound-1}
\end{split}\end{equation}
Moreover, for any $\kappa\in(0,1)$ sufficiently small there exists a constant $K=K(m,H,\alpha,q,d,w(0,1),\kappa)$ such that one has the bound
\begin{equation}
\Bigg\|\sup_{0\leq s<t\leq 1}\frac{\|\big(T^{B^H}h\big)_{s,t}\|_{C^{1-\kappa,2\kappa}_x}}{w(s,t)^{1/q}|t-s|^{1/q'+(\alpha-1)H-\kappa}}\Bigg\|_{L^m}\leq K.\label{eq:dist-bound-3}
\end{equation}
Consequently with $p=\big(1+(\alpha-1)H\big)^{-1}\in(1,2)$,  the mapping $h\mapsto T^{B^H} h$ takes values in
$L^m_\omega C^{(p+\kappa)-\var}_tC^{1-\kappa,2\kappa}_x$ and as such, it
extends continuously to $L^q_t \overline{C^\alpha_x}$. This extension also satisfies the bounds \eqref{eq:dist-bound-1}-\eqref{eq:dist-bound-3}.
\end{corollary}

\begin{proof}
Applying Lemma \ref{lem:integral-estimates} with $t,z\mapsto(x-y)\cdot\int_0^1\nabla h_t(z+\theta x+(1-\theta)y)\dd \theta$ in place of $h$ yields \eqref{eq:dist-bound-1}.
The bound \eqref{eq:dist-bound-3} follows from \eqref{eq:second average bound} and \eqref{eq:dist-bound-1} by Kolmogorov's continuity theorem in the form of Corollary \ref{cor:kolmogorov-4}.
\end{proof}

Corollary \ref{cor:T} motivates introducing some temporary notation. Given \eqref{eq:main exponent}, set $p_{\alpha,H}=\big(\big(1+(\alpha-1)H\big)^{-1}+2\big)/2\in(1,2)$ and for any $h\in L^q_t C^\alpha_x$ we define the event
\begin{equation*}
\Omega_h:=\Big\{\omega\in \Omega: \, T^{B^H}h(\omega)\in C^{p_{\alpha,H}-\var}_tC^{1-\kappa,2\kappa}_x\,\,\forall\kappa>0\Big\}
\end{equation*}
which is therefore of full probability.

The regularity of $T^{B^H}$ obtained from Corollary \ref{cor:T} is sufficient to define a notion of solution via nonlinear Young formalism. For the proof of the next statement we refer to \cite{galeati2021nonlinear}, which can be readily readapted to the $p$-variation framework, see also \cite{anzeletti2021regularisation}.


\begin{lemma}\label{lem:nY-construct}
Let $A:[0,1]\times\R^d\to \R^n$ and $x:[0,1]\to\R^d$ satisfy $A\in C_t^{p-\var}C^{\eta,\loc}_x$ and $x\in C^{\zeta-\var}_t$ such that the exponents $p,\zeta\in[1,\infty)$, $\eta\in (0,1]$ satisfy
\begin{equ}
\frac{1}{p}+\frac{\eta}{\zeta}>1.
\end{equ}
Then the \emph{nonlinear Young integral}
\begin{equ}
y_t=\int_0^t A_{\dd s}(x_s):=\lim_{\ell\to \infty}\sum_{j=0}^{2^\ell-1}A_{j2^{-\ell}t,(j+1)2^{-\ell}t}(x_{j2^{-\ell}t})
\end{equ}
is well-defined. If $A\in C_t^{p-\var}C^{\eta}_x$, then for all $(s,t)\in[0,1]_\leq^2$ $y$ satisfies the bound
\begin{equation}\label{eq:nY-remainder}
|y_{s,t}-A_{s,t}(x_s)|\leq N \llbracket A\rrbracket_{p-\var,C^\eta_x;[s,t]}\llbracket x\rrbracket_{\zeta-\var;[s,t]}^\eta,
\end{equation}
where the constant $N$ depends only on $1/p+\eta/\zeta$.
\end{lemma}

\begin{definition}\label{defn:nYoung-def}
Assume \eqref{eq:main exponent}, $\alpha<0$ and $b\in L^q_t C^\alpha_x$.
Given $\omega\in \Omega_b$, we say that a path $x$ is an $\omega$-path solution to \eqref{eq:SDE} if $x=\varphi+B^H(\omega)$, $\varphi\in C^{\zeta-\var}_t$ for some $\gamma$ satisfying $1/p_{\alpha,H} + 1/\zeta>1$ and the equality
\begin{equ}\label{eq:nYoung def}
\varphi_t=\varphi_0 + \int_0^t \big(T^{B^H}b(\omega)\big)_{\dd s}(\varphi_s)
\end{equ}
holds for all $t\in[0,1]$, the integral being understood in the nonlinear Young sense.
We say that a stochastic process $X$ is a \emph{path-by-path solution} to \eqref{eq:SDE}  if, for $\PP$-a.e. $\omega\in \Omega_b$, $X(\omega)$ is an $\omega$-path solution in the above sense.
Given this formulation of the SDE, the concepts of strong and weak solutions are analogous to the classical ones, see Section \ref{sec:notation} above.
\end{definition}

Typically we encounter more special cases of nonlinear Young integrals than the generality that Lemma \ref{lem:nY-construct} allows. First of all, the spatial growth of $A$ is often quantified (as in e.g. Corollary \ref{cor:T}). Secondly, whenever $\varphi$ is a solution to a nonlinear Young equation, it is automatically  of $p$-variation and its temporal regularity can be often controlled by that of $A$ (see e.g. \cite[Section 3.2]{galeati2021nonlinear} in the H\"older case or Lemma \ref{lem:nY-apriori} in Appendix \ref{app:Young}).

We can then define the notion of flows similarly to Definition \ref{def:flow-classical}. In fact, the following definition extends the previous one: for functional drifts, taking $A=T^\gamma b$, using the Riemann sums characterization of the nonlinear Young integral one can easily verify that
\begin{equation*}
\int_0^t (T^\gamma b)_{\dd s} (\varphi_s) = \int_0^t b_s(\varphi_s+\gamma_s) \dd s \quad \forall\, t\in [0,1].
\end{equation*}
Therefore in the functional case Definitions \ref{def:flow-classical} and \ref{def:flow-nYoung} coincide via the change of variables
\begin{equation}\label{eq:change-variables}
\Psi_{s\to t}(x)=\Phi_{s\to t}(x+\gamma_s)-\gamma_t.
\end{equation}

\begin{definition}\label{def:flow-nYoung}
Assume $A\in C_t^{p-\var} C^{\eta,\loc}_x$ for some $\eta\in(0,1]$, $p\in[1,2)$ satisfying $(1+\eta)/p>1$.
A semiflow associated to the nonlinear Young equation
\begin{equ}\label{eq:nYoung ODE}
y_t = y_0 + \int_0^t A_{\dd s} (y_s) 
\end{equ}
is a jointly measurable map $\Psi:[0,1]^2_\leq\times\R^d\to\R^d$ such that
\begin{itemize}
\item for all $(s,x)\in[0,1]\times\R^d$ one has $\Psi_{s\to \cdot}(x)\in C^{p-\var}_t$ and for all $t\in[s,1]$ one has the equality
\begin{equ}
\Psi_{s\to t}(x)=x+\int_s^t A_{\dd r}\big(\Psi_{s\to r}(x)\big);
\end{equ}
\item for all $(s,r,t,x)\in\times[0,1]^3_\leq\times \R^d$ one has $\Psi_{s\to t}(x)=\Psi_{r\to t}\big(\Psi_{s\to r}(x)\big)$.
\end{itemize}
The definitions of flow, random (semi)flow, adaptedness, and H\"older continuity are then exactly as in Definition \ref{def:flow-classical}.

\end{definition}

We are now in the position to state and prove our existence and uniqueness theorems in the case of distributional drift.
\begin{theorem}\label{thm:dist-existence}
Assume \eqref{eq:main exponent}, $\alpha<0$, and let $b\in L^q_t C^\alpha_x$. Then there exists an adapted random semiflow of solutions to \eqref{eq:SDE} that is furthermore locally $\beta$-H\"older continuous $\PP$-almost surely for all $\beta\in(0,1)$.
\end{theorem}

\begin{proof}
By sacrificing a small regularity, we may and will assume $b\in L^q_t \overline{C^\alpha_x}$.
The proof follows similar steps as that of Theorem \ref{thm:functional-existence}.
We take $m\in[2,\infty)$, to be chosen large enough later as well a sequence of functions $(b^n)_{n\in\N}$ such that $b^n\in L^q_tC^{2}_x$ and $\|b^n\|_{L^q_t C^\alpha_x}\leq \|b\|_{L^q_t C^\alpha_x}$ for all $n\in\N$,
and $\|b^n-b\|_{L^q_t C^{\alpha-1}_x}\to 0$ as $n\to\infty$.
Replacing $b$ by $b^n$ in \eqref{eq:SDE}, the equation clearly admits an adapted random semiflow $\Psi^n_{s\to t}$.
For fixed $(s,t)\in[0,1]^2_\leq$, $x\in\R^d$, and $n,n'\in\N$, by Theorem \ref{thm:stability} one has the bound
\begin{equ}
\big\|\Psi_{s\to t}^n(x)-\Psi_{s\to t}^{n'}(x)\big\|_{L^m}\lesssim \|b^n-b^{n'}\|_{L^q_t C^{\alpha-1}_x}.
\end{equ}
Similarly, for $(s,t)\in[0,1]^2_\leq$, $x,x'\in\R^d$, and $n\in\N$, Theorem \ref{thm:stability} yields
\begin{equ}
\big\|\Psi_{s\to t}^{n}(x)-\Psi_{s\to t}^{n}(x')\big\|_{L^m}\lesssim |x-x'|.
\end{equ}
The temporal regularity is obtained from Lemma \ref{lem:apriori-estim}: in our present notation we get
\begin{equ}
\big\|\Psi_{s\to t}^{n}(x)-\Psi_{s\to t'}^{n}(x)\big\|_{L^m}\lesssim w_{b,\alpha,q}(t,t')^{1/q}|t'-t|^{\alpha H+1/q'}=:\tilde w(t,t')^{1+\alpha H}
\end{equ}
with $\tilde w$ defined by the above equality.
Regularity in the $s$ variable is obtained precisely as in \eqref{eq:s-regularity}.
From these estimates we obtain the convergence
\begin{equ}
\Psi^{n}\to \Psi\qquad\text{in }L^m_\omega C_{s,t}C^{1-\kappa,\loc}_x\cap L^m_\omega C^{\loc}_x C^{p_{\alpha,H}-\var}_{s,t}
\end{equ}
to a limit $\Psi$ just as in Theorem \ref{thm:functional-existence} with all the required properties shown in the same way, except for the fact that $\Psi_{s\to \cdot}(x)$ solves the equation on $[s,1]$ with initial condition $x$ in the nonlinear Young sense.
Since at this point $s$ and $x$ are fixed, we assume for simplicity $s=0, x=0$ and denote $\Psi^{n}_{0\to t}(0)=\psi^{n}_t$, $\Psi_{0\to t}(0)=\psi_t$.
It is sufficient to show the convergence
\begin{equ}
\int_0^t \big(T^{B^H}b^{n}\big)_{\dd s}(\psi^{n}_s)\to \int_0^t\big(T^{B^H}b\big)_{\dd s}(\psi_s)
\end{equ}
in probability for each $t\in[0,1]$. Recall that by Corollary \ref{cor:T} we have that 
\begin{equ}
T^{B^H}(b^{n}-b)\to 0 \qquad\text{in } C^{p_{\alpha,H}-\var}_tC^{1-\kappa,\loc}_x
\end{equ}
in probability. From the above, we have that $\psi^{n}$ converges to $\psi$ (and in particular is bounded) in $C^{p_{\alpha,H}-\var}_t$ in probability. Therefore if we take an auxiliary $\ell\in\N$ and write
\begin{equs}
\int_0^t &\big(T^{B^H}b^{n}\big)_{\dd s}(\psi^{n}_s)- \int_0^t\big(T^{B^H}b\big)_{\dd s}(\psi_s)
\\
&=\int_0^t \big(T^{B^H}b^{\ell}\big)_{\dd s}(\psi^{n}_s)- \int_0^t\big(T^{B^H}b^{\ell}\big)_{\dd s}(\psi_s)
\\
&\qquad-\int_0^t \big(T^{B^H}(b^{\ell}-b^{n})\big)_{\dd s}(\psi^{n}_s)+ \int_0^t \big(T^{B^H}(b^{\ell}-b)\big)_{\dd s}(\psi_s),
\end{equs}
then we can first choose $\ell$ and $n$ large enough to make the third and fourth integrals small, and then we can keep the same $\ell$ and increase $n$ further to make the difference of the first two terms small, using the Lipschitzness of $b^{\ell}$. This concludes the proof.
\end{proof}

\begin{theorem}\label{thm:distributional:PBP-uniqueness}
Assume \eqref{eq:main exponent}, $\alpha<0$, and let $b\in L^q_t C^\alpha_x$. Then there exists an event $\tilde\Omega$ of full probability such that for all $\omega\in\tilde \Omega$, for all $(S,T)\in[0,1]^2_\leq$, $x\in\R^d$, there exists only one $\omega$-path solution to \eqref{eq:SDE} on $[S,T]$ with initial condition $x$; in other words, \emph{path-by-path uniqueness} holds.
\end{theorem}
\begin{remark}
In analogy to Remark \ref{rem:weak}, the strong form of uniqueness coming from Theorem \ref{thm:distributional:PBP-uniqueness} readily implies pathwise uniqueness of solutions defined on random time intervals (e.g. stopping times) as well as uniqueness in law of weak solutions. In fact, it gives us uniqueness in a larger class of possibly non-adapted pathwise solutions, since the nonlinear Young formalism does not require adaptedness of the processes in consideration. 
On the other hand, Theorem \ref{thm:dist-existence} tells us that the unique solution is in fact a strong one.

Notice however that all these considerations only apply in the framework of Definition \ref{defn:nYoung-def}, namely if the SDE is interpreted in a nonlinear Young sense as \eqref{eq:nYoung def}. Differently from the functional one, in the distributional setting there is no  canonical notion of solution, and one can in principle find alternative concepts which fall outside the framework of Definition \ref{defn:nYoung-def} and Theorem \ref{thm:distributional:PBP-uniqueness}; for a practical example, see Definition \ref{defn:weak-solution} further below.
\end{remark}

Theorem \ref{thm:distributional:PBP-uniqueness} follows from a version of Lemma \ref{lem:shaposhnikov} in the nonlinear Young setting, which is a generalization of Theorem 5.1 from \cite{galeati2021nonlinear}.
\begin{lemma}\label{lem:shaposhnikov-young}
Let $A\in C^{p-\var}_t C^{\eta,\loc}_x$ for 
some $\eta\in(0,1]$, $p\in[1,2)$ satisfying $(1+\eta)/p>1$. Suppose that the nonlinear YDE
\begin{equation*}
x_t = \int_0^t A_{\dd s}(x_s)
\end{equation*}
admits a locally $\beta$-H\"older continuous semiflow $\Psi$ with any $\beta\in(0,1)$.
%
%
Then for any $(S,T)\in[0,1]_\leq^2$ and $y\in\R^d$ there exists a unique solution to the nonlinear YDE on $[S,T]$, which is given by $\Psi_{S\to \cdot}(y)$.
\end{lemma}
\begin{proof}
The proof is very similar to that of Lemma \ref{lem:shaposhnikov}, so we will mostly sketch it.
Let $z$ be a solution on $[S,T]$ starting from $y$, which by definition belongs to $C^{q-\var}_t$ with some $q$ such that $1/p+\eta/q>1$.
Thus $z$ is bounded, and in particular after localizing the argument we may assume that $\Psi$ is globally $\beta$-H\"older and that $A\in C^{p-\var}_t C^{\eta}_x$; furthermore, since the inequalities involving $(\eta,p,q)$ are strict, we can assume $\eta\in (0,1)$.

Set $w(s,t):=\llbracket A\rrbracket_{p-\var,C^{\eta}_x;[s,t]}^p$; an application of Lemma \ref{lem:nY-apriori} readily informs us that
\begin{equation}\label{eq:shaposhnikov-young-proof}
|\Psi_{s\to t}(x)-x-A_{s,t}(x)| \lesssim w(s,t)^{\frac{1+\eta}{p}}
\end{equation}
uniformly in $(s,t)\in [0,1]_\leq^2$ and $x\in\R^d$ (the hidden constant can depend on $w(0,1)$); a similar bound also holds for $\Psi_{s\to t}(x)$ replaced by $z_t$.

As before, we fix $\tau\in [S,T]$ and set $f_t:= \Psi_{t\to\tau}(z_t)-\Psi_{S\to \tau}(y)$; in order to conclude, it suffices to show that $f$ is constant.
As in \eqref{eq:sha trick}, we have $|f_{s,t}|\lesssim |\Psi_{s\to t}(z_s)-z_t|^\beta$.
Moreover by definition of solution to the YDE and estimate \eqref{eq:shaposhnikov-young-proof}, it holds that
\begin{align*}
|\Psi_{s\to t}(z_s)-z_t|
= \big|\Psi_{s\to t}(z_s)-z_s - A_{s,t}(z_s) - (z_t-z_s-A_{s,t}(z_s))\big|
\lesssim w(s,t)^{\frac{1+\eta}{p}}.
\end{align*}
%
Combining the two estimates, we get
\begin{align*}
|f_{s,t}|\lesssim w(s,t)^{\frac{\beta(1+\eta)}{p}};
\end{align*}
by assumption, we can choose $\beta$ close enough to $1$ so that $\beta(1+\eta)/p$ is bigger that $1$, implying the conclusion.
\end{proof}

\section{Flow regularity and Malliavin differentiability}\label{sec:flow}

So far we have established the existence of a random H\"older continuous \textit{semiflow} $\Phi_{s\to t}(x)$; the aim of this section is to strengthen this result, by establishing better properties for $\Phi$.
We will start by showing that $\Phi$ is a random \textit{flow}, in the sense that for each fixed $s<t$ the maps $x\mapsto\Phi_{s\to t}(x)$ are invertible, see Theorem \ref{thm:flow-invertibility} below.
The main body of the section is devoted to the proof of Theorem \ref{thm:flow-diffeo}, showing that both $\Phi_{s\to t}$ and its spatial inverse $\Phi_{s\leftarrow t}$ admit continuous derivatives. We conclude the section by showing that the random variables $\Phi_{s\to t}(x)$ possess a rather strong form of Malliavin differentiability, see Theorem \ref{thm:malliavin} below.

From now on, we will use both $\Phi_{s\to t}(x)$ and $\Phi_{s\to t}(x;\omega)$ to denote the semiflow, so to stress the dependence on the fixed element $\omega\in \Omega$ whenever needed; we start with the promised invertibility.

\begin{theorem}\label{thm:flow-invertibility}
Let \eqref{eq:main exponent} hold, $b\in L^q_tC^\alpha_x$, and denote by $\Phi_{s\to t}(x;\omega)$ the semiflow of solutions constructed in Theorems \ref{thm:functional-existence} and \ref{thm:dist-existence}. Then there exists an event $\tilde{\Omega}$ of full probability such that, for all $\omega\in \tilde{\Omega}$ and all $(s,t)\in [0,1]^2_\leq$, the map $x\mapsto \Phi_{s\to t}(x;\omega)$ is a bijection.
\end{theorem}

\begin{proof}
We follow closely the classical arguments by Kunita, cf. \cite[Lemmas II.4.1-II.4.2]{kunita1984stochastic}, as they are completely independent from the driving noise being Brownian.

First, let us define the family of random variables

\begin{equation*}
\eta_{s,t}(x,y) := |\Phi_{s\to t}(x) - \Phi_{s\to t}(y)|^{-1}
\end{equation*}
Set $\gamma=H\wedge 1/q'$ for $\alpha \geq 0$, $\gamma = \alpha H + 1/q'$ in the case $\alpha<0$.
Recall that the estimates in the proof of Theorem \ref{thm:functional-existence}, respectively Theorem \ref{thm:dist-existence}, overall yield
\begin{equation}\label{eq:flow-estim-1}
\| \Phi_{s\to t}(x) - \Phi_{s'\to t'}(y)\|_{L^m} \lesssim |s-s'|^\gamma + |t-t'|^\gamma + |x-y|;
\end{equation}
moreover, by taking expectation in \eqref{eq:inverse-estimate}, we have
\begin{equation}\label{eq:flow-estim-2}
\| |\Phi_{s\to t}(x) - \Phi_{s\to t}(y)|^{-1} \|_{L^m} \lesssim |x-y|^{-1}.
\end{equation}
Fix any $\delta>0$; we can combine estimates \eqref{eq:flow-estim-1} and \eqref{eq:flow-estim-2} and argue as in \cite[Lemma II.4.1]{kunita1984stochastic} to deduce that, for any $s<t$ and any $x$, $x'$, $y$, $y'$ satisfying $|x-y|>\delta$, $|x'-y'|>\delta$, it holds
\begin{equation}\label{eq:flow-estim-3}\begin{split}
\| & \eta_{s,t} (x,y)-\eta_{s',t'}(x',y')\|_{L^m}\\
& \lesssim \delta^{-2}  \Big[ |x-x'|+|y-y'|+ (1+|x|+|x'|+|y|+|y'|)(|t-t'|^\gamma + |s-s'|^\gamma) \Big].
\end{split}\end{equation}
%

From \eqref{eq:flow-estim-3}, one can apply Kolmogorov's continuity theorem to deduce that the map $(s,t,x,y)\mapsto \eta_{s,t}(x,y;\omega)$ is continuous on the domain $\{s<t, |x-y|>\delta\}$ for $\PP$-a.e. $\omega$.
As the argument works for any $\delta>0$, we can find an event $\tilde{\Omega}$ of full probability such that, for all $\omega\in\tilde{\Omega}$, the map $\eta_{s,t}(x,y;\omega)$ is continuous on $\{s<t, |x-y|\neq 0\}$, which implies that it must also be finite for all $s<t, x\neq y$. This clearly implies injectivity of $x\mapsto \Phi_{s,t}(x;\omega)$ for all $s<t$ and $\omega\in\tilde{\Omega}$.

We move to proving surjectivity, which this time is closely based on \cite[II.Lemma 4.2]{kunita1984stochastic}, having established the key inequalities \eqref{eq:flow-estim-1} and \eqref{eq:flow-estim-2}.
Let $\hat{\R}^d=\R^d\cup\{\infty\}$ be the one-point compactification of $\R^d$; set $\hat{x}=x/|x|^2$ for $x\in \R^d\setminus\{0\}$ and $\hat{x}=\infty$ for $x=0$. Define
\begin{equation*}
\tilde \eta_{s,t}(\hat{x}) =
\begin{cases} (1+ |\Phi_{s\to t}(x)|)^{-1}\quad & \text{if } \hat{x}\in \R^d\\
0 & \text{if } \hat{x}=0
\end{cases}
\end{equation*}
Arguing as in \cite[Lemma II.4.2]{kunita1984stochastic} we find
\begin{equation}\label{eq:flow-estim-4}
\| \tilde\eta_{s,t}(\hat{x}) -\tilde\eta_{s',t'}(\hat{y})\|_{L^m} \lesssim |\hat{x}-\hat{y}| + |t-t'|^\gamma + |s-s'|^\gamma;
\end{equation}
by Kolmogorov's theorem, we can find an event of full probability, which we still denote by $\tilde{\Omega}$, such that $\tilde \eta_{s,t}(\hat{x};\omega)$ is continuous at $\hat{x}=0$ and so that $\Phi_{s,t}(\cdot;\omega)$ can be extended to a continuous map from $\hat{\R}^d$ to itself for any $s<t$ and $\omega\in\tilde\Omega$.
This extension, denoted by $\tilde{\Phi}_{s\to t}(x;\omega)$, is continuous in $(s,t,x)$ for every $\omega\in\tilde\Omega$ and thus $\Phi_{s\to t}(\cdot\,; \omega)$ is homotopic to the identity map $\tilde{\Phi}_{s\to s}(\cdot\,;\omega)$, making it surjective.
Its original restriction $\Phi_{s\to t}(\cdot\,; \omega)$ must then be surjective as well, from which we can conclude that $x\mapsto \Phi_{s\to t}(x;\omega)$ is surjective for all $s<t$ and $\omega\in\tilde{\Omega}$.
\end{proof} 

Our next goal is to establish that $\Phi$ is in fact a \textit{random flow of diffeomorphisms}; by this we mean that, in addition to the map $(s,t,x,\omega)\mapsto \Phi_{s\to t}(x;\omega)$ satisfying all the properties listed in Definition \ref{def:flow-classical}, there exists an event of full probability $\tilde{\Omega}$ such that $x\mapsto \Phi_{s\to t}(x;\omega)$ is a diffeomorphism for all $s<t$ and $\omega\in\tilde{\Omega}$. 
We will in fact prove a little bit more:

\begin{theorem}\label{thm:flow-diffeo}
Let \eqref{eq:main exponent} hold, $b\in L^q_tC^\alpha_x$, and $\Phi$ be the associated random flow. Then there exists a constant $\delta(\alpha,H)>0$ and an event $\tilde\Omega$ of full probability such that for any $\omega\in\tilde{\Omega}$ and any $s<t$, the map $x\mapsto \Phi_{s\to t}(x;\omega)$ and its inverse are both $C^{1+\delta,\loc}_x$.
\end{theorem}

In order to prove Theorem \ref{thm:flow-diffeo}, we will first assume $b$ to be sufficiently smooth ($b\in L^q_t C^{1+\kappa}_x$ would suffice), so that the associated $\Phi$ is already known to be a flow of diffeomorphism, and derive estimates which only depend on $\| b\|_{L^q_t C^\alpha_x}$ (cf. Lemma \ref{lem:moments-jacobian} and Proposition \ref{prop:flow-diffeo-estimates} below). Establishing the result rigorously for general $b$ is then accomplished by standard approximation procedures, in the style of Theorems \ref{thm:functional-existence}, \ref{thm:dist-existence}.
We will frequently use the exponent $\eps=(\alpha-1) H+1/q'$ from Lemma \ref{lem:integral-estimates}, recall that \eqref{eq:main exponent} is equivalent to $\eps>0$.

Recall that, for regular $b$, the Jacobian of the flow, namely the matrix $J_{s\to t}^x := \nabla \Phi_{s\to t}(x)\in \R^{d\times d}$, is known to satisfy the \textit{variational equation}
\begin{equation}\label{eq:variational-equation}
J_{s\to t}^x = I + \int_s^t \nabla b_r(\Phi_{s\to r}(x)) J_{s\to r}^x \dd r.
\end{equation}
Already from this fact we can deduce useful moment estimates for $J^x_{s\to t}$. 

\begin{lemma}\label{lem:moments-jacobian}
Assume \eqref{eq:main exponent} and let $b\in L^q_t C^2_x$.
Then there exists $p(\alpha,H)<2$ with the following property: for any $m\in [1,\infty)$, there exists a constant $N=N(m,p,H,\alpha,q,d,\| b\|_{L^q_t C^\alpha_x})$ such that, for all $x\in \R^d$ and $s\in [0,1]$, it holds
\begin{equation}\label{eq:moments-jacobian}
\Big\| \sup_{t\in [s,1]} |J^x_{s\to t}|\, \Big\|_{L^m} + \big\| \llbracket J^x_{s\to \cdot}\rrbracket_{p-\var;[s,1]} \big\|_{L^m} \leq N;
\end{equation}
moreover, for fixed $\delta<\eps$, for any $x\in \R^d$ and $s \leq t \leq t'$ it holds
\begin{equation}\label{eq:moments-jacobian-2}
\| J^x_{s\to t} - J^x_{s\to t'} \|_{L^m} \lesssim |t-t'|^\delta.
\end{equation}
\end{lemma}

\begin{proof}
For fixed $s\in [0,1]$ and $x\in \R^d$, setting $A_{s,t}:= \int_s^t \nabla b_r(\Phi_{s\to r}(x)) \dd r$, equation \eqref{eq:variational-equation} can be regarded as a linear Young differential equation. Arguing as in the proof of Theorem \ref{thm:stability}, one can show that $A$ has finite $p$-variation for some $p<2$ and that in fact there exists $\mu>0$ (depending on the usual parameters and $\| b\|_{L^q_t C^\alpha_x}$, but not on $x$ nor $s$) such that
\begin{equation}\label{eq:jacobian-proof1}
\E\bigg[  \exp\bigg( \mu \bigg| \sup_{s\leq t<t'\leq 1} \frac{|A_{t,t'}|}{w_{b,\alpha,q}(t,t')^{1/q} |t-t'|^\delta} \bigg|^2\bigg)\bigg] <\infty;
\end{equation}
Lemma \ref{lem:young-estimate} in Appendix \ref{app:Young} (with $\tilde{p}=p$) then implies the pathwise estimate
\begin{equation*}
\sup_{t\in [s,1]} |J^x_{s,t}| + \llbracket J^x_{s\to \cdot}\rrbracket_{p-\var;[s,1]} \leq C \exp \big( C \llbracket A \rrbracket_{p-\var; [s,1]}^p \big).
\end{equation*}
Claim \eqref{eq:moments-jacobian} then follows by taking $L^m$-norms on both sides and observing (as in the proof of Theorem \ref{thm:stability}) that \eqref{eq:jacobian-proof1} implies $\E[\exp(\lambda \llbracket A \rrbracket_{p-\var}^p)]<\infty$ for all $\lambda>0$. Similarly, claim \eqref{eq:moments-jacobian-2} also follows from Lemma \ref{lem:young-estimate} (this time applying estimate \eqref{eq:young-estimate-2} therein) combined with \eqref{eq:jacobian-proof1}.
\end{proof}

The next step in the proof of Theorem \ref{thm:flow-diffeo} is given by the following key estimate.

\begin{prop}\label{prop:flow-diffeo-estimates}
Let $b$ be a regular drift, define $J^x_{s\to t}$ as above;
set $\eps=(\alpha-1)H + 1/q'$. Then there exists $\gamma\in (0,1)$ such that, for any $m\in [1,\infty)$, there exists $N=N(m,\gamma,H,\alpha,q,d,\| b\|_{L^q_t C^\alpha_x})$ such that
\begin{equation}\label{eq:flow-diffeo-goal}
\| J^x_{s\to t} - J^y_{s'\to t'}\|_{L^m} \leq N\big[ |x-y|^{\gamma} + |t-t'|^{\eps \gamma} + |s-s'|^{\eps \gamma} \big].
\end{equation}
for all $(s,t), (s',t')\in [0,1]^2_\leq $ and $x,y\in \R^d$.
\end{prop}

The proof requires the following technical refinement of Lemma \ref{lem:integral-estimates}.

\begin{lemma}\label{lem:comparison-integrals}
Assume \eqref{eq:main exponent}, $h\in L^q_t C^1_x$, and let $\varphi^i$, $i=1,2$, be two processes satisfying the assumptions of Lemma \ref{lem:integral-estimates} for the same control $w$; define $\eps$ as therein and set $\psi^i_t=\int_S^t h_r(B^H_r+\varphi^i_r) \dd r$.
Then for $\gamma\in (0,1)$ satisfying
\begin{equation}\label{eq:hypothesis-gamma}
\eps-\gamma H>0, \quad \eps(2-\gamma)-\gamma H>0, \quad \eps(2-\gamma)-\gamma H + (2-\gamma)/q>1,
\end{equation}
and any $m\in [2,\infty)$, there exists $N=N(m,\gamma,H,\alpha,q,d,\| h\|_{L^q_t C^{\alpha-1}_x})$ such that
\begin{equation*}
\| (\psi^1-\psi^2)_{s,t} \|_{L^m} \leq N |t-s|^{\eps-\gamma H} w_{h,\alpha-1,q}(s,t)^{\frac{1}{q}} \big(1+w(s,t)\big) \sup_{r\in [S,T]} \| \varphi^1_r-\varphi^2_r\|_{L^m}^\gamma.
\end{equation*}
\end{lemma}

\begin{remark}
The conditions in \eqref{eq:hypothesis-gamma} should be understood as ``$\gamma$ small enough''. Indeed, note that all three conditions are upper bounds on $\gamma$ and under condition \eqref{eq:main exponent} we can always find $\gamma>0$ satisfying \eqref{eq:hypothesis-gamma}: as $\gamma\downarrow 0$, the three conditions become respectively $\eps>0$, $2\eps>0$, and $2\eps+2/q>1$, all of which are trivial since $q\leq 2$.
\end{remark}

\begin{proof}
The proof is very similar to that of Lemma 	\ref{lem:integral-estimates}, so we will mostly sketch it; the main differences are just the use of Lemma \ref{lem:SSL1} with $n=m$ and some interpolation arguments.

Define $A^i_{s,t} = \E_{s-(t-s)}\int_s^t h_r (B^H_r + \E_{s-(t-s)}\varphi_r) \dd r$, so that $\psi^1-\psi^2$ is the stochastic sewing of $A^1-A^2$.
Arguing similarly as in Lemma \ref{lem:integral-estimates}, we have the estimate
\begin{align*}
\|A_{s,t}\|_{L^m}
& \leq \bigg\| \int_s^t \| P_{|r-s_1|^{2H}} h_r\|_{C^\gamma_x}\, |\E_{s_1} \varphi^1_r-\E_{s_1} \varphi^2_r|^\gamma \dd r \bigg\|_{L^m}\\
& \lesssim |t-s|^{\eps-\gamma H} w_{h,\alpha-1,q}(s,t)^{1/q} \sup_{r\in [S,T]} \| \varphi^1_r-\varphi^2_r\|_{L^m}^\gamma;
\end{align*}
the first condition of Lemma \ref{lem:SSL1} is verified, since $\eps-\gamma H>0$ and $1/q \geq 1/2$.
To control $\E_{s_1} \delta A_{s,u,t}=\E_{s_1} \delta A^1_{s,u,t}-\E_{s_1} \delta A^2_{s,u,t}$, we can decompose it as $\E_{s_1} \delta A_{s,u,t} = I^1-I^2+J^1-J^2$, similarly to Lemma \ref{lem:integral-estimates}. Estimating each one of them separately as therein yields
\begin{align*}
\sup_i \{|I^i|,|J^i|\}\lesssim |t-s|^{2\eps}w_{h,\alpha-1,q}(s,t)^{1/q}w(s_1,t)^{1/q};
\end{align*}
on the other hand, we have
\begin{align*}
\| I^1-I^2\|_{L^m}
& \leq \bigg\| \int_{s_4}^{s_5}\big|P_{|r-s_2|^{2H}}h_r(\E_{s_2}B^H_r+\E_{s_1}\varphi^1_r)-P_{|r-s_2|^{2H}}h_r(\E_{s_2}B^H_r+\E_{s_2}\varphi^1_r)\big|\dd r \\
& \quad -\int_{s_4}^{s_5}\big|P_{|r-s_2|^{2H}}h_r(\E_{s_2}B^H_r+\E_{s_1}\varphi^2_r)-P_{|r-s_2|^{2H}}h_r(\E_{s_2}B^H_r+\E_{s_2}\varphi^2_r)\big|\dd r\bigg\|_{L^m}\\
& \leq \int_{s_4}^{s_5} \| P_{|r-s_2|^{2H}}h_r\|_{C^1_x} \big( \| \E_{s_1} \varphi^1_r-\E_{s_1} \varphi^2_r\|_{L^m} + \| \E_{s_2} \varphi^1_r-\E_{s_2} \varphi^2_r\|_{L^m}\big) \dd r\\
& \lesssim |t-s|^{(\alpha-2)H + 1/q'} w_{h,\alpha-1,q}(s,t)^{1/q} \sup_{r\in [S,T]} \| \varphi_r^1-\varphi^2_r\|_{L^m},
\end{align*}
similarly for $\| J^1-J^2\|_{L^m}$. Interpolating the two bounds together overall yields
\begin{align*}
\| \E_{s_1} \delta A_{s,u,t}\|_{L^m}
\lesssim |t-s|^{\eps(2-\gamma)-\gamma H} w_{h,\alpha-1,q}(s,t)^{1/q} w(s_1,t)^{\frac{1-\gamma}{q}}
\sup_{r\in [S,T]} \| \varphi^1_r-\varphi^2_r\|_{L^m}^\gamma.
\end{align*}
By the hypothesis \eqref{eq:hypothesis-gamma}, the power of $|t-s|$ is positive and the total power of all the controls is greater than $1$.
The conclusion then follows from Lemma \ref{lem:SSL1}.
\end{proof}

\begin{proof}[Proof of Proposition \ref{prop:flow-diffeo-estimates}]
As usual, we can split estimate \eqref{eq:flow-diffeo-goal} into three subestimates, with two of the three parameters $(s,t,x)$ fixed and only one varying. From now on we will fix $\gamma\in (0,1)$ satisfying condition \eqref{eq:hypothesis-gamma}.

\textit{Step 1: $(s,x)$ fixed, $t<t'$.} In this case the desired estimate is just \eqref{eq:moments-jacobian-2} from Lemma \ref{lem:moments-jacobian}, for the choice $\delta=\gamma \eps < \eps$.

\textit{Step 2: $(s,t)$ fixed, $x\neq y$.}  The difference  process $v_t:=J^x_{s,t}-J^{y}_{s,t}$ satisfies an affine Young equation of the form $ \dd v_t = \dd A_t\, v_t + \dd z_t$, $v_s=0$, for
\begin{equation*}
A_t = \int_s^t \nabla b_r(\Phi_{s\to r}(x)) \dd r, \quad z_t = \int_s^t \big[ \nabla b_r(\Phi_{s\to r}(x)) - \nabla b_r(\Phi_{s\to r}(y))\big] J^y_{s\to r} \dd r;
\end{equation*}
invoking as usual Lemma \ref{lem:young-estimate} (for $\tilde{p}=1/2$) and applying estimate \eqref{eq:jacobian-proof1}, one ends up with
\[ \| J^x_{s,t}-J^y_{s,t}\|_{L^m} \lesssim \big\| \llbracket z\rrbracket_{2-\var} \big\|_{L^m}.\]
Observe that $z$ itself can be interpreted as a Young integral: $z_t= \int_s^t \dd \tilde{A}_r J^y_{s\to r}$ for
\[
\tilde{A}_u:=\int_s^u \big[ \nabla b_r(\Phi_{s\to r}(x)) - \nabla b_r(\Phi_{s\to r}(y))\big] \dd r.
\]
Standard properties of Young integral, together with Cauchy's inequality, then yield
\begin{align*}
\big\| \llbracket z\rrbracket_{2-\var} \big\|_{L^m}
\lesssim \big\| \llbracket \tilde{A}\rrbracket_{2-\var}\, \llbracket J^y_{s\to \cdot}\rrbracket_{p-\var} \big\|_{L^m}
\lesssim \big\| \llbracket \tilde{A}\rrbracket_{2-\var}\|_{L^{2m}} \big\| \llbracket J^y_{s\to \cdot}\rrbracket_{p-\var} \big\|_{L^{2m}};
\end{align*}
by estimate \eqref{eq:moments-jacobian}, it only remains to find a bound for $\llbracket \tilde{A}\rrbracket_{2-\var}$.
Recall that by construction $\Phi_{s\to r}(x) = \varphi_{s\to r}(x) + B^H_r$, where the process $\varphi_{s\to \cdot}(x)$ satisfies condition \eqref{eq:drift-regularity} (or even \eqref{eq:apriori-estim} for $\alpha<0$) for $w=w_{b,\alpha,q}$. We can apply Lemma \ref{lem:comparison-integrals} with the choice $h=\nabla b$, $\varphi^1_r=\varphi_{s\to r}(x)$, $\varphi^2_r=\varphi_{s\to r}(y)$ to obtain, for all $s\leq r<u\leq 1$ and all $m\in [1,\infty)$,
\begin{align*}
\| \tilde{A}_{r,u}\|_{L^m}
&\lesssim |r-u|^{\eps-\gamma H} w(r,u)^{1/q} (1+ \| b\|_{L^q_t C^\alpha_x}^q) \sup_{r\in [s,1]} \| \varphi^1_r - \varphi^2_r\|_{L^m}^\gamma\\
& \lesssim |r-u|^{\eps-\gamma H} w(r,u)^{1/q} |x-y|^\gamma
\end{align*}
where in the second inequality we used estimate \eqref{eq:flow-estim-1}. By Lemma \ref{lem:kolmogorov-2} in Appendix \ref{app:kolmogorov} we deduce that, for any $m\in [1,\infty)$ and $\delta<\eps-\gamma H$, it holds
\begin{align*}
\big\| \llbracket \tilde{A}\rrbracket_{2-\var}\|_{L^{2m}} \lesssim \bigg\| \sup_{r<u} \frac{|\tilde A_{r,u}|}{|r-u|^{\delta} w(r,u)^{1/q}}\bigg\|_{L^{2m}} \lesssim |x-y|^\gamma.
\end{align*}
Combining all the above estimates yields the conclusion in this case.

\textit{Step 3: $(t,x)$ fixed, $s<s'$.} This step is mostly a variation on the arguments presented in the previous cases, so we only sketch it. We can write
\[
J^x_{s,t}= J^x_{s,s'} + \int_{s'}^t \nabla b(\Phi_{s\to t}(x))  J^x_{s,r} \dd r
\]
so that the difference $v_t= J^x_{s,t} - J^x_{s',t}$ can be regarded as the solution to an affine Young equation on $[s',t]$, for $A$ and $z$ defined similarly as in Step 2; the only difference is that now $v_{s'} = J^x_{s,s'}-I$ and $z_t = \int_{s'}^t \dd \tilde A_r J^z_{s'\to r}$ for the choice
\begin{equation*}
\tilde{A}_u:=\int_{s'}^u \big[ \nabla b_r(\Phi_{s\to r}(x)) - \nabla b_r(\Phi_{s'\to r}(x))\big] \dd r.
\end{equation*}
From here, the estimates are almost identical to those of Step 2, relying on a combination of Lemmas \ref{lem:young-estimate}, \ref{lem:kolmogorov-2} and \ref{lem:comparison-integrals}; however in this case an application of Step 1 and estimate \eqref{eq:flow-estim-1} gives us
\begin{equation*}
\| J^x_{s'\to s}-I\|_{L^m} \lesssim |s-s'|^{\eps \gamma}, \quad \sup_{r\in [s',1]} \| \Phi_{s\to r}(x)-\Phi_{s'\to r}(x)\|_{L^m}^\gamma \lesssim |s-s'|^{\eps \gamma}. \qedhere
\end{equation*}
\end{proof}

We are now finally ready to complete the

\begin{proof}[Proof of Theorem \ref{thm:flow-diffeo}]
The argument is based on Theorem II.4.4 from \cite{kunita1984stochastic}; assume first $b$ to be a regular field.
It is clear from \eqref{eq:flow-diffeo-goal} that, for any $\delta <\eps \gamma$, the map $(s,t,x)\mapsto \nabla J_{s\to t}^x$ is $\PP$-a.s. locally $\delta$-H\"older continuous, suitable moment estimates depending only on $\| b\|_{L^q_t C^\alpha_x}$.
Furthermore, letting $K_{s\to t}^x$ denote the inverse of $J_{s\to t}^x$ in the sense of matrices, it is well-known that it solves the linear equation
\begin{equ}\label{eq:eq-for-K}
K_{s\to t}^x = I - \int_s^t K_{s\to r}^x\, \nabla b_r(\Phi^x_{s\to r}(x)) 	\dd r;
\end{equ}
arguing as in the proof of Proposition \ref{prop:flow-diffeo-estimates}, one can prove that
\begin{equation*}
\| K^x_{s\to t} - K^y_{s'\to t'}\|_{L^m} \lesssim |x-y|^\gamma + |t-t'|^{\eps \gamma} + |s-s'|^{\eps \gamma}
\end{equation*}
and so that it is $\PP$-a.s. $\delta$-H\"older continuous as well.

In the case of general $b\in L^q_t C^\alpha_x$, we can consider a sequence $b^n$ of regular functions such that $b^n\to b$ in $L^q_t C^\alpha_x$ (up to sacrificing a little bit of spatial regularity as usual), in which case we already know that the associated flows $\Phi^n$ converge to $\Phi$ in $L^m_\omega C_{s,t} C^{\delta,\loc}_x$; combined with the aforementioned moments estimates, one can then upgrade it to convergence in $L^m_\omega C_{s,t} C^{1+\delta,\loc}_x$. In particular, the fields $J^{x,n}_{s\to t}=\nabla \Phi^n_{s\to t}(x)$ and $K^{x,n}_{s\to t}=(\nabla \Phi^n_{s\to t}(x))^{-1}$ converge respectively to $J^x_{s\to t}$ and $K^x_{s\to t}$; by the limiting procedure, there exists an event $\tilde{\Omega}$ of full probability such that, for all $\omega\in\tilde\Omega$, it holds $J^x_{s\to t}(\omega)=\nabla \Phi_{s\to t}(x;\omega)$ and and $J^x_{s\to t}(\omega) K^x_{s\to t}(\omega)=I$ for all $s<t$ and $x\in \R^d$, as well as $J(\omega), K(\omega)\in C_{s,t} C^{\delta,\loc}_x$.

Overall, for every $\omega\in\tilde{\Omega}$, the map $(s,t,x)\mapsto \Phi_{s\to t}(x;\omega)$ has regularity $C_{s,t} C^{1+\delta,\loc}_x$ and its Jacobian admits a continuous inverse $K^x_{s\to t}(\omega)$. But this implies that, for any $s<t$, $\nabla \Phi_{s\to t}(x;\omega)$ is a nondegenerate matrix for all $x\in \R^d$, which by the implicit function theorem readily implies that the inverse of $x\mapsto \Phi_{s\to t}(x;\omega)$ must belong to $C^{1+\delta,\loc}_x$ as well. This concludes the proof.
\end{proof}

It is well known in the regular case that the Jacobian of the flow and the Malliavin derivative satisfy the same type of linear equation.
Therefore, as the last main result of the section, we show Malliavin differentiability of the random variables $X^x_{s\to t}(\omega):= \Phi_{s\to t}(x;\omega)$. To this end, we start with a simple, yet powerful lemma, showing that deterministic perturbations of the driving noise $B^H$ do not affect our solution theory.

\begin{lemma}\label{lem:perturbed-sde}
Assume \eqref{eq:main exponent}, $b\in L^q_t C^\alpha_x$, and $h: [0,1]\to \R^d$ be a deterministic, measurable function; then for any $s\in [0,1]$ and any $x\in \R^d$, there exists a pathwise unique strong solution to the perturbed SDE
\begin{equation}\label{eq:perturbed-sde}
X_t = x + \int_s^t b_r(X_r) \dd r + B^H_{s,t} + h_{s,t} \quad \forall\, t\in [s,1]
\end{equation}
which we denote by $X_{s\to \cdot}(x;h)$; in the distributional case $\alpha<0$, eq. \eqref{eq:perturbed-sde} should be interpreted in the sense of Definition \ref{defn:nYoung-def}.
\end{lemma} 

\begin{proof}
We give two short alternative arguments to verify the claim. On one hand, carefully going through the proofs of Sections \ref{sec:first}-\ref{sec:stability}, the only key properties needed on the process $B^H$ (cf. also Remark \ref{rem:possible-extensions}) are its Gaussianity and the two-sided bounds
\begin{equation*}
\E[ |B^H_t - \E_s B^H_t|^2] \sim  |t-s|^{2H}
\end{equation*}
which are clearly still true for $\tilde{B}^H=B^H+h$, due to $h$ being deterministic.

Alternatively, if we define $\tilde{b}_t(z):= b_r(z+h_r)$, $y=x+h_s$, then any solution $X$ to \eqref{eq:perturbed-sde} must be in a $1$-$1$ correspondence with a solution $Y:=X+h$ to the unperturbed SDE
\[
Y_t = y + \int_s^t \tilde{b}_r(Y_r)\dd r + B^H_{s,t}
\]
and it is clear that $\tilde{b}$ still satisfies condition \eqref{eq:main exponent}, thus implying its well-posedness.
\end{proof}

We can now pass to study Malliavin differentiability of $X^x_{s\to t}$.
To this end, it is convenient to first recall the notion of \emph{$\mathcal{H}$-derivative}. 
Let $\mathcal{H}^H$ denote the Cameron-Martin space associated to $B^H$; we say that a function $F:\Omega\to \R$ is $\mathcal{H}$-continuously differentiable if for $\PP$-a.e. $\omega\in\Omega$, the map $h\mapsto F(\omega+h)$ is Fréchet differentiable from $\mathcal{H}^H$ to $\mathbb{R}$. In particular, this implies the existence of a random bounded linear operator $\partial F(\omega)$, which we call the \emph{$\mathcal{H}$-differential} of $F$, such that $\PP$-a.s.
\begin{align*}
	\partial F(\omega)(h)=\partial_h F(\omega):= \lim_{\eps\to 0} \frac{F(\omega+\eps h)-F(\omega)}{\eps}.
\end{align*}
Denote by $\| \partial F\|$ the (random) operator norm of $\partial F(\omega)$, as a linear operator from $\mathcal{H}^H$ to $\R^d$. It is known (cf. \cite[Section 4.1.3]{nualart2006}) that if $F\in L^2$ and $\| \partial F\|\in L^2$, then $F$ is Malliavin differentiable and its Malliavin differential $DF$ $\PP$-a.s. satisfies $\|D F\|_{\mathcal{H}^H} = \| \partial F\|$.
For this reason, when dealing with $X_{s\to t}^x$, it will be convenient for us to manipulate directly the directional derivatives $\partial_h X^x_{s\to t}$. This notion of derivative allows to consider $h$ coming from a larger class than merely Cameron-Martin paths, see Remark \ref{rem:Malliavin} below for a more detailed explanation.

%
%
%
%

\begin{theorem}\label{thm:malliavin}
Assume \eqref{eq:main exponent} and $b\in L^q_t C^\alpha_x$. In the setting of Lemma \ref{lem:perturbed-sde}, let us set $X^x_{s,t}(h):= X_{s\to t}(x;h)$. Then $\PP$-a.s. the random variables $\partial_h X^x_{s\to t}$ exist for all $h\in C^{2-\var}_t$ and define a (random) linear map $\partial X^x_{s,t}$. Moreover for any $m\in [1,\infty)$ it holds
\begin{equation}\label{eq:generalized-malliavin-moment-estim}
\sup_{s\in [0,1],x\in \R^d} \Big\| \sup_{t\in [s,1]} \| \partial X^x_{s,t}\|_{\mathcal{L}(C^{2-\var};\R^d)} \Big\|_{L^m}<\infty.
\end{equation}
In particular, $X^x_{s\to t}$ is Malliavin differentiable and for any $m\in [1,\infty)$ it holds
%
\begin{equation}\label{eq:malliavin-moment-estim}
\sup_{s\in [0,1], x\in \R^d} \Big\| \sup_{t\in [s,1]} \| D X^x_{s\to t} \|_{\mathcal{H}^H} \Big\|_{L^m} <\infty.
\end{equation}
\end{theorem}

\begin{proof}
For simplicity, we give the proof in the case where $b$ is smooth, so that all the computations are rigorous, but keeping track that the estimate \eqref{eq:malliavin-moment-estim} only depends on $\| b\|_{L^q_t C^\alpha_x}$.
The general case then follows by standard (but a bit tedious) approximation arguments, similar to those of Theorems \ref{thm:functional-existence}-\ref{thm:dist-existence}; for estimate \eqref{eq:malliavin-moment-estim}, one can alternatively invoke \cite[Lemma 1.5.3]{nualart2006}.

For smooth $b$, $\partial_h X^x_{s\to t}$ is classically characterized as the unique solution to the affine equation
\begin{equation}\label{eq:variational-malliavin}
\partial_h X^x_{s\to t} = \int_s^t \nabla b_r(X^x_{s\to t}) \partial_h X^x_{s\to r} \dd r  + h_{s,t}.
\end{equation}
Consider the process $A_t:= \int_s^t \nabla b_r(X^x_{s\to r}) \dd r$ as usual, which satisfies \eqref{eq:jacobian-proof1}, so that it has $\PP$-a.s. finite $p$-variation for some $p<2$ and moreover
\begin{equation}\label{eq:malliavin-proof-1}
\E[\exp(\lambda \llbracket A\rrbracket_{p-\var;[s,1]}^p)]<\infty
\end{equation}
for all $\lambda\in \R$, where the estimate only depends on $\| b\|_{L^q_t C^\alpha_x}$ and does not depend on $x$ or $s$.
Interpreting \eqref{eq:variational-malliavin} as an affine Young equation and applying Lemma \ref{lem:young-estimate} from Appendix \ref{app:Young} with $\tilde{p}=2$, we then find $C>0$ such that
\begin{equation*}
|\partial_h X^x_{s\to t}| \leq C e^{C \llbracket A\rrbracket_{p-\var;[s,t]}^p} \llbracket h\rrbracket_{2-\var;[s,t]};
\end{equation*}
taking first supremum over $h\in C^{2-\var}$ with $\| h\|_{2-\var}=1$ and then over $t\in [s,1]$, we arrive at the pathwise $\PP$-a.s. inequality
\begin{equation*}
\sup_{t\in [s,1]} \| \partial X^x_{s\to t} \|_{\mathcal{L}(C^{2-\var};\R^d)} \leq e^{C \llbracket A\rrbracket_{p-\var;[s,1]}^p}.
\end{equation*}
Taking the $L^m$-norm on both sides, using \eqref{eq:malliavin-proof-1}, then readily yields \eqref{eq:generalized-malliavin-moment-estim}.

Estimate \eqref{eq:malliavin-moment-estim} then follows from the isometric identification of $D X^x_{s,t}$ with $\partial X^x_{s,t}$, so that $\|D F\|_{\mathcal{H}^H}= \| \partial X^x_{s,t}\|$, combined with the functional embedding $\mathcal{H}^H\hookrightarrow C^{2-\var}_t$, see Lemma \ref{lem:embedding-intermediate-space} in Appendix \ref{app:girsanov} for $H\in(0,1/2)$ and recall that $\cH^H\hookrightarrow C^{1-\var}_t$ for $H\geq 1/2$.
\end{proof}

\begin{remark}\label{rem:Malliavin}
Results on differentiability beyond the usual Malliavin sense, in the sense of the existence of $\partial_h X^x_{s,t}$ for $h$ belonging to a larger class than $\mathcal{H}^H$, were already observed for standard SDEs in \cite{kusuoka1993regularity} and have natural explanations in rough path theory, cf. \cite{cass2009,friz2006};
in these works however only $h\in C^{\tilde p-\var}_t$ for some $\tilde p<2$ are allowed.
Here instead, not only are we able to reach $C^{2-\var}_t$, but the result can be further strengthened to allow for some $\tilde p>2$: indeed, the key point is a combination of estimate \eqref{eq:malliavin-proof-1} and Lemma \ref{lem:young-estimate}, which works as long as the condition $1/\tilde{p}>1-1/p$ is satisfied.
\end{remark}
\section{McKean-Vlasov equations}\label{sec:MKV}

Armed with the stability estimate \eqref{eq:stability-estimate}, we can now solve distribution dependent SDEs (henceforth DDSDEs) of the form
\begin{equation}\label{eq:ddsde-general}
X_t = X_0 + \int_0^t F_s(X_s,\mu_s)\dd s + B^H_t, \quad \mu_t=\mathcal{L}(X_t).
\end{equation}
The initial condition $X_0$ is assumed to be $\cF_0$-measurable, in particular, independent of $B^H$.
The idea that estimates of the form \eqref{eq:stability-estimate}, where the difference of two drifts only appears in the weaker norm of $L^q_t C^{\alpha-1}_x$, can be exploited to solve DDSDEs was first introduced in \cite{galeati2021distribution}; the results presented here can be regarded as a natural extension, requiring less time regularity on the drift and allowing to cover $H>1$ as well. In particular, as in the previous sections, we will not need to exploit Girsanov transform, which instead played a prominent role in \cite{galeati2021distribution}.

Since our analysis also includes the case of distributional drifts $F$, we provide a meaningful definition of solution;  observe that in the case $F$ is actually continuous in the space variable (i..e $\alpha>0$), it reduces to the classical one.

\begin{definition}\label{def:solution-DDSDE}
Let $H\in (0,\infty)\setminus \N$ and $F:[0,1]\times \mathcal{P}(\R^d)\to C^\alpha_x$ be a measurable function. We say that a tuple $(\Omega,\bF,\PP; X,B^H)$ is a weak solution to \eqref{eq:ddsde-general} if:
\begin{itemize}
\item[i)] $B^H$ is an $\bF$-fBm of parameter $H$ and $X$ is $\bF$-adapted;
\item[ii)] setting $b^X_t(\cdot):=F_t(\cdot,\mathcal{L}(X_t))$, it holds $b^X\in L^q_t C^\alpha_x$ for some $(q,\alpha)$ satisfying \eqref{eq:main exponent};
\item[iii)] $X$ solves the SDE associated to $b^X$, in the sense of Section \ref{sec:distributional}.
\end{itemize}
\end{definition}

Similarly to Definition \ref{def:solution-DDSDE}, one can immediately extend the concepts of strong existence, pathwise uniqueness and uniqueness in law to the DDSDE \eqref{eq:ddsde-general}.
With a slight abuse, we will use the terminology \textit{input data} of the DDSDE \eqref{eq:ddsde-general} to indicate both the pair $(X_0,B^H)$ (when discussing strong existence and/or pathwise uniqueness of solutions) and the pair $(\xi,\mu^H)=(\mathcal{L}(X_0),\mathcal{L}(B^H))$ (when discussing uniqueness in law).  
We are now ready to formulate our main assumptions on the drift $F$.
\begin{ass}\label{ass:MKV}
Let $H\in (0,\infty)\setminus \N$ fixed, $F:[0,1]\times \cP(\R^d)\to C^\alpha_x$ be a measurable function; we assume that there exist parameters $(\alpha,q)$ satisfying \eqref{eq:main exponent} and $h\in L^q_t$ such that:
\begin{itemize}
\item[i)] for all $t\in [0,1],\, \mu\in \cP(\R^d)$, it holds $\| F_t(\cdot,\mu)\|_{C^\alpha_x} \leq h_t$;
\item[ii)] for all $t\in [0,1],\, \mu,\nu\in \cP(\R^d)$, it holds $\| F_t(\cdot,\mu)-F_t(\cdot,\nu)\|_{C^{\alpha-1}_x} \leq h_t \mathbb{W}_1(\mu,\nu)$;
\end{itemize}
\end{ass}
\begin{remark}
Basic examples of $F$ satisfying Assumption \eqref{ass:MKV} include the following (for their verification, we refer to Section 2.1 from \cite{galeati2021distribution}):
\begin{itemize}
\item[i)] The true McKean--Vlasov case $F_t(\cdot,\mu)=f_t(\cdot)+(g_t\ast \mu)(\cdot)$ for $f,g\in L^q_t C^\alpha_x$;
\item[ii)] Mean-dependence of the form $F_t(\cdot,\mu)=f_t(\cdot\,-\langle \mu\rangle)$, where $\langle \mu\rangle:=\int y\,\mu(\dd y)$;
\item[iii)] The mean $\langle \mu\rangle$ in ii) can be replaced by other functions of statistics (e.g. $\langle \psi,\mu\rangle$ for $\psi\in C^1_x$); one can also take linear combinations of the previous examples.
\end{itemize}
Also, in Assumption \ref{ass:MKV} we only considered the $1$-Wasserstein distance $\mathbb{W}_1$, but in fact all the results below would also hold if we replaced $\mathbb{W}_1$ with $\mathbb{W}_p$ for some $p\in (1,\infty)$.
\end{remark}

\begin{theorem}\label{thm:MKV}
Let $F$ satisfy Assumption \ref{ass:MKV}. Then for any $\cF_0$-measurable $X_0 \in L^1_\omega$ (respectively $\xi\in \mathcal{P}_1(\R^d)$) strong existence, pathwise uniqueness and uniqueness in law of solutions to \eqref{eq:ddsde-general} holds.
\end{theorem}

\begin{proof}
We start by showing strong existence and pathwise uniqueness by means of a contraction argument. Specifically, suppose we are given a filtered probability space $(\Omega,\bF,\PP)$ on which are defined an $\bF$-fBm $B^H$ and an $\mathcal{F}_0$-measurable $X_0\in L^1_\omega$. Consider the space of adapted processes
\[ E:=\Big\{Y:[0,1]\to\R^d:\, Y \text{ is adapted to }\mathcal{F}_t, \sup_{t\in [0,1]} \| Y_t\|_{L^1}<\infty \Big\}\]
which is a complete metric space when endowed with the metric
\[
d_E(Y,Z):=\sup_{t\in [0,1]} e^{-\lambda \int_0^t |h_s|^q \dd s} \| Y_t-Z_t\|_{L^1}
\]
for a parameter $\lambda>0$ to be chosen later.
Define a map $I$ acting on $E$ by letting $I(Y)$ be the unique solution $X$ to the SDE driven by $B^H$, with initial data $X_0$ (c.f. Remark \ref{rem:random ic}) and drift $b^Y_t:=F_t(\cdot\, ,\mathcal{L}(Y_t))$; the map $I$ is well defined thanks to Point i) from Assumption \ref{ass:MKV}, ensuring the solvability of such SDE. Note that $X$ is a solution  to the DDSDE \eqref{eq:ddsde-general} on the space $(\Omega,\bF,\PP)$ with input data $(X_0,B^H)$ if and only if it is a fixed point for $I$.

We claim that $I$ is a contraction on $(E,d_E)$; indeed, given any $Y^1,\,Y^2$, by the stability estimate \eqref{eq:stability-estimate} and Assumption \ref{ass:MKV}, for any $t\in [0,1]$ it holds
\begin{align*}
\| I(Y^1)_t-I(Y^2)_t\|_{L^1}^q
& \lesssim \int_0^t \| F_s(\cdot\,,\mathcal{L}(Y^1_s))-F_s(\cdot\,,\mathcal{L}(Y^2_s))\|_{C^{\alpha-1}}^q \dd s\\
& \lesssim \int_0^t |h_s|^q \,\mathbb{W}_1(\mathcal{L}(Y^1_s),\mathcal{L}(Y^2_s))^q \dd s\\
& \lesssim d_E(Y^1,Y^2)^q \int_0^t |h_s|^q\, e^{q \lambda \int_0^s |h_r|^q \dd r} \dd s\\
& \lesssim (q \lambda)^{-1}\, e^{\lambda q \int_0^t |h_r|^q \dd r}\, d_E(Y^1,Y^2)^q.
\end{align*}
Rearranging the terms, we overall find the estimate
\[
d_E\big(I(Y^1),I(Y^2)\big)^q \leq \frac{C}{q \lambda}\, d_E(Y^1,Y^2)^q,
\]
from which contractivity follows by choosing $\lambda$ appropriately. Pathwise uniqueness then readily follows; as the argument holds for any choice of $\bF$, we can take $\mathcal{F}_t=\sigma\{X_0, B^H_s, s\leq t\}$, yielding strong existence.

To establish uniqueness in law, it suffices to observe that, if $X$ is a weak solution, then we can construct a copy of it on any reference probability space simply by solving therein the SDE associated to $b^X_t(\cdot)= F_t(\cdot,\mathcal{L}(X_t))$: by weak uniqueness for the SDE associated to $b^X$, see Remark \ref{rem:weak}, the solution $\tilde{X}$ constructed in this way must have the same law as the original $X$ and thus be a solution to the DDSDE itself. Given any pair of weak solutions $X^1,X^2$,  possibly defined on different probability spaces, we can then construct a coupling $(\tilde{X}^1,\tilde{X}^2)$ of them on the same probability space, solving the DDSDE for the same input data $(X_0,B^H)$; by the previous argument, it must hold $\tilde{X}^1\equiv \tilde{X}^2$ and so $\mathcal{L}(X^1)=\mathcal{L}(X^2)$.
\end{proof}
\begin{remark}
In fact, going through the same strategy of proof as in \cite{galeati2021distribution} not only allows to establish wellposedness of the DDSDE, but also to establish stability estimates for DDSDEs. Specifically, assume we are given fields $F^i$, $i=1,2$, satisfying Assumption \eqref{ass:MKV} for the same parameters $(\alpha,q)$ and functions $h^i\in L^q_t$ and define the quantity
\[
\| F^1-F^2\|_{\alpha-1,q} :=\bigg( \int_0^1 \sup_{\mu\in \mathcal{P}_1} \big\| F^1_t(\cdot\,,\mu)-F^2_t(\cdot\,,\mu)\big\|_{C^{\alpha-1}_x}^q \dd t\bigg)^{1/q}.
\]
Then for any $m\in [1,\infty)$ there exists a constant $C$, depending on $\alpha,q,H,m,d, \| h^i\|_{L^q}$, such that any two solutions $X^i$ defined on the same space with input data $(X_0^i, B^H)$ satisfy
\begin{equation}\label{eq:stability-ddsde}
\big\| \| X^1-X^2\|_{C^0_t} \big\|_{L^m} \leq C \big(\|X^1_0-X^2_0|\|_{L^m} + \| F^1-F^2\|_{\alpha-1,q}\big);
\end{equation}
in the case of solutions defined on different spaces, using \eqref{eq:stability-ddsde} and coupling argument, we can easily deduce bounds on the Wasserstein distances of their laws. In the true McKean--Vlasov case, namely $F^i_t(\cdot\,,\mu)=f^i_t+g^i_t\ast \mu$ with $f^i,g^i\in L^q_t C^\alpha_x$, it holds
\[
\| F^1-F^2\|_{q,\alpha} \lesssim \| f^1-f^2\|_{L^q_t C^{\alpha-1}_x} + \| g^1-g^2\|_{L^q_t C^{\alpha-1}_x}.
\]
\end{remark}

\section{Weak compactness and weak existence}\label{sec:weak-compactness}

So far we have shown that, under suitable conditions on $b$ (condition \eqref{eq:main exponent}), we have (very) strong existence and uniqueness results. However, as we are now going to show, stochastic sewing also allows to establish weak existence and weak compactness of solutions in the regime \eqref{eq:condition-existence} (defined just before Theorem \ref{thm:weak-existence-intro}), similarly to \cite[Theorem 2.6(i)]{ABLM}, \cite[Theorem 2.8]{anzeletti2021regularisation}. For other applications of sewing techniques and compactness arguments, see also \cite{BecHof2023}.

This section is also our way to say something about the equation in the case $q>2$ that goes beyond the trivial inclusion $L^q_t\subset L^2_t$.

Since here we assume $\alpha<0$, it is a priori not fully clear what it means to be a weak solution to the equation. Contrary to Section \ref{sec:distributional}, where a robust interpretation was accomplished by the nonlinear Young formalism, here we will adopt the following, weaker definition, adapting the notion from \cite{Bass-Chen}. This allows us to prove weak existence more generally, see however Remark \ref{rem:nYoung-BassChen} for a comparison.

\begin{definition}\label{defn:weak-solution}
Let $b\in L^q_t C^\alpha_x$ for some $\alpha<0$. We say that a tuple $(\Omega,{\bF},\PP;X,B^H)$ consisting of a filtered probability space and a pair of continuous processes $(X,B^H)$ is a weak solution to the SDE
\begin{equation}\label{eq:SDE-weak}
X_t = x_0 + \int_0^t b_s(X_s)\dd s + B^H_t
\end{equation}
if $B^H$ is a $\bF$-fBm of parameter $H$, $X$ is $\bF_t$-adapted, and $X_t=x_0+V_t+B^H_t$, where the process $V_t$ has the property that, for any sequence of smooth bounded functions $b^n$ converging to $b$ in $L^q_t C^\alpha_x$, it holds that
\begin{equation*}
\Big\|\int_0^\cdot b^n(s,X_s)\dd s - V_\cdot\Big\|_{C^0_t} \to 0 \quad \text{in probability.}
\end{equation*}
\end{definition}

\begin{theorem}\label{thm:weak-existence}
Let $H\in (0,1)$ and $b\in L^q_t C^\alpha_x$ satisfying \eqref{eq:condition-existence}. Then for any $x_0\in \R^d$ there exists a weak solution to the SDE \eqref{eq:SDE-weak} in the sense of Definition \ref{defn:weak-solution}.
\end{theorem}

\begin{remark}
The above result is only interesting in the regime $H\in (0,1)$ and $q>2$, cf. Remark \ref{rem:other-condition-intro}. Indeed, for $H>1$ condition $\alpha>1/2-1/(2H)$ automatically enforces $\alpha>0$, for which existence follows by classical Peano-type results; instead for $q\leq 2$,  \eqref{eq:condition-existence} implies \eqref{eq:main exponent} and so strong wellposedness follows from the previous sections. 
\end{remark}

First we need the following lemma.

\begin{lemma}\label{lem:singular-integrals}
Let $H\in (0,1)$, $(\alpha,q)$ be parameters satisfying \eqref{eq:condition-existence}; let $X$ be a process defined on a filtered probability space $(\Omega,\bF,\PP)$ of the form $X=\varphi+B^H$, where $B^H$ is an $\bF$-fBm and $\varphi$ satisfies the property \eqref{eq:apriori-estim}.
For any $f\in L^q_t C^\delta_x$, $\delta>0$, let $w_f:=w_{f,\alpha,q}$; then for any $m\in [2,\infty)$ there exists a deterministic constant $K=K(m,d,\alpha,q,H,\| b\|_{L^q_t C^\alpha_x})$, such that
\begin{equation*}
\bigg\| \Big\| \int_s^t f_r(X_r)\dd r\Big\|_{L^m|\mathcal{F}_s} \bigg\|_{L^\infty} \leq K w_f(s,t)^{1/q} |t-s|^{\alpha H + 1/q'}. 
\end{equation*}
As a consequence, for any $\eps>0$ there exists a constant $K=K(\eps,m,d,\alpha,q,H,\| b\|_{L^q_t C^\alpha_x})$ such that
\begin{equation}\label{eq:estim-singular-integrals}
\bigg\| \Big\| \int_0^\cdot f_r(X_r)\dd r \Big\|_{C^{\alpha H + 1/q'-\eps}_t} \bigg\|_{L^m} \leq K \| f\|_{L^q_t C^\alpha_x}.
\end{equation}
By linearity and density, this allows to continuously extend in a unique way the map $f\mapsto \int_0^\cdot f_r(X_r)\dd r$ from $L^q_t \overline{C^\alpha_x}$ to $L^m_\omega C^0_t$.
\end{lemma}

\begin{proof}
We only sketch the proof, since it is very similar to others already presented (cf. Lemma \ref{lem:integral-estimates}).
By Lemma \ref{lem:apriori-estim} and the stochastic sewing (again in the version of \cite[Theorem 2.7]{FHL}), setting $A_{s,t}:=\E_s \int_s^t f_r(\varphi_s+B^H_r)\dd r$ and denoting $\beta=1/q'+\alpha H$, standard computations imply
\begin{align*}
\| A_{s,t}\|_{L^\infty} & \lesssim |t-s|^\beta w_f(s,t)^{1/q},\\
\big\| \|\E_s\delta A_{s,u,t}\|_{L^m| \mathcal{F}_s} \big\|_{L^\infty}
& \lesssim |t-s|^{\beta-H} w_f(s,t)^{1/q} \big\| \| \varphi_{s,u}\|_{L^m|\cF_s}\big\|_{L^\infty}\\
& \lesssim |t-s|^{2\beta-H} w_f(s,t)^{1/q} w_b(s,t)^{1/q}.
\end{align*}
Under condition \eqref{eq:condition-existence}, one can check that the hypotheses of \cite[Theorem 2.7]{FHL} are satisfied, which easily yields all the desired estimates.
\end{proof}

Let us also recall the definition of $\bF$-fBm and the associated Volterra kernel representation \eqref{eq:Volterra} from Section \ref{sec:prelim-fbm}. With these preparations, we can now present the

\begin{proof}[Proof of Theorem \ref{thm:weak-existence}]
As before, we can assume $x_0=0$ without loss of generality.
Let $b\in L^q_t C^\alpha_x$ with $(q,\alpha)$ satisfying \eqref{eq:condition-existence} be given. Since \eqref{eq:condition-existence} is a strict inequality, we can assume without loss of generality that $q<\infty$, $b\in L^q_t\overline {C^\alpha_x}$, and in particular there exists a sequence $\{b^n\}_n\subset L^q_t C^1_x$ such that $b^n\to b$ in $L^q_t C^\alpha_x$ and $\int_s^t \| b^n_r\|_{C^\alpha_x}^q \dd r \leq \int_s^t \| b_r\|_{C^\alpha_x}^q \dd r$ (this can be accomplished by taking $b^n_r=\rho_{1/n}\ast b_r$ for some standard mollifiers $\{\rho_\delta\}_{\delta>0}$, up to replacing $\alpha$ with $\alpha-\eps$).

To each such $b^n$ we can associate a solution $X^n=\varphi^n + B^H$, where by Lemma \ref{lem:apriori-estim} $\varphi^n$ satisfy the bound \eqref{eq:apriori-estim} for $w=w_{\alpha,b,q}$; this implies in particular that $\| \varphi^n_{s,t}\|_m \lesssim |t-s|^{\alpha H + 1/q'}$ uniformly in $n$, which by Kolmogorov's theorem readily implies the tightness of the family $\{\varphi^n\}_n$.
As a consequence, the family $\{(\varphi^n, B^H, W)\}_n$ is tight in $C_t\times C_t\times C_t$.

By Prokhorov's and Skorokhod's theorems, we can construct another probability space $(\tilde{\Omega},\tilde{\mathcal{F}},\tilde{\PP})$ on which there exists a sequence $\{(\tilde{\varphi}^n,\tilde{B}^{H,n}, \tilde{W}^n)\}_n$ such that $(\tilde{\varphi}^n,\tilde{B}^{H,n}, \tilde{W}^n)$ is distributed as $(\varphi^n, B^H, W)$ for each $n$ and $(\tilde{\varphi}^n,\tilde{B}^{H,n}, \tilde{W}^n) \to (\tilde{\varphi},\tilde{B}^{H}, \tilde{W})$ $\tilde{\PP}$-a.s. in $C_t\times C_t\times C_t$. We claim that $\tilde{X}=\tilde{\varphi}+\tilde{B}^H$ is a weak solution to \eqref{eq:SDE-weak}, in the sense of Definition \ref{defn:weak-solution}. For notational simplicity, we drop the tildes for the rest of the proof.

First of all we claim that $B^H$ is still distributed as an fBm of parameter $H$, $W$ as a standard Bm and that the relation $B^H_t = \int_0^t K_H(t,s) \dd W_s$ still holds. The first two statements are an immediate consequence of passing to the limit. For the last one, we can use the fact that for each $n$, the same relation holds between $B^{H,n}$ and $W^n$, the fact that $K_H(t,\cdot)$ is square integrable and standard results on convergence of stochastic integrals (e.g. \cite[Lemma 2.1]{DGHT2011}) to conclude that for any fixed $t$, \eqref{eq:Volterra} holds $\PP$-a.s. The upgrade to a $\PP$-a.s. statement valid for all $t\in [0,1]$ follows from combining this fact with the uniform convergence of $B^{H,n}$ to $B^H$.

Next, since $X^n=\varphi^n+B^{H,n}$ is still a solution to the SDE \eqref{eq:SDE-weak} with regular drift $b^n$, $\varphi^n$ is adapted to $\mathcal{F}^n_t:=\sigma\{ B^{H,n}_s:s\leq t\}=\sigma\{W^n_s: s\leq t\}$; so for any $s<t$, any $t_1,\ldots, t_n \leq s$ and any pair of continuous bounded functions $F,G$ it holds
\begin{align*}
\E\big[F(W^n_{s,t})G(W^n_{t_1}, \varphi^n_{t_1},\ldots,W^n_{t_n},\varphi^n_{t_n})\big]
= \E\big[F(W^n_{s,t})\big]\,\E\big[G(W^n_{t_1}, \varphi^n_{t_1},\ldots,W^n_{t_n},\varphi^n_{t_n})\big].
\end{align*}
Passing to the limit as $n\to\infty$, the same relation holds for $W$ and $\varphi$ in place of $W^n$ and $\varphi^n$, which shows that $W$ is an $\bF$-Bm for $\mathcal{F}_t:=\sigma\{(W_s,\varphi_s):s\leq t\}$; in particular, $B^H$ is an $\bF$-fBm. 
Similarly, since $\varphi^n$ uniformly satisfy the bound \eqref{eq:apriori-estim} w.r.t. $\mathcal{F}^n_t$, it holds
\begin{equs}
\E\big[|\varphi^n_{s,t}|^m \, &G(W^n_{t_1}, \varphi^n_{t_1},\ldots,W^n_{t_n},\varphi^n_{t_n})\big]
\\
&\lesssim \big(w(s,t)^{1/q} |t-s|^{\alpha H + 1/q'}\big)^m \E\big[G(W^n_{t_1}, \varphi^n_{t_1},\ldots,W^n_{t_n},\varphi^n_{t_n})\big].
\end{equs}
Passing to the limit as $n\to\infty$ we conclude that $\varphi$ satisfies \eqref{eq:apriori-estim} w.r.t. the filtration $\mathcal{F}_t$.

Finally, it remains to show that $X$ satisfies the relation $X_t= V_t + B^H_t$ for $V$ satisfying the requirements of Definition \ref{defn:weak-solution}.
First, since $B^H$ is an $\bF$-fBm and $\varphi$ satisfies \eqref{eq:apriori-estim}, Lemma \ref{lem:singular-integrals} applies, so that the process $V_t:=\int_0^t b_r(X_r)\dd r$ is well defined; by this we mean that the map $f\mapsto \int_0^\cdot f_r(X_r)\dd r$ admits a unique extension and $V$ is the limit in $L^m_\omega C^0_t$ of the processes $\int_0^\cdot b^n_t(X_r)\dd r$, for any sequence of smooth $b^n\to b$ in $L^q_T \overline{C^\alpha_x}$. By linearity, we have
\begin{equation}\label{eq:weak-existence-proof-eq1}
\E \bigg[ \Big\| \int_0^\cdot f_r(X_r)\dd r - V_\cdot \Big\|_{C^{\alpha H + 1/q'-\eps}_t}^m \bigg]^{1/m} \lesssim \| f-b\|_{L^q_t C^\alpha_x}
\end{equation}
for any regular $f$; a similar estimate holds for any $X^n$, with $b$ replaced by $b^n$, with the hidden constants being uniform in $n$.
In order to conclude, again thanks to Lemma \ref{lem:singular-integrals}, it suffices to show that $\varphi^n\to V$; for any $f$ as above, it holds
%
\begin{align*}
\E\big[\| \varphi^n-V\|_{C^0_t}\big]
& \leq \E \bigg[\Big\| \int_0^\cdot [b^n-f]_r(X^n_r) \dd r \Big\|_{C^0_t}\bigg] + \E\bigg[\Big\| \int_0^\cdot [f_r(X^n_r)-f_r(X_r)] \dd r \Big\|_{C^0_t}\bigg] 
\\
&\qquad+ \E\bigg[\Big\| \int_0^\cdot f_r(X_r) \dd r - V_\cdot \Big\|_{C^0_t}\bigg]\\
& \lesssim \|b^n-f\|_{L^q_t C^\alpha_x} + \E \bigg[\Big\| \int_0^\cdot [f_r(X^n_r)-f_r(X_r)] \dd r \Big\|_{C^0_t}\bigg] + \|b-f\|_{L^q_t C^\alpha_x}
\end{align*}
where we applied several times estimate \eqref{eq:weak-existence-proof-eq1}.
Since $f$ is regular, $b^n\to b$ and $X^n\to X$, passing to the limit we get
\begin{align*}
\limsup_{n\to\infty} \E\bigg[\Big\| \int_0^\cdot b^n_r(X^n_r) \dd r - V_\cdot \Big\|_{C^0_t}\bigg] \lesssim 2 \|b-f\|_{L^q_t C^\alpha_x};
\end{align*}
by the arbitrariness of $f$, we can conclude that $\varphi^n \to V = \varphi$ and so that $X$ is a weak solution.
\end{proof}

\begin{remark}\label{rem:nYoung-BassChen}
Under Assumption \eqref{eq:main exponent}, the unique strong solution $X$ to the SDE constructed in Section \ref{sec:distributional} satisfies Definition \ref{defn:weak-solution}, as readily seen by applying Lemma \ref{lem:integral-estimates} with $h^n=b^n-b$.
In most situations, pathwise solutions $X$ to \eqref{eq:SDE-weak} in the nonlinear Young sense (cf. Definition \ref{defn:nYoung-def}) which are $\mathbb{F}_t$-adapted are also weak solutions in the sense of Definition \ref{defn:weak-solution}.
Indeed in order to construct such $X$, usually one must have already verified that $T^{B^H}$ extends to a bounded operator from $L^q_t C^\alpha_x$ to $L^m_\omega C^{p-\var}_t C^{\eta,\loc}_x$ (similarly to Corollary \ref{cor:T}) and that $X=\varphi+B^H$ with $\varphi\in C^{\zeta-\var}_t$ $\PP$-a.s., for suitable parameters $(p,\eta,\zeta)$ satisfying $1/p +\eta/\zeta>1$.
Linearity of $T^{B^H}$ and stability of nonlinear Young integration $(A,x)\mapsto \int_0^\cdot A(\dd s,x_s)$ (cf. \cite[Theorem 2.7-4)]{galeati2021nonlinear}) then yields
\begin{align*}
	\Big\| \int_0^\cdot b^n(s,X_s)\dd s-\int_0^\cdot T^{B^H}b(\dd s,\varphi_s) \Big\|_{C^0_t}
	\lesssim \big\| T^{B^H}(b^n-b)\big\|_{C^{p-\var}_t C^{\eta}_{\| \varphi\|_{\infty}}} (1+\| \varphi\|_{C^{\zeta-\var}_t})
\end{align*}
where the r.h.s. converges in probability to $0$ due to the aforementioned mapping properties of $T^{B^H}$ and the assumption $b^n\to b$ in $L^q_t C^\alpha_x$.

The converse implication, namely whether the weak solution constructed in Theorem \ref{thm:weak-existence} is also a pathwise solution in the nonlinear Young sense, might only be true for a more restricted range of parameters.
Let us only sketch the power counting, omitting the arbitrarily small exponents everywhere.
The averaged field $T^{B^H}b$ can be constructed as in Corollary \ref{cor:T}, as an element of $C^{2-\var}_t C^{\alpha+1/(2H),\loc}_x$. Furthermore, we know from Lemma \ref{lem:apriori-estim} that $\varphi\in C^{r-\var}_t$ with $1/r=1+\alpha H$. Therefore if
\begin{equation}\label{eq:nYoung-BassChen}
\frac{1}{2}+\left(\alpha + \frac{1}{2H} \right)(\alpha H + 1) >1,
\end{equation}
then the nonlinear Young integral $\int_0^\cdot (T^{B^H}b)_{\dd t}(\varphi_t)$ is well-defined and agrees with $V$.
Note that the regime \eqref{eq:nYoung-BassChen} is nontrivial in the sense that it allows for drifts for which strong uniqueness is not known, since the right-hand side is strictly greater than $1$ for $\alpha=1-1/(2H)$.
We also remark that \eqref{eq:nYoung-BassChen} is sufficient, but not necessary to define $\int_0^\cdot (T^{B^H}b)_{\dd t}(\varphi_t)$, since for particular choices of $b$ the averaged field $T^{B^H}b$ may enjoy better regularity than $C^{2-\var}_t C^{\alpha+1/(2H),\loc}_x$, see e.g. \cite{anzeletti2021regularisation} for such situations.

For a deeper discussion about equivalence of different solution concepts for distributional drifts, including the nonlinear Young one, Definition \ref{defn:weak-solution} and others, we refer to \cite[Theorem 2.15]{anzeletti2021regularisation} and \cite[Theorem 2.11]{butkovsky2023}.
\end{remark}

\section{$\rho$-irregularity}\label{sec:rho}
The goal of this section is to derive some pathwise properties for solutions of \eqref{eq:SDE}, without appealing to Girsanov transform.
Indeed, in the time-homogeneous setting Girsanov is unavailable for $H>1$,\footnote{In the case $H>1$ and $b\in C^\alpha_x$, $\int_0^\cdot b(B^H_r)\dd r\in C^{1+\alpha}_t$, so that Girsanov would require the condition $1+\alpha>H+1/2$; covering the whole regime $\alpha>1-1/(2H)$ would lead to the condition $1-1/(2H)>H-1/2$, which cannot hold for $H>1$.}
while in the time-dependent case it doesn't apply for any value of $H>0$ (since we can allow drifts which are only $L^q$ in time, for values of $q$ arbitrarily close to $1$).
For more details see Appendix \ref{app:girsanov}.

As a meaningful representative of a larger class of pathwise properties, we will focus on the notion of \textit{$\rho$-irregularity}, first introduced in \cite{CG16} in the context of regularisation by noise for ODEs; it has later found several applications in regularisation for PDEs, see \cite{ChoGub2015, ChoGub2014, ChoGes2019,galeati2020prevalence}, and more recently in the inviscid mixing properties of shear flows \cite{galeati2021mixing}. Let us also mention the recent work \cite{romito2022} for an alternative notion of irregularity, partially related to this one.

\begin{definition}
Let $\gamma\in (0,1)$, $\rho>0$. We say that a function $h\in C([0,1],\R^d)$ is $(\gamma,\rho)$-irregular if there exists a constant $N$ such that
\begin{equ}
\Big|\int_s^te^{i\xi\cdot h_r}\dd r\Big|\leq N\,|\xi|^{-\rho}|t-s|^\gamma \quad \forall \xi\in\R^d,\quad 0\leq s\leq t\leq 1;
\end{equ}
we denote by $\| \Phi^h\|_{\mathcal{W}^{\gamma,\rho}}$ the optimal constant. We say that $h$ is $\rho$-irregular for short if there exists $\gamma>1/2$ such that it is $(\gamma,\rho)$-irregular.
\end{definition}

It was shown in~{\cite{CG16, galeati2020prevalence}} that for any $H\in(0,\infty)\setminus\mathbb{N}$, $B^H$ is $\rho$-irregular for any $\rho<1/(2H)$; we establish the same for a class of perturbations of $B^H$ satisfying the following assumption.

\begin{ass}\label{asn:abstract-condition}
Let $\varphi:[0,1]\to\R^d$ be a continuous adapted process which admits moments of any order; moreover, there exist $\beta > 0$ and a control $w$ such that, for any $m \in [1,\infty)$, there exists a constant $C_m$ such that
\begin{equation}\label{eq:abstract-condition}
\big\|\| \varphi_t -\mathbb{E}_s \varphi_t\|_{L^1|\cF_s}\big\|_{L^m} \leq C_m  w(s,t)^{1/2}|t-s|^{\beta}\quad \forall\, 0\leq s\leq t\leq 1. 
\end{equation}
\end{ass}

\begin{theorem}\label{thm:rho-irr}
Let $H\in (0, +\infty) \setminus \mathbb{N}$ and let $\varphi$ satisfy Assumption \ref{asn:abstract-condition} with $\beta=H$;
then $X:=\varphi+B^H$ is $\PP$-almost surely $\rho$-irregular for any $\rho< 1 /(2H)$. More precisely, for any such $\rho$ and any $m\in [1,\infty)$ there exists $\gamma=\gamma(m,\rho)>1/2$ such that
\begin{equation}\label{eq:rho-irr-moments}
\E[ \| \Phi^X\|_{\mathcal{W}^{\gamma,\rho}}^m ]<\infty.
\end{equation}
\end{theorem}

\begin{remark}
Let us make some observations on Assumption \ref{asn:abstract-condition} and Theorem \ref{thm:rho-irr}:
\begin{itemize}
\item Lemmas \ref{lem:drift-regularity} and \ref{lem:apriori-estim} provide sufficient conditions on $q$ and $\alpha$ that guarantee that solutions of \eqref{eq:SDE} with $b\in L^q_t C^\alpha_x$ satisfy Assumption \ref{asn:abstract-condition}. Note that in some cases we can therefore obtain $\rho$-irregularity of solutions but not uniqueness.
%
\item Our usual toolbox could in principle be also used to study Gaussian moments of $\Phi^X$ (under a somewhat stronger condition than \eqref{eq:abstract-condition}). For simplicity we do not pursue this in detail.
\item 
In terms of exponents, the condition \eqref{eq:abstract-condition} appears to require the same order of ``regularity'', namely $1/2+H$, as Girsanov transform (see Appendix \ref{app:girsanov}).
However, \eqref{eq:abstract-condition} is a significantly weaker condition: instead of controlling the usual increments $\varphi_t-\varphi_s$, one only needs to control the stochastic increments $\varphi_t-\E_s\varphi_t$, which can be much smaller.
%
\item In \cite{CG16,galeati2020prevalence} the \textit{additive perturbation problem} is studied in detail; the authors try to establish, in a deterministic framework, whether a path $h+\varphi$ can be shown to be $\rho$-irregular, given the knowledge that $h$ is so and $\varphi$ enjoys higher H\"older regularity. Such results usually come with a loss of regularity in the exponent $\rho$ at least $1/2$, cf. \cite[Theorem 1.6]{CG16} and \cite[Lemma 78]{galeati2020prevalence}; the use of more probabilistic arguments and stochastic sewing techniques from Theorem \ref{thm:rho-irr} instead allows to cover the whole range $\rho<1/(2H)$ without difficulties.
\end{itemize}
\end{remark}

\begin{proof}
In order to conclude, it suffices to prove the following claim: for any $\rho<1/(2H)$, we can find $\gamma>1/2$ such that for any $m \in [1, \infty)$ it holds
\begin{equation}\label{eq:main-estim}
\Big\| \int_s^t e^{i \xi \cdot X_r} \mathd r \Big\|_{L^m} \lesssim_m |t-s|^{\gamma}  |\xi|^{-\rho}
\quad \forall \, \xi \in \mathbb{R}^d, 0 \leqslant s \leqslant t \leqslant 1. 
\end{equation}
It's clear that in~\eqref{eq:main-estim} we can restrict to $| \xi | \geqslant 1$ (or $| \xi | \geqslant R$) whenever needed, since for small $\xi$ the estimate is trivial.
Once~\eqref{eq:main-estim} is obtained, we can deduce that, for any $\tilde{\rho}<\rho-d/m$, it holds
\begin{equation}\label{eq:auxiliary-estim1}
\mathbb{E} \left[ \int_{\mathbb{R}^d} | \xi |^{\tilde \rho} \bigg| \int_s^t e^{i\xi\cdot X_r} \mathd r \bigg|^m \mathd \xi \right]
= \mathbb{E} \big[\| \mu^X_{s, t} \|_{\mathcal{F} L^{\tilde\rho, m}}^m\big ] \lesssim |t-s|^{\gamma m};
\end{equation}
here we follow the notation from~{\cite{galeati2020prevalence}}, so that $\mu^X_{s,t}$ denotes the occupation measure of $X$ on $[s,t]$ and $\mathcal{F} L^{\rho, m}$ denote Fourier--Lebesgue spaces.
Applying Lemma~57 from~{\cite{galeati2020prevalence}} to~\eqref{eq:auxiliary-estim1}, together with Assumption \ref{asn:abstract-condition}, yields
\begin{align*}
\mathbb{E} \big[\| \mu^X_{s, t} \|_{\mathcal{F} L^{\tilde\rho, \infty}}^m\big] 
& \lesssim \mathbb{E} \big[\| X \|_{C_t}^d  \| \mu^X_{s, t} \|_{\mathcal{F} L^{\tilde \rho,m}}^m\big]\\
& \lesssim \mathbb{E} \big[\| X \|_{C_t}^{2 d} \big]^{1 / 2}\, \mathbb{E} \big[\| \mu^X_{s, t} \|_{\mathcal{F} L^{\tilde\rho, m}}^{2 m}\big]^{1/2}
\lesssim |t-s|^{\gamma m}.
\end{align*}
By the arbitrariness of $m$ and Kolmogorov's continuity criterion, one then deduces that $\mu^X\in C^{\tilde \gamma}_t \mathcal{F}L^{\tilde\rho,\infty}_x$ for any $\tilde{\gamma}<\gamma$ and $\tilde{\rho}<\rho$; but this is equivalent to saying that $X$ is $(\tilde{\gamma},\tilde{\rho})$-irregular, cf. \cite[Section 3.2]{galeati2020prevalence}. The arbitrariness of $\rho<1/(2H)$ readily implies the conclusion as well as the moment estimate \eqref{eq:rho-irr-moments}.

In order to prove the claim \eqref{eq:main-estim}, we will apply Lemma \ref{lem:SSL1}, with
$(S,T)=(0,1)$, and $n=m$.
Fix $\xi \in \mathbb{R}^d$; arguing as in Lemma \ref{lem:SSL2}, it is easy to check that $\int_0^\cdot e^{i\xi\cdot X_r} \dd r$ is the stochastic sewing of
\[ A_{s, t} := \mathbb{E}_{s-(t-s)} \int_s^t e^{i \xi \cdot (\mathbb{E}_{s-(t-s)} \varphi_r + B^H_r)} \mathd r. \]
%
Note that for any $r\in (s,t)$ one has
\begin{equ}
\big|\mathbb{E}_{s-(t-s)}e^{i\xi\cdot B^H_r}\big|=\big|\mathbb{E}_{s-(t-s)}e^{i\xi\cdot (B^H_r-\E_{s-(t-s)}B^H_r)}\big|=e^{-c|\xi|^2|r-s+(t-s))|^{2H}}
\end{equ}
and therefore we have
\begin{equation}\label{eq:auxiliary-estim2}
| A_{s,t} | \lesssim e^{-c |\xi|^2 |t-s|^{2 H}} |t-s| 
\lesssim |\xi|^{-\rho} |t-s|^{1-\rho H}
\end{equation}
where we used the basic inequality $e^{-c|y|^2}\lesssim |y|^{-\rho}$. By the assumption on $\rho$, $\eps_1:=1/2-\rho H>0$, and therefore the condition \eqref{eq:SSL-cond1} is satisfied with $w_1(s,t)=N|\xi|^{-2\rho}(t-s)$.

As for the second condition of Lemma \ref{lem:SSL1}, we have for $(s,u,t)\in\overline{[0,1]}_\leq^3$ that
\begin{align*}
	\| \mathbb{E}_{s_-} \delta A_{s, u, t} \|_{L^m}
	& \leq  \int_u^t \big\| \mathbb{E}_{u-(t-u)}e^{i \xi \cdot B^H_r} (e^{i \xi \cdot \mathbb{E}_{s-(t-s)} \varphi_r} - e^{i \xi \cdot \mathbb{E}_{u-(t-u)} \varphi_r}) \big\|_{L^m}\dd r\\
	&\quad+\int_s^u \big\| \mathbb{E}_{s-(t-s)}e^{i \xi \cdot B^H_r} (e^{i \xi \cdot \mathbb{E}_{s-(t-s)} \varphi_r} - e^{i \xi \cdot \mathbb{E}_{s-(u-s)} \varphi_r}) \big\|_{L^m}\dd r=:I+J.
	\end{align*}
As usual, $I$ and $J$ are treated identically, so we only consider the former. We write
\begin{align*}
I	& = \int_u^t e^{-c | \xi |^2 | r - u + t - u |^{2 H}} \big\| e^{i \xi \cdot \mathbb{E}_{s - (t - s)} \varphi_r} - e^{i \xi \cdot \mathbb{E}_{u-(t-u)} \varphi_r} \big\|_{L^m}\dd r\\
	& \leq e^{- \tilde{c} | \xi |^2 | t - s |^{2 H}} | \xi | \int_u^t \|\mathbb{E}_{s - (t - s)} \varphi_r -\mathbb{E}_{u - (t - u)} \varphi_r \|_{L^m}\dd r\\
 	& \lesssim e^{- \tilde{c} |\xi|^2 |t-s|^{2H}}  |\xi|\, w(s_-,t)^{1/2} |t-s|^{1+H},
\end{align*}
where in the second line we used $(s,u,t)\in\overline{[0,1]}_{\leq}^3$ and in the last one  we used Assumption \ref{asn:abstract-condition}.
Applying again the basic inequality $e^{- \tilde{c}|y|^{2}} \lesssim |y|^{-1-\rho}$, we obtain
\begin{equ}
\| \mathbb{E}_{s_-} \delta A_{s,u,t} \|_{L^m} \lesssim |\xi|^{-\rho}w(s_-,t)^{1/2}|t-s|^{1-H\rho}.
\end{equ}
Therefore, condition \eqref{eq:SSL-cond2} is satisfied with $\eps_2=\eps_1=1/2-\rho H$ and $w_2(s,t)=N|\xi|^{-\rho}w^{1/2}(s,t)(t-s)^{1/2}$ and by \eqref{eq:SSL-conc3} we finally get 
\begin{equ}
\Big\| \int_s^t e^{i \xi \cdot X_r} \mathd r \Big\|_{L^m} \lesssim|\xi|^{-\rho}|t-s|^{1/2+\eps_1}\big(1+w(s,t)\big),
\end{equ}
yielding \eqref{eq:main-estim}.
\end{proof}

\section{Applications to transport and continuity equations}\label{sec:transport}

Having established well-posedness of the characteristic lines $\dd X_t= b_t(X_t)\dd t + \dd B^H_t$, the next natural step is to investigate the associated stochastic transport equation
\begin{equation}\label{eq:stoch-transport}
\partial_t u + b\cdot\nabla u + \dot B^H\cdot \nabla u =0.
\end{equation}
Natural questions in PDE theory and regularization by noise for \eqref{eq:stoch-transport} are its well-posedness, cf. the seminal work \cite{flandoli2010well}, and propagation of the regularity of initial data, first addressed in \cite{fedrizzi2013noise}.
Both features need not be true in the absence of noise; among the vast literature, let us mention: the work \cite{modena2018} where counterexamples to uniqueness are provided even for Sobolev differentiable drifts;
\cite{brue2021positive} where it is shown how uniqueness of the generalized Lagrangian flow (in the sense of DiPerna-Lions \cite{diperna1989ordinary}) does not imply uniqueness of trajectorial solutions to the ODE;
finally \cite{brue2021sharp}, providing sharp examples that DiPerna-Lions flows can at most propagate a ``logarithmic derivative'' of regularity of the initial data $u_0$, but not better. As we will see in Theorem \ref{thm:transport-equation}, the presence of $B^H$ allows to prevent all such pathologies, yielding nontrivial regularisation by noise results even in situations where uniqueness of solutions is already known to hold.

Rather than working directly with equation \eqref{eq:stoch-transport}, following  \cite{flandoli2010well}, it is useful to introduce the transformation $\tilde u_t(x)=u_t(x+B^H_t)$, $\tilde{b}_t(x)=b_t(x+B^H_t)$, which relates it to
\begin{equation}\label{eq:stoch-transport-transformed}
\partial_t \tilde u + \tilde{b}\cdot \nabla \tilde u=0.
\end{equation}
This transformation formally assumes $B^H$ to be differentiable, but the resulting equation \eqref{eq:stoch-transport-transformed} is then well defined (at least for bounded $b$) for any continuous path $B^H$.
More rigorously, we are implicitly assuming that the chain rule applies, which amounts to working with $B^H$ as a geometric rough path, see \cite{catellier2016rough} for the rigorous equivalence between \eqref{eq:stoch-transport}-\eqref{eq:stoch-transport-transformed} in this case. In the Brownian case, this means that the multiplicative noise must be interpreted in the Stratonovich sense, as in \cite{flandoli2010well}.
On the other hand, the resulting PDE \eqref{eq:stoch-transport-transformed} is well defined also for values $H\leq 1/4$, where the rough path formalism no longer applies, and indeed it can be regarded as a \textit{PDE with random drift} $\tilde{b}$, rather than a stochastic PDE.

A nice feature of the regular regime $H>1$, included in our setting, is that here $B^H$ is $\PP$-a.s. differentiable and so \eqref{eq:stoch-transport} is perfectly well defined and the above transformation is completely rigorous (as soon as $(u_t)_{t\in[0,1]}$ is bounded in some function space) and does not involve any ``choice'' of the rough lift.
The above considerations motivate the following definition; from now on we will use both notations $\tilde{u}_t(x)$ and $\tilde{u}_t(x;\omega)$ to denote $u_t(\omega, x+B^H_t(\omega))$, in order to stress the fixed realization $\omega\in \Omega$ whenever needed; similarly for $\tilde{b}_t(x)$ and $\tilde{b}_t(x;\omega)$.

\begin{definition}\label{defn:solution-transport}
For a fixed $\omega\in \Omega$, we say that $v$ is a weak solution to the PDE \eqref{eq:stoch-transport-transformed} associated to $\tilde{b}_t(x;\omega)$ if $v\in L^1_t W^{1,1,\loc}_{x}$, $\tilde{b}\cdot\nabla v\in L^1_t L^{1,\loc}_x$ and for any smooth, compactly supported function $\varphi:[0,1]\times\R^d\to \R$ and any $t\in [0,1]$ it holds
\begin{equation}\label{eq:defn-solution-transport}
\langle\varphi_t,v_t\rangle-\langle \varphi_0,v_0 \rangle =\int_0^t [\langle \partial_t\varphi_s ,v_s\rangle + \langle \varphi_s, \tilde{b}_s(\cdot\,;\omega)\cdot\nabla v_s \rangle] \dd s.
\end{equation}
We say that a stochastic process $u$ is a pathwise solution to the stochastic transport equation \eqref{eq:stoch-transport} if for $\PP$-a.e. $\omega\in\Omega$, the corresponding $\tilde{u}_t(x;\omega)$ is a weak solution to \eqref{eq:stoch-transport-transformed} associated to $\tilde{b}_t(x;\omega)$, in the above sense. Finally, a pathwise solution is said to be strong if it is adapted to the filtration generated by $B^H$.
\end{definition}

Similarly to equations \eqref{eq:stoch-transport}-\eqref{eq:stoch-transport-transformed}, we can relate the stochastic continuity equation
\begin{equation}\label{eq:stoch-continuity}
\partial_t \mu + \nabla\cdot (b\, \mu) + \dot B^H\cdot \nabla \mu =0
\end{equation}
to its random PDE counterpart
\begin{equ}\label{eq:stoch-continuity-transformed}
\partial_t \tilde\mu + \nabla \cdot (\tilde b\, \tilde\mu)=0
\end{equ}
by means of the transformation $\tilde{\mu}_t(x;\omega)=\mu_t(\omega,x+B^H_t(\omega))$.
In the next definition, $\mathcal{M}_+=\mathcal{M}_+(\R^d)$ denotes the set of nonnegative finite Radon measures. For $\mu\in\mathcal{M}_+$ we write $\mu\in L^p_x$ to mean that $\mu$ admits an $L^p$-integrable density w.r.t. the Lebesgue measure, in which case with a slight abuse we will identify $\mu(\dd x)=\mu(x) \dd x$.

\begin{definition}\label{defn:solution-continuity}
For a fixed $\omega\in \Omega$, we say that $\rho$ is a weak solution to the PDE \eqref{eq:stoch-continuity-transformed} associated to $\tilde{b}_t(x;\omega)$ if $\rho_t\in \mathcal{M}_+$ for Lebesgue-a.e. $t$,
\[\int_0^1\int_{\R^d} |\tilde{b}_t(x;\omega)| \rho_t(\dd x)<\infty\]
and for any smooth, compactly supported $\varphi:[0,1]\times\R^d\to \R$ and any $t\in [0,1]$ it holds
\[
\langle\varphi_t,\rho_t\rangle-\langle \varphi_0,\rho_0 \rangle =\int_0^t \langle \partial_t\varphi_s + b_s(\cdot\,;\omega)\cdot\nabla \varphi ,\rho_s\rangle \dd s.
\]
We say that a stochastic process $\mu$ is a pathwise solution to the stochastic continuity equation \eqref{eq:stoch-continuity} if for $\PP$-a.e. $\omega\in\Omega$, the corresponding $\tilde{\mu}_t(x;\omega)$ is a weak solution to \eqref{eq:stoch-continuity-transformed} associated to $\tilde{b}_t(x;\omega)$, in the above sense. Finally, a pathwise solution is said to be strong if it is adapted to the filtration generated by $B^H$.
\end{definition}

As it is clear from Definitions \ref{defn:solution-transport}-\ref{defn:solution-continuity}, in order to treat equations \eqref{eq:stoch-transport-transformed}-\eqref{eq:stoch-continuity-transformed} in an analytically weak sense, we need $\tilde{b}$ to enjoy some local integrability and thus to be a well defined measurable function (up to equivalence class). Therefore in the case of coefficients $b\in L^q_t C^\alpha_x$ with $\alpha<0$, throughout this section we will additionally impose that
\begin{equation}\label{eq:transport-condition}
b\in L^r_t L^r_x + L^r_t L^\infty_x \quad \text{for some } r>1;
\end{equation}
we denote by $r'$ the conjugate exponent, i.e. $1/r'+1/r=1$. In the case $\alpha>0$, we will use the convention $r'=1$; in this case under \eqref{eq:main exponent} condition \eqref{eq:transport-condition} is immediately satisfied for $r=q$.
Let us mention that, in the distributional case $\alpha<0$, other approaches for giving meaning \eqref{eq:stoch-transport-transformed}-\eqref{eq:stoch-continuity-transformed} are possible, see Remark \ref{rem:transport-nonlinear-young} below, so it is not obvious whether an assumption of the form \eqref{eq:transport-condition} is needed; still, we will adopt it as it allows us to apply nice analytical tools, while already covering a sufficiently rich class of drifts. 

\begin{remark}\label{rem:transport}
Let us collect a few useful observations:
\begin{itemize}
\item[i)] By standard arguments, whenever a weak solution $v$ to \eqref{eq:stoch-transport-transformed} exists (in the sense of Definition \ref{defn:solution-transport}), then (up to redefining it on a Lebesgue negligible set of $t\in [0,1]$) $t\mapsto v_t$ is continuous w.r.t. suitable weak topologies; in particular it always makes sense to talk about initial/terminal conditions for such equations. The same considerations apply for pathwise solutions, as well as solutions to the continuity equations \eqref{eq:stoch-continuity}-\eqref{eq:stoch-continuity-transformed}; from now on we will always work with these weakly continuous in time versions, without specifying it.
\item[ii)] If $\rho$ is a weak solution to \eqref{eq:stoch-continuity-transformed}, then its mass $\rho_t(\R^d)$ is preserved by the dynamics. In particular, if $\rho\in L^q_t L^p_x$, then it actually belongs to $L^q_t L^{\tilde p}_x$ for all $\tilde{p}\in [1,p]$.
\item[iii)] In Definition \ref{defn:solution-transport} we enforce identity \eqref{eq:defn-solution-transport} to hold for all $\varphi$ smooth and compactly supported, but by standard density arguments it is clear that as soon as more information on $v$ (resp. $u$) and $b$ is available, then \eqref{eq:defn-solution-transport} can be extended to a larger class of $\varphi$, as long as all the terms appearing are well defined. For instance if $v\in L^\infty_t W^{1,p}_x$ and $b\in L^\infty_t L^\infty_x$, then it suffices to know that $\varphi, \partial_t \varphi\in L^1_t L^{p'}_x$, $p'$ being the conjugate of $p$.
\item[iv)] Definitions \ref{defn:solution-transport}-\ref{defn:solution-continuity} and the above observations extend easily to the case of backward equations on $[0,T]$ with terminal conditions $u_T$, $\mu_T$, rather than forward ones with initial $u_0$, $\mu_0$.
\end{itemize}
\end{remark}

The next statement summarizes the main result of this section.

\begin{theorem}\label{thm:transport-equation}
Let $b$ satisfy Assumption \eqref{eq:main exponent} and additionally \eqref{eq:transport-condition} if $\alpha<0$. Then:
\begin{itemize}
\item[i)] For any $p\in [r',\infty)$ and $u_0\in W^{1,p}_x$, there exists a strong pathwise solution $u$ to \eqref{eq:stoch-transport}, which belongs to $L^m_\omega L^\infty_t W^{1,p}_x$ for all $m\in [1,\infty)$.

If moreover $p>r'$, then path-by-path uniqueness holds in the class $L^\infty_t W^{1,p}_x$, in the following sense: there exists an event $\tilde{\Omega}$ of full probability such that, for all $\omega\in\tilde \Omega$ and all $v_0\in W^{1,p}_x$, there can exist at most one weak solution $v \in L^\infty_t W^{1,p}_x$ to the PDE \eqref{eq:stoch-transport-transformed} associated to $\tilde{b}_t(x;\omega)$ and with initial condition $v_0$.
\item[ii)] For any $p\in [r',\infty)$ and any positive measure $\mu_0\in L^p_x$, there exists a strong pathwise solution $\mu$ to \eqref{eq:stoch-continuity}, which belongs to $L^m_\omega L^\infty_t L^p_x$ for all $m\in [1,\infty)$.

Moreover path-by-path uniqueness holds in the class $L^\infty_t L^p_x$, in the following sense:
there exists an event $\tilde{\Omega}$ of full probability such that, for all $\omega\in\tilde \Omega$ and all $\mu_0\in L^p_x$, there can exist at most one weak solution $\rho \in L^\infty_t L^p_x$ to the PDE \eqref{eq:stoch-continuity-transformed} associated to $\tilde{b}_t(x;\omega)$ and with initial condition $\mu_0$.
\end{itemize}
\end{theorem}

Theorem \ref{thm:transport-equation} will be proved by mostly analytical techniques, once they are combined with the information coming from the previous sections.
We will first establish existence of pathwise solutions to equations \eqref{eq:stoch-transport}-\eqref{eq:stoch-continuity} satisfying the desired a priori bounds, see Proposition \ref{prop:transport-existence}.

Uniqueness will be established by two different methods.
In the transport case, we will first establish a priori bounds for solutions the dual equation (backward continuity equation) in Proposition \ref{prop:backward-continuity} and then perform a duality argument (Lemma \ref{lem:duality-transport}); see \cite{diperna1989ordinary} and \cite{beck2019stochastic} for significant precursors in this direction.

For the continuity equation we will instead infer uniqueness from Ambrosio's superposition principle (cf. Theorem \ref{thm:ambrosio}) combined with our path-by-path uniqueness results (Theorems \ref{thm:functional:PBP-uniqueness}-\ref{thm:distributional:PBP-uniqueness}).
To the best of our knowledge, it is the first time these two results are combined in this way to infer path-by-path uniqueness for \eqref{eq:stoch-continuity}; let us mention however that in \cite [Section 4]{beck2019stochastic} the opposite idea is developed, proving path-by-path uniqueness for the SDE starting from the corresponding results for \eqref{eq:stoch-continuity}.

Before giving the proofs, let us recall a few notations and basic facts. We will use $\Psi$ to denote the random flow of diffeomorphisms associated to the (random) ODE $\dot \varphi = \tilde{b}_t(\varphi)$, where we recall the fundamental relation $X_t=\varphi_t+B^H_t$ as well as \eqref{eq:change-variables}.
Similarly to Section \ref{sec:flow}, we will use the notations $J^x_{s\to t} := \nabla \Psi_{s\to t}(x)$, $K^x_{s\to t} := (J^x_{s\to t})^{-1} = \nabla \Psi_{s \leftarrow t}(\Psi_{s\to t}(x))$; we also set $j_{s\to t}(x):=\det J^x_{s\to t}$, similarly for $j_{s\leftarrow t}(x)$.
Recall that, in the case of regular $b$, we have the relations
\begin{equation}\label{eq:determinant-formula}
j_{s\to t}(x) = \exp\Big(\int_s^t {\rm div} b_r (\Phi_{s\to r}(x)) \dd r\Big), \ \
j_{s\leftarrow t}(x) = \exp\Big(-\int_s^t {\rm div} b_r (\Phi_{r\leftarrow t}(x)) \dd r\Big).
\end{equation}

\begin{prop}\label{prop:transport-existence}
Let $b$ satisfy Assumption \eqref{eq:main exponent} and additionally \eqref{eq:transport-condition} if $\alpha<0$. Then:
\begin{itemize}
\item[i)] For any $p\in [r',\infty)$ and $u_0\in W^{1,p}_x$, there exists a strong pathwise solution $u$ to \eqref{eq:stoch-transport}, which belongs to $L^m_\omega L^\infty_t W^{1,p}_x$ for all $m\in [1,\infty)$.
\item[ii)] For any $p\in [r',\infty)$ and any positive measure $\mu_0$ such that $\mu_0\in L^p_x$, there exists a strong pathwise solution $\mu$ to \eqref{eq:stoch-continuity}, which belongs to $L^m_\omega L^\infty_t L^p_x$ for all $m\in [1,\infty)$.
\end{itemize}
\end{prop}

\begin{proof}
Let us first assume $b$ to be smooth and derive estimates which only depend on $\| b\|_{L^q_t C^\alpha_x}$.
In this case, the unique solution to \eqref{eq:stoch-transport-transformed} is given by $\tilde u_t(x)= u_0(\Psi_{0\leftarrow t}(x))$.
Let us give the bound on $\|\nabla \tilde u\|_{L^p}$, the one for $\| \tilde u\|_{L^p}$ being similar; also observe that these quantities coincide with the corresponding ones for $u$. It holds
\begin{align*}
\sup_{t\in [0,1]} \| \nabla \tilde u_t\|_{L^p}^p
& = \sup_{t\in [0,1]} \int_{\R^d} |\nabla \tilde u_t(x)|^p \dd x\\
& \leq \sup_{t\in [0,1]} \int_{\R^d} |\nabla u_0(\Psi_{0\leftarrow t}(x))|^p |\nabla \Psi_{0\leftarrow t}(x)|^p \dd x\\
& = \sup_{t\in [0,1]} \int_{\R^d} |\nabla u_0(y)|^p |\nabla \Psi_{0\leftarrow t}(\Psi_{0\to t}(y))|^p j_{0\to t}(y) \dd y\\
& \leq \int_{\R^d} |\nabla u_0(y)|^p \sup_{t\in [0,1]} |K_{0\to t}(y))|^p \, \sup_{t\in [0,1]} j_{0\to t}(y)\, \dd y.
\end{align*}
Taking the $L^m_\omega$-norm on both sides, we arrive at
\begin{align*}
\Big \| \sup_{t\in [0,1]} \| \nabla \tilde u_t\|_{L^p}^p \Big\|_{L^m}
& \leq \int_{\R^d} |\nabla u_0(y)|^p \Big\| \sup_{t\in [0,1]} |K_{0\to t}(y))|^p \, \sup_{t\in [0,1]} j_{0\to t}(y) \Big\|_{L^m} \, \dd y\\
& \leq \| \nabla u_0\|_{L^p}^p\, \sup_{y\in \R^d} \Big\| \sup_{t\in [0,1]} |K_{0\to t}(y))|^p \Big\|_{L^{2m}}^{1/2} \, \Big\| \sup_{t\in [0,1]} j_{0\to t}(y) \Big\|_{L^{2m}}^{1/2}.
\end{align*}
The finiteness of arbitrary moments of $\sup_{t\in [0,1]} j_{0\to t}(y)$ comes from identity \eqref{eq:determinant-formula}, combined with Lemma \ref{lem:integral-estimates} applied to $h={\rm div} b$ and $\varphi_r=\Phi_{0\to r}(y)-B^H_r$. This estimate is clearly uniform in $y\in\R^d$.
The similar bounds for $K$ follow as in Section \ref{sec:flow}, using the fact that $K$ solves the linear Young equation \eqref{eq:eq-for-K}. Up to relabelling $m=m' p$, we have thus shown that
\begin{equation}\label{eq:existence-transport-proof1}
\| \nabla u\|_{L^m_\omega L^\infty_t L^p_x} \lesssim \| \nabla u_0\|_{L^p_x}.
\end{equation}
We now pass to the case of $\mu$; for regular $b$, solutions are given by the identity
\begin{equation*}
\tilde \mu_t(x) = \mu_0(\Psi_{0\leftarrow t}(x)) \exp\Big(-\int_0^t {\rm div} b_r (\Phi_{r \leftarrow t}(x)) \dd r \Big).
\end{equation*}
Arguing similarly to above, it holds
\begin{align*}
\Big\| \sup_{t\in [0,1]} \| \tilde \mu_t\|_{L^p_x}^p \Big\|_{L^m}
& = \Big\| \sup_{t\in [0,1]} \int_{\R^d} |\mu_0(\Psi_{0\leftarrow t}(x))|^p \exp\Big(-p\int_0^t {\rm div} b_r (\Phi_{r \leftarrow t}(x)) \dd r \Big) \dd x \Big\|_{L^m} \\
& = \Big\| \sup_{t\in [0,1]} \int_{\R^d} |\mu_0(y)|^p \exp\Big((1-p) \int_0^t {\rm div} b_r (\Phi_{0\to r}(y)) \dd r \Big) \dd y \Big\|_{L^m}\\
& \leq \sup_{y\in \R^d} \Big\| \sup_{t\in [0,1]} \exp\Big((1-p) \int_0^t {\rm div} b_r (\Phi_{0\to r}(y))\dd r \Big) \Big\|_{L^m} \int_{\R^d} |\mu_0(y)|^p \dd y,
\end{align*}
and so invoking again Lemma \ref{lem:integral-estimates} and relabelling  $m$ we arrive at
\begin{equation}\label{eq:existence-transport-proof2}
\| \tilde \mu\|_{L^m_\omega L^\infty_t L^p_x} \lesssim \| \mu_0\|_{L^p_x}.
\end{equation}

Having established the uniform estimates \eqref{eq:existence-transport-proof1}-\eqref{eq:existence-transport-proof2}, both existence claims for general $b$ now follow from a standard compactness argument, see for instance  \cite{pardoux1975equations} or \cite[Theorem 15]{flandoli2010well}, so we will only sketch it quickly.

Consider smooth approximations $b^n\to b$, $u_0^n\to u_0$ and denote by $u^n$ the associated solutions; by reflexivity of $L^p_t L^p_\omega W^{1,p}_x$, we can extract a (not relabelled) subsequence such that $u^n\rightharpoonup u$ weakly in $L^p_t L^p_\omega L^p_x$. By properties of weak convergence, the limit $u$ still belongs to $L^m_\omega L^\infty_t W^{1,p}_x$ and is progressively measurable, since the sequence $u^n$ was so; also observe that, as in Remark \ref{rem:transport}-i), we can assume $u$ to be weakly continuous in time, so that it is in fact adapted.
By the linear structure of the PDE, one can then finally verify that $u$ is indeed a pathwise solution. Let us stress that here is where for $\alpha<0$ the assumption \eqref{eq:transport-condition} is crucial, since otherwise it is unclear whether $b^n\cdot\nabla u^n$ converges to $b\cdot\nabla u$ in a weak sense (both w.r.t. $L^m_\omega$ and by testing against $\varphi\in C^\infty_c$); indeed since $p\geq r'$, all objects are well defined in $L^m_\omega L^1_t L^{1,{\rm loc}}_x$ and the claim follows from $b^n\to b$ and $u^n\rightharpoonup u$.
The case of $\mu$ can be treated similarly; the only difference is that, since $b\in L^r_t L^r_x + L^r_t L^\infty_x$ and $\mu\in L^m_\omega L^\infty_t (L^{r'}_x\cap L^1_x)$ by Remark \ref{rem:transport}, the additional $\PP$-a.s. integrability constraint $\langle |\tilde b(\omega)|,\tilde\mu(\omega)\rangle<\infty$ coming from Definition \ref{defn:solution-continuity} is also satisfied.
\end{proof}

We now turn to establishing existence of sufficiently regular solutions to the continuity equation with well chosen terminal data; handling the backward nature of the equation yields slightly worsened estimates compared to those of Proposition \ref{prop:transport-existence}.

\begin{prop}\label{prop:backward-continuity}
Let $T\in [0,1]$ and $\mu_T\in L^p$  compactly supported.
Then there exists a pathwise solution $\mu$ to \eqref{eq:stoch-continuity} on $[0,T]$ with terminal condition $\mu\vert_{t=T}=\mu_T$; moreover for any $m\in [1,\infty)$ and any $\tilde p<p$ it holds
$\mu \in L^\infty_t L^m_\omega L^{\tilde p}_x$.
\end{prop}

\begin{proof}
We can assume ${\rm supp} \mu_T \subset B_R$ for some $R\geq 1$. We will assume $b$ to be regular and show how to derive suitable a priori estimates; the general case then follows by arguing similarly to Proposition \ref{prop:transport-existence}. The solution is given explicitly by
\begin{equation*}
\mu_t(x) = \mu_T(\Psi_{t\to T}(x)) \exp\Big( \int_t^T {\rm div} b_r (\Psi_{t\to r} (x)) \dd r\Big).
\end{equation*}
For any fixed $t\in [0,T]$, it holds
\begin{align*}
\int_{\R^d} |\mu_t(x)|^{\tilde p} \dd x
& = \int_{\R^d} |\mu_T(\Psi_{t\to T}(x))|^{\tilde p} \exp\Big( \tilde p\int_t^T {\rm div} b_r (\Psi_{t\to r} (x) \dd r)\Big) \dd x\\
& = \int_{\R^d} |\mu_T(y)|^{\tilde p} \exp\Big( (\tilde p-1) \int_t^T {\rm div} b_r (\Psi_{r\leftarrow T} (y) \dd r)\Big) \dd y\\
& \leq \| \mu_T \|_{L^p_x}^{\tilde p} \bigg( \int_{B_R} \exp\Big( \frac{ p(\tilde p-1)}{p-\tilde p}  \int_t^T {\rm div} b_r (\Psi_{r\leftarrow T} (y) \dd r)\Big) \dd y \bigg)^{1-\frac{\tilde p}{p}}
\end{align*}
where in the last passage we used first ${\rm supp} \mu_T\subset B_R$ and then H\"older's inequality.
Applying again the change of variable $x=\psi_{t\leftarrow T}(y)$ and the formula for $j_{t\to T}(x)$, overall we find a costant $\kappa=\kappa(p,\tilde p)$ such that
\begin{align*}
\big\| \| \mu_t\|_{L^{\tilde p}_x} \big\|_{L^m}
\leq \| \mu_T\|_{L^p_x}^{\tilde p}\,
\bigg\| \int_{\Psi_{t\to T}(B_R)} \exp\Big( \kappa \int_t^T {\rm div} b_r (\Psi_{t\to r} (y) \dd r)\Big) \dd y \bigg\|_{L^m}^{1-\frac{\tilde p}{p}}.
\end{align*}
It remains to estimate the last quantity appearing on the r.h.s. above. To this end, let us set $N_y := j_{t\to T}(y)^\kappa$; as usual by Lemma \ref{lem:integral-estimates} it holds $\| N_y\|_{L^m}\lesssim 1$, with an estimate uniform in $y$, $t$ and $T$ and only depending on $\| b\|_{L^q_t C^\alpha_x}$.

Thanks to estimates \eqref{eq:flow-estim-1} and Lemma \ref{lem:kolmogorov-3}, one can show that for any $\tilde m\in [1,\infty)$ and $\lambda>1$, uniformly in $t\in [0,T]$ it holds
\begin{equation*}
\big\| \| \Psi_{t\to T}\|_{C^{0,\lambda}} \big\|_{L^{\tilde m}} <\infty\quad \text{ where } \quad \| \Psi_{t\to T}\|_{C^{0,\lambda}}:= \sup_{|x|\geq 1} |x|^{-\lambda} |\Psi_{t\to T}(x)|;
\end{equation*}
this is because one can first show finiteness of the associated $C^{\eta,\lambda'}_x$-norm by Lemma \ref{lem:kolmogorov-3}, and then deduce from it that $\Psi_{t\to T}$ also belongs to $C^{0,\lambda}_x$ for $\lambda=\lambda'+\eta$ (such an embedding readily follows from the definitions of such spaces).

Therefore it holds
\begin{align*}
\Big\| \int_{\Psi_{t\to T}(B_R)} N_y \dd y \Big\|_{L^m}
& \leq \sum_{n\in \N} \Big\| \chi_{ \| \Psi_{t\to T}\|_{C^{0,\lambda}} \in [n,n+1)}  \int_{\Psi_{t\to T}(B_R)} N_y\, \dd y \Big\|_{L^m}\\
& \leq \sum_{n\in \N} \Big\|   \int_{B_{(n+1)R^\lambda}} \chi_{ \| \Psi_{t\to T}\|_{C^{0,\lambda}}\geq n} N_y\, \dd y \Big\|_{L^m}\\
& \leq \sum_{n\in \N} \int_{B_{(n+1)R^\lambda}} \| \chi_{ \| \Psi_{t\to T}\|_{C^{0,\lambda}}\geq n}\|_{L^{2m}} \|N_y\|_{L^{2m}}\, \dd y\\
& \lesssim \sum_{n\in \N} (n+1)^d R^{\lambda d}\, \PP (\| \Psi_{t\to T}\|_{C^{0,\lambda}}\geq n)^{\frac{1}{2m}}\\
& \lesssim R^{\lambda d} \sum_{n\in \N} n^{d-\frac{\tilde m}{2m}} \big\| \| \Psi_{t\to T}\|_{C^{0,\lambda}}\big\|_{L^{\tilde m}}^{\frac{\tilde m}{2m}}
\end{align*}
where in the last passage we used Markov's inequality. Choosing $\tilde{m}$ large enough, so to make the series convergent, then yields the conclusion.
\end{proof}

The importance of integrability of solutions to the backward continuity equation comes from the following (deterministic) duality lemma.

\begin{lemma}\label{lem:duality-transport}
Let $b$ satisfy \eqref{eq:transport-condition} and let $v$, $\rho$ be analytic weak solutions to respectively the forward transport and backward continuity equations associated to $\tilde{b}_t(\cdot;\omega)$; assume that $v\in L^\infty_t W^{1,p_1}_x$ and $\rho\in L^{r'}_t (L^1_x\cap L^{p_2}_x)$ for some $p_1,\,p_2$ satisfying
\[
p_1,\,p_2\in [1,\infty),\quad \frac{1}{p_1}+\frac{1}{p_2}+\frac{1}{r}=1.
\]
Then it holds
\begin{equation*}
\langle v_T, \rho_T\rangle = \langle v_0, \rho_0\rangle.
\end{equation*}
\end{lemma}

\begin{proof}
The argument is relatively standard in the analytic community and is based on the use of mollifiers and commutators, see the seminal work \cite{diperna1989ordinary}.
Let $v^\eps=v\ast g^\eps$ for some standard mollifiers $g^\eps$; since $v^\eps$ is spatially smooth, we can test it against $\rho$ (cf. Remark \ref{rem:transport}-iii)), which combined with the respective PDEs yields the relation
\begin{equation*}
\langle v^\eps_T , \rho_T\rangle  - \langle v^\eps_0 , \rho_0\rangle = \int_0^T \langle (\tilde b\cdot \nabla v)^\eps - \tilde b\cdot\nabla v^\eps, \rho\rangle \dd s.
\end{equation*}
In order to conclude, it then suffices to show that the r.h.s. converges to $0$ as $\eps\to 0$. Recall that by assumption $b= b^1+b^2$ with $b^1\in L^r_t L^r_x$, $b^2\in L^r_t L^\infty_x$, so that the same holds for $\tilde{b}$; we show how to deal with $\tilde b^1$, the other case being similar.
By our assumptions, H\"older's inequality and properties of mollifiers, it is easy to check that both $(\tilde b^1\cdot\nabla v)^\eps$ and $\tilde b^1\cdot\nabla v^\eps$ converge to $\tilde b^1\cdot\nabla v$ in $L^r_t L^{\tilde r}_x$, where $\tilde{r}\in (1,\infty)$ is defined by $1/\tilde{r}=1/r+1/p_1$. But then
\begin{align*}
\bigg| \int_0^T \langle (\tilde b^1_t\cdot \nabla v_t)^\eps - \tilde b^1_t\cdot\nabla v^\eps_t, \rho_t\rangle \dd t \bigg|
& \leq \int_0^T \| (\tilde b^1_t\cdot \nabla v_t)^\eps - \tilde b^1_t\cdot\nabla v^\eps_t\|_{L^{\tilde r}_x}\, \| \rho_t\|_{L^{p_2}_x} \dd t\\
& \leq \| (\tilde b^1\cdot \nabla v)^\eps - \tilde b^1\cdot\nabla v^\eps\|_{L^r_t L^{\tilde r}_x} \| \rho\|_{L^{r'}_t L^{p_2}_x}
\end{align*}
where the last term converges to $0$.
\end{proof}

As a final ingredient, we give the aforementioned Ambrosio's superposition principle; we stress that the statement is deterministic, but we will apply it for fixed realizations of the random drift $\tilde{b}(\cdot;\omega)$.
Although the full statement is a bit technical, we invite the reader to consult the (more heuristical) Theorem 3.1 from \cite{ambrosio2008transport} to understand the role it plays in our analysis.

\begin{theorem}[Theorem 3.2 from \cite{ambrosio2008transport}]\label{thm:ambrosio}
Let $\mu$ be a weak solution to the continuity equation $\partial_t \mu + \nabla\cdot (\mu f)=0$ such that $\mu_t\in \mathcal{M}_+(\R^d)$ for all $t$ and
\[
\int_0^1 \int_{\R^d} |f_t(x)|\, \mu_t(\dd x)\, \dd t<\infty.
\]
Then $\mu$ is a superposition solution, namely there exists a measure $\eta\in \mathcal{M}_+(\R^d \times C_t)$, concentrated on the pairs $(x,\varphi)$ satisfying the relation
\[
\varphi_t = x + \int_0^t f_s(\varphi_s)\dd s,
\]
such that $\mu_t = (e_t)_\sharp \eta$ for all $t\in [0,1]$, where $e_t(x,\varphi)=\varphi_t$ is the evaluation map and $(e_t)_\sharp \eta$ denote the pushforward measure.
\end{theorem}

We are now ready to give the
\begin{proof}[Proof of Theorem \ref{thm:transport-equation}]
Both existence statements come from Proposition \ref{prop:transport-existence}, so we only need to check path-by-path uniqueness.

Let us start with the continuity equation. We claim that the event $\tilde{\Omega}$ of full probability on which path-by-path uniqueness for \eqref{eq:stoch-continuity} holds is the one for which we have uniqueness of solutions to the ODE $\dot{\varphi}_t=\tilde{b}_t(\varphi_t;\omega)$ for all $x\in \R^d$; its existence is granted by Theorems \ref{thm:functional:PBP-uniqueness}-\ref{thm:distributional:PBP-uniqueness}, which additionally imply that $\varphi_t=\Psi_{0\to t}(x;\omega)$.
Indeed, suppose we are given any weak solution $\rho\in L^\infty_t L^p_x$ to \eqref{eq:stoch-continuity-transformed}; by our assumptions, and possibly Remark \ref{rem:transport}-ii), it holds $\int_0^1 \int_{\R^d} |\tilde{b}_t(x;\omega)| \mu_t(\dd x) \dd t<\infty$.
We can then apply Theorem \ref{thm:ambrosio} to deduce that $\rho$ is a superposition solution; since uniqueness of solutions to $\dot{\varphi}_t=\tilde{b}_t(\varphi_t;\omega)$ holds, we readily deduce that $\rho_t = \Psi_{0\to t}(\cdot;\omega)_\sharp \rho_0$, which gives uniqueness.

We now pass to consider the transport case; by linearity, we only need to find an event $\tilde \Omega$ on which any weak solution $v\in L^\infty_t W^{1,p}_x$ to \eqref{eq:stoch-transport-transformed} with $v_0=0$ is necessarily the trivial one. By Remark \ref{rem:transport}-i), we know that any solution is weakly continuous in time, thus it suffices to verify that $v_t=0$ for all $t$ in a dense subset of $[0,1]$.
To this end, let us fix a countable collection $\{f^n\}_n$ of compactly supported smooth functions which are dense in $C^\infty_x$ and a countable dense set $\Gamma\subset [0,1]$.
By Proposition \ref{prop:backward-continuity}, for any $f^n$ and $\tau\in\Gamma$, we can find a pathwise solution $\mu^{\tau,n}$ to the backward continuity equation on $[0,\tau]$ which $\PP$-a.s. belongs to $L^q_t L^q_x$ for all $q\in [1,\infty)$.
Since everything is countable, we can then find an event $\tilde{\Omega}\subset \Omega$ on which $\mu^{\tau,n}(\omega)$ are all defined at once and have the above regularity; we claim that this is the desired event where uniqueness of weak solutions to \eqref{eq:stoch-transport-transformed} in $L^\infty_t W^{1,p}_x$ holds.
Indeed, since $q$ is arbitrarily large and $p>r'$, we can apply Lemma \ref{lem:duality-transport} and use the fact that $v_0=0$ to deduce that
\begin{equation*}
0 = \langle v_0, \mu^{\tau,n}(\cdot\,;\omega)\rangle = \langle v_\tau, f^n\rangle \quad \forall\, \tau\in \Gamma,\, f^n;
\end{equation*}
by density of $f^n$, it follows that $v_{\tau}=0$ for all $\tau\in \Gamma$, which by density of $\Gamma$ and continuity finally implies $v\equiv 0$.
\end{proof}

\begin{remark}\label{rem:transport-nonlinear-young}
In \cite[Section 5.2]{galeati2021noiseless}, the authors show how to solve the transport equation \eqref{eq:stoch-transport} in a pathwise manner under the assumption that $T^{B^H}b \in C^\gamma_t C^2_x$ for some $\gamma>1/2$; in this case, one can treat purely distributional drifts $b$, without enforcing \eqref{eq:transport-condition}. However, this assumption is satisfied under more restrictive conditions than \eqref{eq:main exponent}, e.g. if $b\in L^\infty_t C^\alpha_x$ for some $\alpha>2-1/(2H)$. We believe that existence and uniqueness for \eqref{eq:stoch-transport} (resp. \eqref{eq:stoch-continuity}) should hold under \eqref{eq:main exponent} even when $\alpha<0$, without the need for \eqref{eq:transport-condition}, but we leave this problem for future investigations.
\end{remark}

\paragraph{Acknowledgments.}
MG thanks Konstantinos Dareiotis for valuable discussions during the development of the parallel article \cite{multi}.
The authors thank the institutions MFO Oberwolfach and TU Wien for their hospitality during their research visits.

\paragraph{Conflicts of Interest Statement.} There are no conflicts of interest.

\paragraph{Financial support.} MG was funded by the Austrian Science Fund (FWF) Stand-Alone programme P 34992. LG was funded by the DFG under Germanys Excellence Strategy - GZ 2047/1, project-id 390685813 and later by the SNSF Grant 182565 and by the Swiss State Secretariat for Education, Research and Innovation (SERI) under contract number MB22.00034 through the project TENSE.

\appendix
\section{Kolmogorov continuity type criteria}\label{app:kolmogorov}

Let us recall (a conditional version of) the classical Azuma--Hoeffding inequality.

\begin{lemma}\label{lem:azuma-hoeffding-conditional}
Let $k\in\N$ and $\{Y_i\}_{i=0}^k$ be a sequence of $\R^d$-valued martingale differences with respect to some filtration $\{\mathcal{F}_i\}_{i=0}^k$, with $Y_0=0$; assume that there exist deterministic constants $\{\delta_i\}_{i=1}^k$ such that $\PP$-a.s. $|Y_i|\leq\delta_i$ for all $i$.
Then for
\begin{equation*}
S_j:=\sum_{i=1}^j Y_i,\qquad\Lambda:=\delta_1^2+\cdots+\delta_k^2,
\end{equation*}
one has the $\PP$-a.s. inequality
\begin{equation}\label{eq:azuma-hoeffding-conditional}
\E\bigg[ \exp\Big(\frac{|S_k|^2}{4 d \Lambda}\Big)\bigg\vert \mathcal{F}_0\bigg]\leq 3.
\end{equation}
\end{lemma}

\begin{proof}
The proof goes along the same lines as standard Azuma--Hoeffding; since we haven't found a direct reference in the literature, we present it here.

First, observe that we can reduce ourselves to the case $d=1$ by reasoning componentwise, the general one following from a simple application of conditional Jensen's inequality.

Next, we claim that the following version of Hoeffding's lemma holds: given a random variable $X$ and a filtration $\mathcal{F}$ such that $\E[X\vert \mathcal{F}]=0$ and $a\leq X \leq b$ $\PP$-a.s., it holds
\begin{equation}\label{eq:hoeffding-lem}
\E[\exp(\lambda X)\vert \mathcal{F}] \leq \exp\bigg( \frac{\lambda^2 (b-a)^2}{8}\bigg)\quad \forall\, \lambda\in\R.
\end{equation}
By homogeneity, it suffices to prove \eqref{eq:hoeffding-lem} for $b-a=1$; in this case, we have the basic inequality $e^{\lambda x} \leq (b-x)e^{\lambda a} + (x-a)e^{\lambda b}$ for all $x\in [a,b]$. Evaluating in $X$ and taking conditional expectation we obtain
\begin{equation*}
\E [e^{\lambda X}\vert \mathcal{F}]\leq (a+1)e^{\lambda a} - a e^{\lambda (a+1)} = e^{H(\lambda)}, \quad H(\lambda):=\lambda a + \log (1+a - e^\lambda a).
\end{equation*}
It can be readily checked that $H(0)=H'(0)=0$ and $H''(\lambda)\leq 1/4$, which by Taylor expansion yields $H(\lambda)\leq \lambda^2/8$ and thus \eqref{eq:hoeffding-lem}.

Next, given the sequence $\{Y_k\}_k$ as in the hypothesis, we can assume by homogeneity $\Lambda=1$ and apply recursively Hoeffding's lemma as follows:
\begin{align*}
\E[ \exp(\lambda S_k)\vert \mathcal{F}_0]
& = \E\big[ \exp(\lambda S_{k-1})\, \E[\exp(\lambda Y_k) \vert \mathcal{F}_{k-1}] \big\vert \mathcal{F}_0\big]\\
& \leq \exp\big( \lambda^2 (2 \delta_k)^2/8\big) \E[ \exp(\lambda S_{k-1})\vert \mathcal{F}_0]
\leq \ldots \leq e^{\lambda^2/2}.
\end{align*}
By the inequality $e^{|x|}\leq e^x+e^{-x}$ and Chernoff's conditional bound, we have
\begin{align*}
\PP(|S_k|>a\vert \mathcal{F}_0) \leq \inf_{\lambda >0}e^{-\lambda a}\, \E[ e^{\lambda |S_k|}] \leq 2 \inf_{\lambda>0} e^{-\lambda a + \lambda^2/2} = 2 e^{-a^2/2}.
\end{align*}
Therefore we arrive at
\begin{align*}
\E\bigg[ \exp\Big( \frac{|S_k|^2}{4}\Big)\bigg\vert \mathcal{F}_0\bigg]
= \int_0^{+\infty} \PP\bigg(|S_k|> \sqrt{4|\log s|}\bigg)\, \dd s \leq 1 + 2\int_1^{+\infty} s^{-2} \dd s = 3.
\end{align*}
\end{proof}

Next, we pass to a conditional Kolmogorov-type lemma, stated in a way which is suitable for our purposes.

\begin{lemma}\label{lem:kolmogorov}
Let $E$ be a Banach space, $X:[0,T]\to E$ be a continuous random process; suppose there exist $\alpha,\,\beta\in (0,1]$, a control $w:[0,T]^2\to [0,\infty)$, a constant $K>0$ and a $\sigma$-algebra $\mathcal{F}$ such that
\begin{equation}\label{eq:kolmogorov-hypothesis}
\E\bigg[\exp\bigg(\frac{\| X_{s,t} \|_E^2}{|t-s|^{2\alpha} \, w(s,t)^{2\beta}}\bigg)\bigg\vert \mathcal{F}\bigg] \leq K \quad \forall\, s<t.
\end{equation}
Then for any $\eps>0$ there exists a constant $\mu=\mu(\eps)>0$ such that
\begin{equation}\label{eq:kolmogorov-conclusion}
\E\bigg[\exp\bigg( \mu\, \sup_{s<t} \frac{\| X_{s,t}\|_E^2}{|t-s|^{2(\alpha-\eps)} \, w(s,t)^{2\beta}}\bigg)\bigg \vert \mathcal F\bigg] \leq e \, K.
\end{equation}
\end{lemma}

\begin{proof}
Since we are already assuming $X$ to be continuous, the supremum over $s<t$ appearing in \eqref{eq:kolmogorov-conclusion} equals the supremum over $s,\, t$ taken over dyadic points.
Up to rescaling, we may assume wlog $T=1$.

For any $n\in \N$ and $k\in \{0,\ldots, 2^n\}$, set $t^n_k= k 2^{-n}$ and define a random variable
\[
J=\sum_{n=1}^\infty 2^{-2n} \sum_{k=0}^{2^n-1} \exp\bigg( \frac{\| X_{t^n_k,t^n_{k+1}}\|_E^2}{2^{-2n\alpha} w(t^n_k, t^n_{k+1})^{2\beta}}\bigg);
\]
by \eqref{eq:kolmogorov-hypothesis}, it holds $\E[J\vert \mathcal{F}]\leq K$. Now take $s,t$ to be dyadic points satisfying $|t-s|\sim 2^{-m}$, then by standard chaining arguments (see e.g. the proof of \cite[Theorem 3.1]{Friz-Hairer}) it holds
\[
\| X_{s,t}\|_E \lesssim \sum_{n\geq m} \sup_k \| X_{t^n_k,t^n_{k+1}} \|_E;
\]
on the other hand, by the definition of $J$, it holds
\[
\| X_{t^n_k,t^n_{k+1}} \|_E
\leq 2^{-n\alpha} w(t^n_k, t^n_{k+1})^\beta \sqrt{\log(2^{2n} J)}
\lesssim_\eps 2^{-n(\alpha-\eps)} w(s,t)^\beta (1+\sqrt{\log J})
\]
so that
\begin{align*}
\| X_{s,t} \|_E
& \lesssim \sum_{n\geq m} 2^{-n(\alpha-\eps)} w(s,t)^\beta (1+\sqrt{\log J})\\
& \lesssim 2^{-m(\alpha-\eps)} w(s,t)^\beta (1+\sqrt{\log J})
\sim |t-s|^{\alpha-\eps} w(s,t)^\beta (1+\sqrt{\log J}).
\end{align*}
Overall, we deduce the existence of a constant $C=C(\eps)>0$ such that
\begin{equation}\label{eq:kolmogorov-proof}
\sup_{s<t} \frac{\| X_{s,t}\|_E}{|t-s|^{\alpha-\eps} w(s,t)^\beta} \leq C (1+\sqrt{\log J}).
\end{equation}
The conclusion now readily follows by applying $x\mapsto \exp(\mu x^2)$ on both sides of \eqref{eq:kolmogorov-proof} and choosing $\mu=\mu(\eps)$ so that $2\mu C^2(\eps) =1$, so that
\begin{equation*}
\E\Big[\exp\big(\mu C^2(1+\sqrt{\log J}\big)^2\Big\vert \mathcal{F}\Big]
\leq \E\Big[\exp\big(2\mu C^2(1+\log J)\big)\Big\vert \mathcal{F}\Big]
= e\, \E[J\vert \mathcal{F}] \leq e K.
\end{equation*}
\end{proof}

Going through an almost identical argument, one can also obtain the following result, whose proof is therefore omitted.

\begin{lemma}\label{lem:kolmogorov-2}
Let $E$ be a Banach space, $X:[0,T]\to E$ be a continuous random process; suppose there exist $\alpha,\,\beta\in (0,1]$, $m\in (1,\infty)$, a control $w:[0,T]^2\to [0,\infty)$, a constant $K>0$ and a $\sigma$-algebra $\mathcal{F}$ such that
\begin{equation}\label{eq:kolmogorov-hypothesis-2}
\E\big[\,\| X_{s,t}\|_E^m\big \vert \mathcal{F}]^{1/m} \leq K |t-s|^\alpha\, w(s,t)^\beta \quad \forall\, s<t.
\end{equation}
Then for any $0<\gamma<\alpha-1/m$ there exists a constant $C=C(\alpha,\gamma,m)>0$ such that
\begin{equation}\label{eq:kolmogorov-conclusion-2}
\E\bigg[ \bigg(\sup_{s<t} \frac{\| X_{s,t}\|_E}{|t-s|^{\gamma} \, w(s,t)^{\beta}}\bigg)^m \bigg \vert \mathcal F\bigg]^{1/m} \leq C \, K.
\end{equation}
\end{lemma}

Let us also mention that, although for simplicity we assumed in Lemmas \ref{lem:kolmogorov} and \ref{lem:kolmogorov-2} to work with a norm $\|\cdot\|_E$, it suffices for it to be a seminorm instead.

Next, we need some basic lemmas in order to control the space-time regularity of random vector fields $A:[0,1]\times \R^d\to \R^m$.
We start by considering the time independent case.
%
%
\begin{lemma}\label{lem:kolmogorov-3}
Let $F:\R^d\to\R^n$ be a continuous field and suppose there exist $\alpha\in (0,1]$, $m\in (1,\infty)$, a constant $K>0$ and a $\sigma$-algebra $\mathcal{F}$ such that
\begin{equation}\label{eq:kolmogorov-hypothesis-3}
\| F(x)-F(y)\|_{L^m\vert \mathcal{F}} \leq K |x-y|^\alpha \quad \forall\, x,y\in\R^d.
\end{equation}
Then for any choice of parameters $\lambda,\eta\in (0,1]$ such that $\eta<\alpha-d/m$, $\lambda > \alpha-\eta$ there exists a constant $C=C(\alpha,m,d,n,\eta,\lambda)$ such that
\begin{equation}\label{eq:kolmogorov-conclusion-3}
\big\|\, \llbracket F\rrbracket_{C^{\eta,\lambda}} \big\|_{L^m\vert \mathcal{F}} \leq C \, K.
\end{equation}
\end{lemma}

\begin{proof}
By arguing componentwise, we can restrict to $n=1$; by homogeneity, we can assume $K=1$. Recall that by the classical Garsia-Rodemich-Rumsay lemma, there exists a constant $c = c(d,\eta,\alpha,m)$ such that, for any deterministic continuous function $f$ and any $R>0$, it holds
\begin{equation*}
\llbracket f\rrbracket_{C^\eta(B_R)}^m \leq c \int_{B_R\times B_R} \frac{|f(x)-f(y)|^m}{|x-y|^{2d+\eta m}} \dd x \dd y;
\end{equation*}
thus taking conditional expectation and applying Fubini, we find
\begin{equation*}
\E\big[ \llbracket F\rrbracket_{C^\eta(B_R)}^m \big|\mathcal{F}\big] \lesssim R^{(\alpha-\eta)m} \quad \forall\, R\geq 1.
\end{equation*}
Finally observe that
\begin{align*}
\E\big[ \llbracket F\rrbracket_{C^{\eta,\lambda}}^m|\mathcal{F}\big]
& \leq \sum_{R=2^j,\,j\in\N} \E\big[ R^{-\lambda m} \llbracket F \rrbracket_{C^\eta(B_R)}^m |\cF\big] \lesssim \sum_{j\in\N} 2^{-j m(\eta+\lambda-\alpha)}
\end{align*}
with the last quantity being finite under our assumptions.
\end{proof}

A combination of Lemmas \ref{lem:kolmogorov-2} and \ref{lem:kolmogorov-3} immediately yields the following.

\begin{corollary}\label{cor:kolmogorov-4}
Let $G:[0,1]\times \R^d\to\R^n$ be a continuous random vector field  and assume there exist parameters $\alpha,\beta_1,\beta_2\in (0,1]$, $m\in (1,\infty)$, a control $w$, a constant $K>0$ and a $\sigma$-algebra $\mathcal{F}$ such that
\begin{equation}\label{eq:kolmogorov-hypothesis-4}
\| G_{s,t}(x)-G_{s,t}(y)\|_{L^m\vert \mathcal{F}} \leq K |x-y|^\alpha |t-s|^{\beta_1} w(s,t)^{\beta_2}\quad \forall\, x,y\in\R^d,\, s<t.
\end{equation}
Then for any choice of parameters
\[
\gamma <\beta_1-\frac{1}{m},\quad \eta<\alpha-\frac{d}{m},\quad \lambda>\alpha-\eta
\]
there exists $C>0$, depending on all the previous parameters except $K$, such that
\begin{equation}\label{eq:kolmogorov-conclusion-4}
\bigg\| \sup_{0\leq s<t\leq 1} \frac{\llbracket G_{s,t}\rrbracket_{C^{\eta,\lambda}_x}}{|t-s|^\gamma w(s,t)^{\beta_2}}\bigg\|_{L^m|\cF}\leq C K.
\end{equation}
\end{corollary}

\section{Some a priori estimates for Young equations}\label{app:Young}
In this appendix we prove some basic bounds on (linear and nonlinear) Young differential equations, which are used several times in the article.
Such estimates are folklore, but since we did not find an appropriate version in the literature, we provide short proofs.

\begin{lemma}\label{lem:nY-apriori}
Let $A\in C_t^{p-\var} C^{\eta}_x$ with $\eta\in(0,1)$, $p\in[1,2)$ satisfying $(1+\eta)/p>1$; set $w_A(s,t):=\llbracket A\rrbracket_{p-\var,C^{\eta}_x;[s,t]}^p$.
Let $y$ be any solution to the nonlinear Young equation
\begin{equ}
y_t=y_0+\int_0^t A_{\dd s}(y_s)
\end{equ}
on $[0,1]$; then one has the bounds
\begin{equation}\label{eq:priori-estim-young}
|y_{s,t}| \lesssim w_A(s,t)^{\frac{1}{p}} + w_A(s,t), \qquad |y_{s,t}-A_{s,t}(y_s)| \lesssim w_A(s,t)^{\frac{1+\eta}{p}} + w_A(s,t)^{\frac{1}{p}+\eta}
\end{equation}
valid for all $(s,t)\in[0,1]_\leq^2$, where the hidden constants only depend on $(\eta,p)$. Similar bounds also hold for solutions only defined on an interval $[S,T]\subset [0,1]$.
\end{lemma}

\begin{proof}
By definition, $y$ must be of finite $q$-variation for some $q$ satisfying $1/p +\eta/q>1$; applying \eqref{eq:nY-remainder} with $x=y$ one finds
\begin{equation*}
|y_{s,t}|
\leq |A_{s,t}(y_s)| + |y_{s,t} - A_{s,t}(y_s)|
\lesssim w_A(s,t)^{\frac{1}{p}} \big(1 + \llbracket y \rrbracket_{q-\var;[s,t]}^\eta\big),
\end{equation*}
which in particular shows that $y$ is of finite $p$-variation. Then going through the same computation with $q=p$ and applying \cite[Proposition 5.10-(i)]{FVBook}, there exists a constant $C$ such that, for any $s\leq t$, it holds
\begin{align*}
\llbracket y\rrbracket_{p-\var;[s,t]}
\leq C w_A(s,t)^{\frac{1}{p}} \big(1 + \llbracket y \rrbracket_{p-\var;[s,t]}^\eta\big)
\leq \tilde C w_A(s,t)^{\frac{1}{p}} + \frac{1}{2} w_A(s,t)^{\frac{1}{p}} \llbracket y\rrbracket_{p-\var;[s,t]},
\end{align*}
where in the second step we used the fact that $\eta\in (0,1)$ and Young's inequality. 
This readily implies a local bound of the form
\begin{align*}
\llbracket y\rrbracket_{p-\var;[s,t]} \lesssim w_A(s,t)^{\frac{1}{p}}\quad \text{ for all $s<t$ such that } w_A(s,t)\leq 1.
\end{align*}
We can then apply \cite[Proposition 5.10-(ii)]{FVBook} to deduce that, for all $(s,t)\in [0,1]_{\leq}^2$, 
\begin{equation}\label{eq:priori-estim-young-proof}
\llbracket y\rrbracket_{p-\var;[s,t]} \lesssim w_A(s,t)^{\frac{1}{p}} + w_A(s,t).
\end{equation}
The first inequality in \eqref{eq:priori-estim-young} immediately follows from \eqref{eq:priori-estim-young-proof}, the second one from a combination of \eqref{eq:priori-estim-young-proof} with \eqref{eq:nY-remainder} for $x=y$.
\end{proof}

In the next statement instead we pass to consider more standard affine Young equations. In particular $t\mapsto A_t$ is an $\R^{d\times d}$-valued map of finite $p$-variation and the notation $\int_0^t \dd A_s\, x_s$ denotes a usual Young integral, equivalently the (deterministic) sewing of the germ $\Sigma_{s,t}:= A_{s,t} x_s$.

\begin{lemma}\label{lem:young-estimate}
Let $x$ be a solution to the affine Young equation
\[
\dd x_t = \dd A_t\, x_t + \dd z_t, \quad x\vert_{t=0}=x_0,
\]
where $A\in C^{p-\var}_t \R^{d\times d}$ and $z\in C^{\tilde{p}-\var}_t$, for some $p\in [1,2)$ and $\tilde{p}\geq p$ such that $1/p+1/\tilde{p}>1$; assume $z_0=0$. Then there exists a constant $C=C(p,\tilde{p})>0$ such that
\begin{equation}\label{eq:young-estimate-1}
\sup_{t\in [0,1]} |x_t| + \llbracket x \rrbracket_{\tilde{p}-\var} \leq C e^{C \llbracket A \rrbracket_{p-\var}^{p}} \big(|x_0|+\llbracket z \rrbracket_{\tilde p-\var}\big).
\end{equation}
When $z=0$, setting $w(s,t):=\llbracket A\rrbracket_{p-\var;[s,t]}^p$, it holds
\begin{equation}\label{eq:young-estimate-2}
|x_{s,t}| \leq C\, w(s,t)^{1/p}\, e^{C \llbracket A \rrbracket_{p-\var}^{p}} |x_0|\quad \forall \, s\leq t.
\end{equation}
\end{lemma}

\begin{proof}
%

Let us first apply the change of variable $\theta=x-z$, so that $\theta$ solves
\[ \dd \theta_t = \dd A_t\, \theta_t + \dd A_t\, z_t = \dd A_t\, \theta_t + \dd \tilde{z}_t \]
where $\tilde{z}_t:=\int_0^t \dd A_s\, z_s$.
The advantage of this maneuver is that $\tilde{z}$ is also of finite $p$-variation and controlled by (a multiple of) $w^{1/p}$.
Indeed, by Young integration it holds
\begin{equ}\label{eq:young-tilde-z}
|\tilde{z}_{s,t}|
 \lesssim |A_{s,t} z_s| + w(s,t)^{1/p} \llbracket z \rrbracket_{\tilde p-\var;[s,t]}\lesssim w(s,t)^{1/p}\llbracket z \rrbracket_{\tilde p-\var}.
\end{equ}
%
For any $s<t$, define
\[
\llbracket \theta \rrbracket_{w;[s,t]} := \sup_{s\leq r <u\leq t} \frac{|\theta_{r,u}|}{w(r,u)^{1/p}},
\]
and similarly for $\tilde{z}$. Manipulating the equation for $\theta$ in a standard manner, one finds a constant $C>0$ such that, for any $s<t$, it holds
\begin{equation}\label{eq:gronwall-young-proof}
\llbracket \theta \rrbracket_{w;[s,t]} \leq |\theta_s| + C w(s,t)^{1/p}\llbracket \theta \rrbracket_{w;[s,t]} + \llbracket \tilde{z} \rrbracket_{w;[s,t]}.
\end{equation}
If $Cw(0,1)^{1/p}\leq 1/2$ then the \eqref{eq:gronwall-young-proof} buckles with $s=0,t=1$. 
Otherwise, define recursively an increasing sequence $t_i$ by $t_0=0$ and $C w(t_i,t_{i+1})^{1/p}\in (1/3,1/2)$ and $t_n=1$ for some $n$. set $J_i:=\sup_{r\in [t_i,t_{i+1}]} |\theta_r|$ with the convention $J_{-1}=|x_0|$. Then thanks to our choice of $t_i$ and equation \eqref{eq:gronwall-young-proof}, it holds
\begin{align*}
J_i
& \leq |\theta_{t_i}| + w(t_i,t_{i+1})^{1/p} \llbracket \theta \rrbracket_{w;[t_i,t_{i+1}]}\\
& \leq (1+2 w(t_i,t_{i+1})^{1/p} )|\theta_{t_i}| + 2 w(t_i,t_{i+1})^{1/p} \llbracket \tilde{z} \rrbracket_{w;[t_i,t_{i+1}]}\\
& \leq \big(1+ C^{-1} \big) J_{i-1} + C^{-1}  \llbracket \tilde{z} \rrbracket_{w;[0,1]}
\end{align*}
Recursively this implies
\begin{align*}
\sup_{t\in [0,1]} |\theta_t|
= \sup_i J_i
\leq \big( 1+C^{-1}\big)^n (J_0 + \llbracket \tilde{z} \rrbracket_{w;[0,1]})
\leq e^{\frac{n}{C}} \big(|x_0| + \llbracket \tilde{z} \rrbracket_{w;[0,1]}\big).
\end{align*}
Finally observe that, by superadditivity of $w$ and our choice of $t_i$, it holds
\[
n = (3C)^p \sum_i w(t_i,t_{i+1}) \leq (3C)^p w(0,1),
\]
and therefore by \eqref{eq:young-tilde-z}
\begin{equ}
\sup_{t\in [0,1]} |\theta_t|\leq e^{C' \llbracket A \rrbracket_{p-\var}^p} \big(|x_0| + \llbracket z\rrbracket_{\tilde p-\var} \big)
\end{equ}
with some other constant $C'>0$.
Substituting this bound back to \eqref{eq:gronwall-young-proof}, we similarly get
\begin{equ}
\llbracket \theta \rrbracket_{w;[0,1]}\leq e^{C' \llbracket A \rrbracket_{p-\var}^p} \big(|x_0| + \llbracket z\rrbracket_{\tilde p-\var} \big).
\end{equ}
Combining everything yields the claimed bounds \eqref{eq:young-estimate-1}-\eqref{eq:young-estimate-2}.
\end{proof}

\section{Fractional regularity and Girsanov's transform}\label{app:girsanov}

We collect in this appendix several definitions of fractional regularity and show how, in certain regularity regimes, they can be combined with our results, so to verify the applicability of Girsanov's transform to the singular SDEs in consideration.

We start by recalling several classical definitions of fractional spaces for paths $f:[0,1]\to E$, $E$ being a Banach space.
For $\beta\in (0,1)$ and $p\in [1,\infty)$, the fractional Sobolev space $W^{\beta,p}=W^{\beta,p}(0,1;E)$ is defined as the set of $f\in L^p(0,1;E)$ such that
\begin{equation*}
\|f\|_{W^{\beta,p}} := \| f\|_{L^p} + \llbracket f \rrbracket_{W^{\beta,p}}<\infty, \quad
\llbracket f \rrbracket_{W^{\beta,p}}:=\Big( \int_{[0,1]^2} \frac{\| f_{s,t}\|_E^p}{|t-s|^{\beta p +1}}\, \dd s \dd t\Big)^{\frac{1}{p}}.
\end{equation*}
Similarly, we define the spaces the Besov--Nikolskii spaces $N^{\beta,p}=N^{\beta,p}(0,1;E)$ as the collections of all $f\in L^p(0,1;E)$ such that
\begin{equation*}
\|f\|_{N^{\beta,p}} := \| f\|_{L^p} + \llbracket f \rrbracket_{N^{\beta,p}}<\infty, \quad
\llbracket f \rrbracket_{N^{\beta,p}}:=\sup_{h\in (0,1)} |h|^{-\beta} \Big( \int_0^{1-h} \| f_{s,s+h}\|_E^p \, \dd s\Big)^{\frac{1}{p}}.
\end{equation*}
In the case $p=\infty$, we will set $W^{\beta,p}=N^{\beta,p}=C^\beta$.
Although we will not need it, let us mention that these spaces are particular instances of the Besov spaces $B^\beta_{p,q}$ as defined in \cite{simon1990}, indeed $W^{\beta,p}=B^\beta_{p,p}$ and $N^{\beta,p}=B^\beta_{p,\infty}$.

There is a final class of spaces we will need, which is an original contribution of this work; many processes arising from stochastic sewing can be shown to belong to this class, thanks to Lemmas \ref{lem:kolmogorov}-\ref{lem:kolmogorov-2}.
Given $\beta\in (0,1]$, $p\in  [1,\infty)$ with $\beta > 1/p$, we define the space $D^{\beta,p}=D^{\beta,p}(0,1;E)$ as the set of all $f$ for which there exists a continuous control $w=w(f)$ such that
\begin{equation}\label{eq:defn-inermediate-space}
\| f_{s,t}\|_E \leq |t-s|^{\beta-\frac{1}{p}}\, w(s,t)^{\frac{1}{p}}\quad \forall\, s<t.
\end{equation}
Observe that by superadditivity, if such a control $w$ exists, then the optimal choice must be necessarily given by
\begin{equation*}
w(s,t)=\llbracket f\rrbracket_{D^{\beta,p};[s,t]}^p:= \sup \sum_{i=1}^n \frac{\| f_{t_i,t_{i+1}}\|_E^p}{|t_{i+1}-t_i|^{\beta p-1}}
\end{equation*}
where the supremum runs over all possible finite partitions $s=t_0 < t_1<\ldots<t_n=t$ of $[s,t]$.
We can therefore endow the space $D^{\beta,p}$ with the norm
\begin{equation}\label{eq:defn-intermediate-space-norm}
\| f\|_{D^{\beta,p}} := \| f_0\|_E + \llbracket f\rrbracket_{D^{\beta,p}}, \quad
\llbracket f\rrbracket_{D^{\beta,p}} = \llbracket f\rrbracket_{D^{\beta,p};[0,1]},
\end{equation}
which makes them Banach spaces; observe the analogy with the definition of $C^{p-\var}$ and its characterization via controls. In particular, if a function $f$ is known to satisfy \eqref{eq:defn-inermediate-space}, then it must hold $\llbracket f\rrbracket_{D^{\beta,p}}\leq w(0,1)^{1/p}$.

For $\beta>1/p$, we define $W^{\beta,p}_0=\{f\in W^{\beta,p}: f_0=0\}$ (as we will shortly see, this is a good definition, as elements of $W^{\beta,p}$ are continuous functions); similarly for $N^{\beta,p}_0$ and $D^{\beta,p}_0$. 

The next proposition summarises the embeddings between these classes of spaces, as well as the Cameron--Martin spaces $\mathcal{H}^H$ and spaces of finite $q$-variation.

\begin{prop}\label{lem:embedding-intermediate-space}
Let $\beta\in (0,1]$, $p\in  [1,\infty)$ with $\beta > 1/p$; then, the following hold:
\begin{itemize}
\item[i)] for any $\eps>0$, we have
$ W^{\beta,p} \hookrightarrow D^{\beta,p} \hookrightarrow N^{\beta,p} \hookrightarrow W^{\beta-\eps,p}$;
\item[ii)] if $\bar\beta\leq\beta$ and $\beta-1/p\geq\bar\beta-1/\bar p$, then $N^{\beta,p}\hookrightarrow N^{\bar\beta,\bar p}$; in particular, $N^{\beta,p} \hookrightarrow C^{\beta-1/p}$;
\item[iii)] $N^{\beta,p} \hookrightarrow C^{1/\beta-\var} \hookrightarrow N^{\beta,1/\beta}$;
\item[iv)] let $H\in (0,1/2)$ and $E=\R^d$, then for any $\eps>0$ it holds
\begin{equation*}
W^{H+\frac{1}{2}+\eps,2}_0 \hookrightarrow \mathcal{H}^H\hookrightarrow W^{H+\frac{1}{2}-\eps,2}_0;
\end{equation*}
in particular, $\mathcal{H}^H\hookrightarrow C^{q-\var}$ for any $q>(H+1/2)^{-1}$.
\end{itemize}
\end{prop}

\begin{proof}
i) The last embedding $N^{\beta,p} \hookrightarrow W^{\beta-\eps,p}$ is classical and can be found in \cite[Corollary~23]{simon1990}.
The embedding $W^{\beta,p} \hookrightarrow D^{\beta,p}$ follows from \cite[Theorem~2]{friz2006}; in particular, by Garsia-Rodemich-Rumsay lemma, the associated control $w_f$ can be taken as
\[
w_f(s,t) = \int_{[s,t]^2} \frac{\| f_{r,u}\|_E^p}{|r-u|^{1+\beta p}} \, \dd r \dd u.
\]
It remains to show the embedding $\mathcal{D}^{\beta,p} \hookrightarrow N^{\beta,p}$; this follows the same technique used to show that $C^{p-\var}\hookrightarrow N^{1/p,p}$, see e.g. \cite[Proposition~4.3]{liu2020characterization}.
Indeed, for any $h\in [0,T]$, it holds
\begin{align*}
\| f_{h+\cdot} - f_{\cdot}\|_{L^p}^p
= \int_0^{1-h} \| f_{t,h+t}\|_E^p \dd t \leq |h|^{\beta p-1} \int_0^{1-h} w(t,h+t) \dd t,
\end{align*}
where $w(s,t)=\llbracket f\rrbracket_{D^{\beta,p};[s,t]}^p$.
Denoting by $K$ the largest integer such that $Kh\leq 1-h$, we have
\begin{align*}
\int_0^{1-h} w(t,h+t) \dd t&\leq \int_0^{Kh} w(t,h+t) +|h| w(0,1)
\\
& = \sum_{i=0}^{K-1} \int_{ih}^{(i+1)h} w(s,h+s) \dd s+|h| w(0,1)
\\
&= \int_0^h \sum_{i=0}^{K-1} w(ih+s,(i+1)h+s) \dd s+|h| w(0,1)
\\
&\leq \int_0^h w(0,1) \dd s +|h| w(0,1)= 2 |h| w(0,1)
\end{align*}
where in the last inequality we used the superadditivity of $w$. Overall we conclude that $\llbracket f\rrbracket_{N^{\beta,p}}^p\leq 2\llbracket f\rrbracket_{D^{\beta,p}}^p$.
%

ii) These embeddings can be found in e.g. \cite[Corollary~22]{simon1990}, \cite[Corollary~26]{simon1990}.

iii) These embeddings can be found in e.g. \cite[Proposition~4.1]{liu2020characterization},  \cite[Proposition~4.3]{liu2020characterization}.

iv) The second embedding $\mathcal{H}^H\hookrightarrow W^{H+\frac{1}{2}-\eps,2}_0$ is the result of \cite[Theorem~3]{friz2006}; the last one follows from it combined with $N^{q,2}\hookrightarrow C^{1/q-\var}$.
It only remains to show the first embedding. Although we believe it to be common knowledge, we haven't found a proof in the literature, thus we give a detailed one.

Given $ f\in W^{H+1/2+\eps,2}_0$, in order to verify that $f\in \mathcal{H}^H$, we need to check that $K^{-1}_H f\in L^2$, where 
\begin{equation*}
K^{-1}_H f = s^{1/2-H} D_{0+}^{1/2-H} s^{H-1/2} D^{2H}_{0+},
\end{equation*}
see eq. (12) from \cite{NO1}; $D_{0+}^\gamma$ denotes the Riemann-Liouville fractional derivative of order $\gamma$, for which again we refer to \cite{NO1}.

By using standard embeddings between $W^{\delta,2}$ spaces and potential spaces $I^+_{\delta,2}$ (cf. \cite[Proposition 5]{decreusefond2005stochastic}), up to losing an arbitrary small fraction of regularity, we know that for any $f\in W^{H+1/2+\eps,2}_0$  it holds $h:=D^{2H}_{0+} f \in W^{1/2-H+\eps/2,2}$ (this is the only point in the proof where the condition $f(0)=0$ is needed). Thus we are left with verifying that, for the choice $\gamma=1/2-H$, it holds
\begin{equation*}
(K^{-1}_H f)_t = C_{\gamma} \bigg( t^{-\gamma} h_t + \gamma t^\gamma \int_0^t \frac{t^{-\gamma} h_t - s^{-\gamma} h_s}{|t-s|^{1+\gamma}} \dd s \bigg) \in L^2(0,1;\R^d).
\end{equation*}
From now on we will drop the constants $C_\gamma$ and $\gamma$ for simplicity.

For the first term, observing that $t^{-\gamma}\in L^r$ for any $r$ such that $1/r<1/2-H$ and that $h\in W^{1/2-H+\eps/2,2}\hookrightarrow L^p$ for $1/p=H-\eps/2$, it's easy to check by H\"older's inequality that $t^{-\gamma} h_t \in L^2$.

By time rescaling and addition and subtraction, we can split the integral term respectively into
\begin{equation*}
I^1_t := \int_0^t \frac{h_t-h_s}{|t-s|^{1+\gamma}} \dd s, \quad
I^2_t := t^{-\gamma} \int_0^1 \frac{1-s^{-\gamma}}{(1-s)^{1+\gamma}} h_{t s} \dd s.
\end{equation*}
For $I^1$, it holds
\begin{align*}
\int_0^1 |I^1_t|^2 \dd t
\leq \int_0^1 \bigg( \int_0^1 \frac{|h_t-h_s|}{|t-s|^{1+\gamma}} \dd s\bigg)^2 \dd t
\lesssim \int_{[0,1]^2} \frac{|h_t-h_s|^2}{|t-s|^{1+2\gamma+\eps}} \dd s \dd t \lesssim \| h\|_{W^{\gamma+\eps/2,2}},
\end{align*}
where in the middle passage we used Jensen's inequality.
To handle $I^2$, define $F^\gamma_s := (1-s^{-\gamma})/(1-s)^{1+\gamma}$; $F^\gamma$ is only unbounded at the points $s=0$ and $s=1$, where it behaves asymptotically respectively as $-s^{-\gamma}$ and $(1-s)^{-\gamma}$, therefore $F^\gamma\in L^1\cap L^2$.
As before, $h\in L^p$ for $1/p=H-\eps/2< 1/2$, therefore by H\"older's inequality
\begin{align*}
	|I^2_t|
	\leq t^{-\gamma} \| F^\gamma\|_{L^{p'}} \| h_{t\cdot}\|_{L^p}
	\sim t^{-\gamma-\frac{1}{p}} \| F^\gamma\|_{L^{p'}} \| h\|_{L^p} 
	\sim t^{\eps/2-1/2}
\end{align*}
which readily implies $I^2\in L^2$ as well.
\end{proof}

\begin{remark}\label{rem:Girsanovapp}
By Proposition \ref{lem:embedding-intermediate-space},
for a deterministic path $g$ to belong to the Cameron-Martin space $\mathcal{H}^H$ for $H\in (0,1/2)$, it suffices to verify that $g\in \mathcal{D}^{\beta,p}$ for parameters $p\in (1,2]$ and $\beta>0$ satisfying
\begin{equ}\label{eq:girsanov1}
\beta-\frac{1}{p}>H,
\end{equ}
in which case we have the estimate $\| g\|_{\mathcal{H}^H} \lesssim \| g\|_{\mathcal{D}^{\beta,p}}$.
Therefore, if a stochastic process $h$ is adapted and belongs to $\cD^{\beta,p}$, then for a sequence of stopping times $(\tau_n)_{n\in\N}$ satisfying $\tau_n\nearrow\infty$, the laws of $B^H$ are $B^H_\cdot+h_{\cdot\wedge\tau_n}$ are mutually absolutely continuous.
If the stronger Novikov-type condition
\begin{equ}\label{eq:novikov}
\E\big[\exp\lambda\|h\|_{\cD^{\beta,p}}^2\big]<\infty\quad \forall\,\lambda>0
\end{equ}
holds, then one can infer the stronger conclusion that the laws of $B^H$ are $B^H_\cdot+h$ are equivalent and that the Radon-Nikodym derivative admits moments of any order, see \cite[Proposition 3.10]{galeati2021distribution} for a similar statement.
\end{remark}

With the above considerations in mind, we are now ready to present a result on the applicability of Girsanov's transform, which is the main motivation for this appendix.

\begin{lemma}\label{lem:girsanov}
Assume \eqref{eq:main exponent} and that 
\begin{equation}\label{eq:restriction}
1-1/(Hq')<0.
\end{equation}
Let $b\in L^q_t C^\alpha_x$, $x_0\in \R^d$, and denote by $\mu$ the law of the solution $X$ to the associated SDE \eqref{eq:SDE}. Then Girsanov's transform applies and $\mu$ is equivalent to $\mathcal{L}(x_0+B^H)$. As a consequence, ${\rm supp}\, \mu = C([0,1];\R^d)$.
\end{lemma}

\begin{proof}
Without loss of generality we may assume $\alpha<0$ and $x_0=0$.
In view of Remark \ref{rem:Girsanovapp}, we need to verify \eqref{eq:novikov} with $h=\varphi=X-B^H$ and with some $\beta$, $p$ satisfying \eqref{eq:girsanov1}.

Let $\kappa>0$ small enough so that $H$, $\alpha-\kappa$, and $q$ also satisfy \eqref{eq:main exponent}, and let $\tilde b\in L^q_tC^{\alpha-\kappa}_x$ with norm $1$.
By Lemmas \ref{lem:apriori-estim}, \ref{lem:integral-estimates}, and \ref{lem:kolmogorov} we have that with some $\mu>0$
\begin{equ}\label{eq:girsanov2}
\E\bigg[\exp\bigg( \mu\Big\|\int_0^\cdot \tilde b_r(B^H_r+\varphi_r)\dd r \Big\|_{\cD^{1+(\alpha-\kappa) H-\kappa,q}}^2\bigg) \bigg]<\infty.
\end{equ}
Note that for sufficiently small $\kappa$ the exponents satisfy \eqref{eq:girsanov1} as a consequence of \eqref{eq:main exponent}.
Therefore \eqref{eq:girsanov2} looks like \eqref{eq:novikov}, except the arbitrariness of the coefficient. One can then proceed by an interpolation argument as in \cite[Proposition~3.8]{galeati2021distribution}: for any $\kappa>0$ and $\lambda>0$ there exists $b^{-}$ and $b^{+}$ such that $b=b^-+b^+$ and
\begin{equ}
\frac{2 \lambda}{\mu} \|b^-\|_{L^q_tC^{\alpha-\kappa}_x}^2\leq1,\qquad\qquad\|b^+\|_{L^q_t C^0_x}=:K<\infty,
\end{equ}
where $K$ may depend on all parameters.
Then we can write
\begin{equs}
\E&\bigg[\exp\bigg( \lambda\Big\|\int_0^\cdot b(B^H_r+\varphi_r)\dd r \Big\|_{\cD^{1+(\alpha-\kappa) H-\kappa,q}}^2\bigg) \bigg]
\\
&\leq e^{2K^2}\E\bigg[\exp\bigg( \mu\, \frac{2\lambda}{\mu} \Big\|\int_0^\cdot b^-(B^H_r+\varphi_r)\dd r \Big\|_{\cD^{1+(\alpha-\kappa) H-\kappa,q}}^2\bigg)\bigg]<\infty,
\end{equs}
applying \eqref{eq:girsanov2} with $\sqrt{2\lambda/\mu}b^-$ in place of $\tilde b$ in the last step.
\end{proof}

\begin{remark}
The restriction \eqref{eq:restriction} in Lemma \ref{lem:girsanov} is necessary. Indeed, even taking a space-independent drift $b\in L^q$, so that $\varphi\in W^{1,q}$, the condition $1-1/q>(H+1/2)-1/2$ necessary for the Sobolev embedding implies \eqref{eq:restriction}.
The reader may feel this pathological and rightly so: for such a $b$ we can deduce everything about the law of $B^H+\varphi$ from the law of $B^H$.
Note that this also motivates the use of ``stochastic regularity'' as in e.g. \eqref{eq:drift-regularity}, which assigns to deterministic functions (like $\varphi$ in this example) infinite regularity.

Note also that \eqref{eq:restriction} enforces $H\in(0,1/2)$. We do not discuss the regime of large $H$ in detail, as Girsanov's transform becomes less end less useful as $H$ increases.
For example, for $H>2$ one has $B^H\in C^2$ and (in the nontrivial case $\alpha<1$) $\varphi\notin C^2$, yielding trivially the mutual singularity of the laws of $B^H$ and $X=B^H+\varphi$. Once again, the way out is to use ``stochastic regularity'' as a substitute for Girsanov.
\end{remark}

\bibliography{myBiblio2}{}
\bibliographystyle{plain}

\end{document}